\documentclass[smallextended,envcountsect]{svjour3}
\usepackage{dsfont}
\usepackage{epsfig,color,graphicx,subfigure,multirow,lscape}
\usepackage[colorlinks,linkcolor=blue,citecolor=blue]{hyperref}
\usepackage{tikz,cite}
\usepackage{graphicx,dsfont,booktabs,threeparttable}
\usepackage{amsmath}
\usepackage{empheq}
\usepackage{amssymb}
\usepackage{appendix}
\usetikzlibrary{arrows, automata}
\smartqed

\usepackage{mathtools}
\usepackage[normalem]{ulem}
\usepackage{color}
\usepackage{enumerate}
\usepackage{geometry}

\journalname{}

\newtheorem{lem}{Lemma}[section]
\newtheorem{cor}{Corollary}[section]
\newtheorem{prop}{Proposition}[section]
\newtheorem{rem}{Remark}[section]

\newtheorem{algo}{Algorithm}[section]

\newtheorem{asmp}{Assumption}[section]

\def\be{\begin{eqnarray}}
\def\ee{\end{eqnarray}}
\def\ben{\begin{eqnarray*}}
\def\een{\end{eqnarray*}}
\def\ba{\begin{array}}
\def\ea{\end{array}}
\def\bi{\begin{itemize}}
\def\ei{\end{itemize}}

\newcommand{\bz}{\mathbf z}

\newcommand{\setto}{\rightrightarrows}

\def\cB{{\mathcal B}}

\def\cD{{\mathcal D}}

\def\cG{{\mathcal G}}
\def\cH{{\mathcal H}}

\def\cL{{\mathcal L}}

\def\cP{{\mathcal P}}

\def\bR{{\mathbb R}}

 \DeclareMathOperator*{\diag}{Diag}

\DeclareMathOperator{\Fix}{Fix}
\DeclareMathOperator{\Diag}{Diag}

\DeclareMathOperator{\gra}{gra}

\DeclareMathOperator{\zer}{zer}
\def\endproof{\hfill $\Box$ \vskip .5cm}
\def\prox{{\rm Prox}}

\def\null{{\rm null}}

\def\[{\begin{equation}}
\def\]{\end{equation}}

\newcommand{\D}{\Delta}

\newcommand{\R}{\mathbb R}
\newcommand{\N}{\mathbb N}

\begin{document}

\title{A primal-dual splitting algorithm with convex combination and larger step sizes for  composite monotone inclusion problems}
\titlerunning{A PDSA with convex combination and larger step sizes for CMIPs}
\author{Xiaokai Chang$^{1}$ \and Junfeng Yang$^2$ \and  Jianchao Bai$^3$ \and Jianxiong Cao$^1$}

\institute{1 \ School of Science, Lanzhou University of Technology, Lanzhou, Gansu,  P. R. China (xkchang@lut.edu.cn, caojianxiong2007@126.com).
X. Chang was supported by the National Natural Science Foundation of China (12161053) and the Natural Science Foundation  for Distinguished Young Scholars of Gansu Province (22JR5RA223).
J. Cao was supported by the National Natural Science Foundation of China (12261058) and the National Defense Basic Scientific Research (JCKY2022427C001). \\
2 \ Corresponding author. School of Mathematics, Nanjing University, Nanjing, P. R. China (jfyang@nju.edu.cn). This author's work was supported by the National Natural Science Foundation of China (12431011, 12371301). \\
3 \ School of Mathematics and Statistics, Northwestern Polytechnical University, {Xi'an}, P. R. China (bjc1987@163.com). This author's work was supported by the National Natural Science Foundation of China (12471298) and the  Shaanxi Fundamental Science Research Project for Mathematics and Physics (23JSQ031).}

\date{Received: date / Accepted: date}

\maketitle

\begin{abstract}
The primal-dual splitting algorithm (PDSA) by Chambolle and Pock is efficient for solving structured convex optimization problems. It adopts an extrapolation step and achieves convergence under certain step size condition. Chang and Yang recently proposed a modified PDSA for bilinear saddle point problems, integrating a convex combination step to enable convergence with extended step sizes. In this paper, we focus on composite monotone inclusion problems (CMIPs), a generalization of convex optimization problems. While V\~{u} extended PDSA to CMIPs, whether the modified PDSA can be directly adapted to CMIPs remains an open question. This paper introduces a new PDSA for CMIPs, featuring the inclusion of both an extrapolation step and a convex combination step. The proposed algorithm is reformulated as a fixed-point iteration by leveraging an extended firmly nonexpansive operator. Under a significantly relaxed step size condition, both its convergence and sublinear convergence rate results are rigorously established. For structured convex optimization problem, we establish its sublinear convergence rate results measured by function value gap and constraint violations. Moreover, we show through a concrete example that our condition on the involved parameters cannot be relaxed. Numerical experiments on image denoising, inpainting, matrix games, and LASSO problems are conducted to compare the proposed algorithm with state-of-the-art counterparts, demonstrating the efficiency of the proposed algorithm.
\end{abstract}

\keywords{Primal-dual splitting algorithm \and composite monotone inclusion problem \and fixed-point iteration \and extrapolation \and convex combination}
\subclass{49M29 \and  65K10 \and 65Y20 \and 90C25 }

\section{Introduction}
\label{sec_introduction}
Let $A: \cH\setto\cH$ and $B: \cG\setto\cG$ be maximal monotone set-valued operators defined on
real Hilbert spaces $\cH$ and $\cG$ respectively, and let $K:\cH\rightarrow \cG$ be a nonzero bounded linear operator with its adjoint denoted by $K^*$.
In this paper, we focus on the composite monotone inclusion problem (CMIP) presented in the following form:
\be\label{inclusion}
\mbox{Find}~~ x \in \cH~~ \mbox{such that}~~ 0 \in A(x) + K^* B(K x).
\ee
The Attouch-Th\'{e}ra dual problem of \eqref{inclusion} as presented in \cite{Attouch96dual} is formulated as follows:
\be\label{inclusion-dual}
\mbox{Find}~~ y \in \cG ~~ \mbox{such that}~~ 0 \in -KA^{-1}(-K^*y) + B^{-1}(y).
\ee
Here, $A^{-1}$ stands for the inverse operator of $A$. Problem \eqref{inclusion} occurs naturally within the context of partial differential equations (PDEs), which are often derived from mechanical problems, as well as in differential inclusions, game theory, and a variety of other disciplines, see, e.g., \cite{SPLIT21Arias} and the references therein.
Among others, problem \eqref{inclusion} encompasses, as a particular instance,  the following composite convex optimization problem:
\be\label{two-block}
\min_{x\in\cH} ~ g(x)+f(K x).
\ee
Here, the functions $g :\cH\rightarrow\bR\cup \{+\infty\}$ and $f :\cG\rightarrow\bR\cup \{+\infty\}$ are extended-real-valued, closed, proper, and convex.
In fact, under certain regularity conditions, solving the composite convex optimization problem \eqref{two-block} is tantamount to solving its optimality condition, which can be cast as the inclusion problem (\ref{inclusion}) with $A=\partial g$ and $B=\partial f$. Here, $\partial$ denotes the subdifferential operator.
Let $f^*(y) = \sup_{u\in\cG}\{ \langle y, u\rangle - f(u)\}$, for $y\in \cG$, be the Legendre-Fenchel conjugate of $f$.  Then,  problem \eqref{two-block} can be reformulated as the following bilinear saddle point problem:
\be\label{saddle-point}
\min_{x\in \cH}\max_{y\in \cG} ~g(x)+\langle K x,y\rangle -f^*(y).
\ee
The dual problem of \eqref{two-block} is formulated as
$\max_{y\in\cG} -f^*(y) - g^*(-K^*y)$, which, under certain regularity condition, reduces to the following inclusion problem:
\be \label{dual-two-block-0form}
\mbox{Find}~~ y \in \cG ~~ \mbox{such that}~~
0\in - K \partial g^*(-K^*y) + \partial f^*(y).
\ee
Given the well known identities $\partial f^* = (\partial f)^{-1}$ and
$\partial g^* = (\partial g)^{-1}$, it can be readily observed that problem \eqref{dual-two-block-0form} represents a special instance of \eqref{inclusion-dual}.
Problems \eqref{two-block}-\eqref{dual-two-block-0form} have numerous applications in signal and image processing, traffic theory, optimal transport,  and so on, see, e.g., \cite{CP2016imaging,Chambolle2011A,Esser2010General,He2012Convergence} and the references therein.

\subsection{Related work}
The Douglas-Rachford splitting method (DRSM, \cite{Lions1979Splitting,Douglas-Rachford56}) was initially developed as a numerical approach for solving linear systems that arise during the discretization of PDEs.  It serves as a benchmark method for solving the inclusion problem \eqref{inclusion} when $K$ is the identity operator $I$. Under such circumstance, DRSM constructs a sequence $\{x_n\}_{n\in\N}$ through the recurrence relation
\ben
x_{n+1} = x_n - J_{\tau A} (x_n)+J_{\tau B} (2J_{\tau A} (x_n)-x_n).
\een
Here, $J_{\tau A}=(I+\tau A)^{-1}$ denotes the resolvent operator of $A$  and $\tau>0$ is a parameter.
The weak convergence of the shadow sequence $\{J_{\tau A} (x_n)\}_{n\in\N}$   was established in \cite{Svaiter11DR}.
In the case where $K \neq I$,   the maximal monotonicity of $K^\ast BK$ is required to ensure the weak convergence of the DRSM.
Additionally, computing the resolvent operator of $K^\ast BK$ is typically far more challenging than computing that of $B$.

Since problem \eqref{inclusion} with $K \neq I$ encompasses numerous {challenging} problems in optimization and real-world applications, it is highly desirable to devise full-splitting algorithms that simultaneously solve both \eqref{inclusion} and \eqref{inclusion-dual}.
An algorithm is referred to as full-splitting if it only requires to evaluate the resolvent operators of $A$ and $B$, forward computations of $K$ and its adjoint operator $K^*$, scalar multiplications and vector additions.
Many algorithms satisfying these conditions have been derived in the literature, see, e.g., \cite{splitting13Vu,Bot2013DR,improvement15Bot,Combettes11SKEW} and the references therein.
A well known primal-dual splitting algorithm (PDSA) is the one proposed by Chambolle and Pock in \cite{Chambolle2011A} for solving the composite convex optimization problem \eqref{two-block}, which is also commonly referred to as primal-dual hybrid gradient (PDHG) algorithm in the literature.
Later, the PDSA was extended by V\~{u} in \cite{splitting13Vu} to solve CMIPs \eqref{inclusion}, which adopts an extrapolation step
and generates a sequence $\{(x_n, y_n)\}_{n\in\N}$   via the recurrence
\begin{subequations}\label{PDSA}
\begin{align}
\bar{x}_{n + 1} &= J_{\tau A}(x_n - \tau K^*y_n), \label{PDSA-1}\\
z_{n + 1} &= \bar{x}_{n + 1} + (\bar{x}_{n + 1}-x_n), \label{PDSA-2} \\
\bar{y}_{n + 1} &= J_{\sigma B^{-1}}(y_n + \sigma K z_{n + 1}),\label{PDSA-3} \\
(x_{n+1},y_{n+1})&=(x_n,y_n)+\rho(\bar{x}_{n + 1}-x_n,  \bar{y}_{n + 1}-y_n).
\end{align}
\end{subequations}
Here, $\tau, \sigma >0$ are step sizes and $\rho\in(0,2)$ is a relaxation parameter.
Clearly, at each iteration, the PDSA algorithm given by \eqref{PDSA} requires evaluating the resolvent operators $J_{\tau A}$ and $J_{\sigma B^{-1}}$, the linear operator $K$ and its adjoint $K^*$ just once.
The convergence of the PDSA \eqref{PDSA} has been proven under the condition that the step sizes $\tau$ and $\sigma$ satisfy the inequality
$\tau \sigma \|K\|^2 \leq1$, where $\|K\|$ denotes the spectral norm of $K$.
In fact, the extrapolation step \eqref{PDSA-2} is crucial for the convergence of PDSA, see \cite{He22On}
for {divergent examples} without this extrapolation step.
In this paper, we refer to the iterative framework defined in \eqref{PDSA} as Chambolle-Pock's primal-dual hybrid gradient (CP-PDHG) algorithm for brevity.
Recently, Bo\c{t} \textit{et al.} introduced a PDSA with minimal lifting in \cite{PD2023Bot} to deal with multi-block monotone inclusion problems.

The technique of incorporating a convex combination step within an iterative algorithm was adopted by Malitsky in \cite{Malitsky2019Golden}, in which a golden ratio algorithm (GRA) was proposed for solving monotone variational inequality problems. This algorithm employs a convex combination step of the form:
\be
v_{n+1} = \frac{\psi-1}{\psi} x_{n} + \frac{1}{\psi}v_{n} \quad (1 < \psi \leq (1+\sqrt{5})/2).  \label{convex_com}
\ee
Recently, Chang and Yang, in \cite{ChY2020Golden,ChYZ2022GRPDAL}, introduced a golden ratio primal-dual algorithm (GRPDA) to solve the bilinear saddle point problem \eqref{saddle-point} and established its convergence results under the relaxed step size condition $\tau\sigma\|K\|^2 < \psi$.
In \cite{ChY2022relaxed}, the allowable range of $\psi$ was  further broadened, leading to more flexible parameter selection.
Below, we denote the algorithm in \cite{ChY2022relaxed} that uses $\psi$ from its newly expanded region as PDAc, given that it is a primal-dual type algorithm featuring a convex combination step.
Precisely because of the presence of the convex combination step, it remains uncertain how to derive the convergence results for GRA, GRPDA and PDAc from the  perspective of fixed-point iteration.

\subsection{Motivation and contributions}
The extrapolation step \eqref{PDSA-2} and the convex combination step \eqref{convex_com} prove to be highly effective when formulating PDSAs.
Specifically, these steps play a pivotal role in attaining convergence results and enabling the use of larger step sizes, as illustrated in \cite{Chambolle2011A,ChY2020Golden,ChYZ2022GRPDAL} for solving composite convex optimization problems.
However, when addressing the CMIP presented in \eqref{inclusion}, there remains uncertainty regarding the feasibility of incorporating a convex combination step in designing PDSAs. Specifically, it is unclear whether such a step can enable the use of larger step sizes while guaranteeing convergence from the fixed-point iteration perspective.
Inspired by this question and the promising performance of PDSAs with convex combination steps, in this paper, we put forward a new PDSA for solving the inclusion problem \eqref{inclusion}.
To be more specific, the algorithm we propose incorporates both extrapolation and convex combination steps and allows for more flexible step sizes.
The key idea of our proof lies in rephrasing the proposed algorithm as a fixed-point iteration of an extended firmly nonexpansive operator.
Our contributions are summarized as follows.

\bi
\item[(i)] We introduce a basic full-splitting iterative scheme, which encompasses well known splitting methods as particular instances, such as PDSA \cite{splitting13Vu}, CP-PDHG \cite{Chambolle2011A}, and PDAc \cite{ChY2022relaxed}. By making a specific selection of the extrapolation and convex combination steps within the proposed framework, we present a novel PDSA, whose convergence and sublinear convergence rate results are established from the perspective of fixed-point iteration.

\item[(ii)] Another contribution of this work is that the step sizes involved can be larger than those in  \cite{ChY2020Golden,ChYZ2022GRPDAL,ChY2022relaxed}. To be more precise, the allowable range of $\tau\sigma\|K\|^2$ can extend up to $(0, 4)$ when the relaxation and convex combination parameters are selected appropriately. According to our numerical findings, the algorithm consistently benefits from moderate over-relaxation when combined with a relatively small convex combination parameter. This strategic balance yields considerable improvements in numerical performance.
    Besides, after leveraging the monotonicity of $A$ and $B$, the reasoning process involves exclusively identical reformulations without any relaxations. Consequently, the derived upper bound for the self-adjoint positive definite operator $P$, which is crucial for ensuring the convergence of the fixed-point iteration of the operator $T_P$, is sharp, see Remark \ref{rem-tight} (i). A concrete example is also given in Section \ref{sec:he} to show the sharpness of the condition on $P$.

\item[(iii)]  We propose an empirical approach for the dynamic adjustment of the relaxation and convex combination parameters to enhance the performance of the proposed algorithm. This heuristic proves to be highly effective for the class of bilinear saddle point problems.
\ei

\subsection{Organization}
The rest of this paper is organized as follows. In Section \ref{sec:Preliminarries}, we present some useful facts and notation.
Then we introduce the basic scheme of our PDSA and show the connection with some classical splitting algorithms in Section \ref{sec:Resolvent splitting}.
A new PDSA with both extrapolation and convex combination steps is also derived in this section by taking a specific choice in the basic scheme.
Section \ref{sec:Convergence analysis} is devoted to the convergence analysis of the new PDSA from the perspective of fixed-point iteration, and $o(1/N)$ ergodic convergence rate is obtained. Moreover, we establish $\mathcal{O}(1/N)$ ergodic sublinear convergence rate results for the composite convex optimization problem \eqref{two-block}, measured by the function value gap and constraint violations.
In Section \ref{sec:discussion}, some further issues are discussed for the proposed algorithm.
In Section \ref{sec-experiments}, we  illustrate the efficacy of the proposed algorithm by benchmarking it against several state-of-the-art algorithms
in  solving image denoising, inpainting, matrix-game, and LASSO problems.
Finally,  concluding remarks are provided in Section \ref {sec:conclusion}, along with discussions on some future research directions.

\section{Preliminaries}
\label{sec:Preliminarries}
Denote the inner-product  and the induced norm on a real Hilbert space by $\langle\cdot,\cdot\rangle$ and $\|\cdot\|$, respectively. A set-valued operator $A\colon\cH\setto\cH$ is said to be monotone if
 $$ \langle x-y,u-v\rangle\geq 0\qquad\forall (x,u),(y,v)\in\gra A:=\bigr\{(x,u)\in \cH\times\cH~|~u\in A(x)\bigl\}. $$
A monotone operator is maximally monotone if no proper extension of it is monotone.
The inverse of $A$, denoted by $A^{-1}$, is defined through its graph $\gra A^{-1} =\{ (u,x) \in \cH\times\cH ~|~(x,u) \in \gra A\}$.
As used before, the identity operator is denoted by $I$. Whenever necessary, the underlying space of $I$ will be denoted by a subscript.
The resolvent operator of $A\colon\cH\setto\cH$ is defined by $J_A:=(I+A)^{-1}$.
When $A$ is maximally monotone, its resolvent operator $J_A$ is single-valued, everywhere defined, and firmly nonexpansive \cite{Bauschke2011book}.
For any $\tau>0$, we have the following Moreau decomposition
\be\label{Moreau}
J_{\tau A} + \tau J_{\tau^{-1} A^{-1}}\circ \tau^{-1} I = I.
\ee
In particular, $J_{A^{-1}} = I-J_A$.
The set of zeros of $A$ is given by $\zer A = \{x\in\cH~|~0\in Ax\}$.
The set of fixed points of a single-valued operator $T : \cH \rightarrow \cH $ is denoted by $\Fix T :=\{x\in\cH ~|~ Tx = x\}$.

Let $K:\cH\rightarrow \cG$ be a bounded linear operator. Throughout this paper, we denote the operator norm of $K$ by $L$,   i.e., $L := \|K\| = \|K^*\| = \sup \{\langle K x, y\rangle~|~ \|x\|=\|y\|=1,\, x\in \cH,\, y\in \cG\}$, and use the two notation interchangeably.
The set of all bounded linear operators defined on $\cH$ is denoted by $\cB(\cH)$.
Let $M\in \cB(\cH)$ be self-adjoint and positive definite (denoted as $M\succ 0$), i.e., $\langle x,My\rangle=\langle Mx,y\rangle$ for any $x,y\in \cH$
and $\langle x,Mx\rangle >0$ for any nonzero $x \in \cH$. Then, we let
 $\langle x,y\rangle_M := \langle x, My\rangle$ for any $x,y\in\cH$ and
$\|x\|_M := \sqrt{\langle x, x\rangle_M}$.
When the operator $M$ is self-adjoint yet not positive definite, we continue to employ the notation $\|y\|_{M}^2 := \langle y, My\rangle$. In such a scenario, the value of $\|y\|_{M}^2$ is not guaranteed to be nonnegative.
For any $x, y, z\in \cH$, there hold
\be
2\langle x-y, x-z\rangle_M&=&  \|x-y\|_M^2 +  \|x-z\|_M^2  -\|y-z\|_M^2,\label{id}\\
\|M x+(I-M)y\|^2&=&  \|x\|_M^2+\|y\|_{I-M}^2-\|x-y\|_{M(I-M)}^2.\label{id2}
\ee

Denote the set of solutions of the inclusion problem \eqref{inclusion} by $\cP$  and that of its Attouch-Th\'{e}ra dual problem \eqref{inclusion-dual} by $\cD$.
Denote the set of primal-dual solutions of \eqref{inclusion}-\eqref{inclusion-dual} by
\be\label{cond-solver}
\Omega:=\{(x^\star,y^\star)\in\cH\times \cG~|~-K^*y^\star\in A(x^\star)~\mbox{and}~y^\star\in B(Kx^\star)\}.
\ee
Indeed,  it is well known from \cite{Bauschke2011book} that $\Omega$ is a subset of $\cP\times\cD$, and  $\cP\neq\emptyset$ if and only if $\cD\neq\emptyset$.
Throughout the paper, we make the following blanket assumptions.
\begin{asmp}\label{asmp-0}
The resolvent operators of $A$ and $B$  are easy to evaluate.
\end{asmp}
\begin{asmp}\label{asmp-1}
The set $\Omega$ defined in \eqref{cond-solver} is nonempty. Hence, both $\cP$ and $\cD$ are nonempty.
\end{asmp}

Below, we introduce a notion called extended averaged operator. Later, we will show that the proposed algorithm can be expressed as a fixed-point iteration of a certain operator $T_P$ (see Lemma \ref{lem-a}); under specific circumstances, $T_P$ is $P$-averaged (see Remark \ref{rem-P} (iii)).
\begin{definition}[Extended averaged operator]
Let $M\in \cB(\cH)$ be self-adjoint and positive definite and $D$ be a nonempty subset of $\cH$.
An operator $T: D \rightarrow \cH$  is called extended averaged with operator $M$, or $M$-averaged, if there exists a nonexpansive operator $R: D \rightarrow \cH$ such that $T = (I - M) + MR$.
\end{definition}

\begin{lem}[Characterization of extended averaged operator]\label{lem-ne}
Let $D$ be a nonempty subset of $\cH$, $T : D \rightarrow \cH$ be a single-valued operator, and  $M\in \cB(\cH)$ be self-adjoint and positive definite.
$T$ is $M$-averaged if and only if one of the following conditions holds:
\bi
\item[(i)] $(I-M^{-1}) + M^{-1}T$ is nonexpansive.
\item[(ii)] For any {$x, y\in D$, it holds}
\ben
\|Tx-Ty\|_{M^{-1}}^2 \leq \|x-y\|_{M^{-1}}^2-
\|(I-T)x-(I-T)y\|_{M^{-1}(I-M)M^{-1}}^2.
\een
\ei
\end{lem}
\proof
Note that, by definition, $T$ is $M$-averaged if and only if $T = (I - M) + MR$ with $R := (I-M^{-1}) + M^{-1}T$ being nonexpansive. Thus $T$ is $M$-averaged if and only if (i) holds. On the other hand, for any   $x,y\in D$,  it follows from the definition of $R$ and \eqref{id2} that
\begin{align}
    \|Rx - Ry\|^2
&= \|(I-M^{-1})(x-y) + M^{-1}(Tx - Ty)\|^2 \nonumber \\
&= \|x - y\|_{I- M^{-1}}^2 +\|Tx - Ty\|_{M^{-1}}^2 -\|(I -T)x-(I- T)y\|_{M^{-1}(I-M^{-1})}^2. \label{prop1}
\end{align}
Then, the equivalence between (i) and (ii) follows from the definition of $R$, the equality \eqref{prop1}, and
the relation $\|x - y\|_{I- M^{-1}}^2 = \|x - y\|^2 -  \|x-y\|_{M^{-1}}^2$.
\endproof

\begin{rem}
Let $T$ be an $M$-averaged operator with $M\succ 0$. If $M \prec I$, then $I-M\succ 0$, and it follows from Lemma \ref{lem-ne} (ii) that $T$ is nonexpansive with respect to the $M^{-1}$-norm.
However, if the condition $M \prec I$ does not hold, the nonexpansiveness of $T$ cannot be guaranteed, even though $T$ is $M$-averaged with $M\succ 0$.
\end{rem}

In this paper, we adopt the notation $\mathcal{O}(1/n)$ and $o(1/n)$ to characterize convergence rates. A nonnegative sequence $\{a_n\}$ is of order $\mathcal{O}(1/n)$ if there exists a constant $C > 0$ such that $a_n \leq C/n$ for all $n\geq 1$, while $a_n = o(1/n)$ signifies that $\lim_{n\rightarrow\infty} na_n = 0$. Next, we recall an important lemma which is useful for refining the ${\cal O}(1/n)$ rate to a faster rate $o(1/n)$.
\begin{lem}[\hspace{-0.01cm}{\cite[Lemma 1.1]{DLPYin17}}]\label{lem-con-rate}
Let $\{a_n\}_{n\in\N} \subseteq \R$ be a nonnegative sequence, monotonically nonincreasing and summable, i.e., $\sum_{n=1}^\infty a_n<+\infty$.
Then, we have $a_n= o(1/n)$ as $n\rightarrow\infty$.
\end{lem}

\section{A new PDSA and its properties}
\label{sec:Resolvent splitting}
In this section, we first introduce a basic PDSA scheme for solving the CMIP shown in \eqref{inclusion}.
Then, we present a new PDSA within this scheme. Finally, we figure out some of its properties.
The new PDSA contains both an extrapolation step like \eqref{PDSA-2} and a convex combination step like  \eqref{convex_com}.

\subsection{A basic PDSA scheme}
{Given    iterates} $\{(x_k,y_k,v_k)\}_{k\leq n}$ and a relaxation parameter $\eta>0$, we introduce the following basic PDSA scheme:
\be\label{bs}
\left\{\ba{l}
\text{compute~}  v_{n+1} \text{~from~} \{(x_k,v_k)\}_{k\leq n}, \smallskip \\
x_{n+1}=J_{\tau A}(v_{n+1}-\tau K^*y_{n}), \smallskip \\
\text{compute~}  (z_{n+1}, \omega_{n+1}) \text{~from~} \{x_k\}_{k\leq n+1}, \smallskip  \\
y_{n+1}=y_{n}+\eta\sigma \big( Kz_{n+1} -J_{B/\sigma}( y_{n}/\sigma+K \omega_{n+1}) \big),
\ea\right.
\ee
where $\tau>0$ and $\sigma>0$ are step sizes, $v_{n+1}$ and $(z_{n+1}, \omega_{n+1})$ are computed using only vector addition and scalar multiplication of the underlying sequences.
It is clear from \eqref{bs} that the resolvent operators $J_{\tau A}$ and $J_{B/\sigma}$ need to be computed only once per iteration.

Through an appropriate choice of $(v_{n+1}, z_{n+1}, \omega_{n+1})$ and $\eta$, the basic scheme \eqref{bs}
is capable of encompassing many well known operator splitting methods for solving \eqref{inclusion} or its special case \eqref{two-block}.
For instance, by choosing $(v_{n+1}, z_{n+1}, \omega_{n+1})=(x_{n}, 2x_{n+1}-x_{n}, 2x_{n+1}-x_{n})$ and $\eta=1$, the basic scheme \eqref{bs} simplifies to the
PDSA given in \eqref{PDSA}. For this choice, when confined to the composite convex optimization problem \eqref{two-block}, it is equivalent to CP-PDHG  proposed in \cite{Chambolle2011A}.
This equivalence can be discerned by leveraging Moreau's decomposition \eqref{Moreau} to deduce
\ben
{y_{n+1}}
=y_{n}+\sigma \left[Kz_{n+1} -J_{B/\sigma}( y_{n}/\sigma+K z_{n+1})\right]
=J_{\sigma B^{-1}}(y_n + \sigma K z_{n+1}).
\een
By choosing $v_{n+1}=\theta x_{n}+(1-\theta)v_{n}$ with $\theta=(\psi-1)/\psi$, i.e.,
the convex combination step \eqref{convex_com} is adopted, and setting  $z_{n+1}=\omega_{n+1} =x_{n+1}$ and  $\eta=1$,
the basic scheme \eqref{bs} simplifies to the PDAc proposed in \cite{ChY2022relaxed}. Specifically, the iterative scheme of PDAc for solving \eqref{two-block} is given by
\ben
\left\{\ba{ll}
v_{n+1}&=\theta x_{n}+(1-\theta)v_{n},\\
x_{n+1}&=\prox_{\tau g}(v_{n+1}-\tau K^*y_{n}),\\
y_{n+1}&=y_{n}+\sigma \big( Kx_{n+1} -\prox_{f/\sigma}( y_{n}/\sigma+K x_{n+1})\big),
\ea\right.
\een
where $\prox_g := J_{\partial g}$ denotes the proximity operator of $g$.
An advantage of PDAc lies in its ability to take larger step sizes, see \cite[Theorem 2.1]{ChY2022relaxed} for details.
However, because of the convex combination step, the derivation of the convergence of PDAc from the perspective of fixed-point iteration remains ambiguous.
The same difficulty has also been pointed out in \cite{Malitsky2019Golden} for the golden ratio algorithm.
Furthermore, it is uncertain whether PDAc can be directly extended to solve the CMIP problem \eqref{inclusion}.

In this paper, we propose the following particular choice of $(v_{n+1}, z_{n+1}, \omega_{n+1})$ in \eqref{bs}:
$v_{n+1} =\theta x_{n}+(1-\theta)v_{n}$,
$z_{n+1} =x_{n+1}+ (\theta/\eta)(x_{n+1}-v_{n+1})$ and
$\omega_{n+1}=x_{n+1}$,
where the involved parameters will be clarified below.
For convenience,  we define
\be\label{def:parameter-set}
\gamma := \tau\sigma \text{~and~}
\Theta := \big\{(\theta,\eta, \gamma)~|~ \theta \in (0,2), \, \eta  \in (0,2), \, \gamma\in (0,\infty), \, \gamma \|K\|^2<(2-\theta)(2-\eta)\big\}.
\ee
The detailed algorithmic framework of the basic scheme \eqref{bs} with the above mentioned choice of $(v_{n+1}, z_{n+1}, \omega_{n+1})$ is summarized below in Algorithm \ref{alg-os}. We still refer to the resulting algorithm as PDSA with convex combination.

\vskip5mm
\hrule\vskip2mm
\begin{algo}
[PDSA with convex combination]\label{alg-os}
{~}\vskip 1pt {\rm
\begin{description}
\item[{\em Step 0.}] Choose $(\theta,\eta,\gamma)\in\Theta$, $\tau>0$, and set $\sigma = \gamma/\tau$. Initialize  $x_0= v_{0}\in \cH$, $y_0\in \cG$ and $n=0$.
\item[{\em Step 1.}] Compute $v_{n+1}$, $x_{n+1}$, $z_{n+1}$ and $y_{n+1}$ sequentially as follows:
\begin{subequations}\label{os}
    \begin{align}
v_{n+1} &=\theta x_{n}+(1-\theta)v_{n}, \label{os-z}\\
x_{n+1}&=J_{\tau A}(v_{n+1}-\tau K^* y_{n}), \label{os-x}\\
z_{n+1} &=x_{n+1}+ \theta(x_{n+1}-v_{n+1})/\eta, \label{os-extra}\\
y_{n+1}&=y_{n}+\eta \sigma \big(Kz_{n+1} -J_{B/\sigma}( y_{n}/\sigma+Kx_{n+1})\big).\label{os-y}
\end{align}
\end{subequations}
\item[{\em Step 2.}] Set $n\leftarrow n + 1$ and return to Step 1.
  \end{description}
}
\end{algo}
\vskip1mm\hrule\vskip5mm

It is evident that the expression in \eqref{os-z} bears resemblance to that in \eqref{convex_com} and represents a convex combination step.
Besides,  it is easy to derive by induction that $v_{n+1}$ is a convex combination of $\{x_i\}_{i\leq n}$, and thus the update in \eqref{os-extra} can also be viewed as an extrapolation step in the direction $(x_{n+1}-v_{n+1})$, instead of $(x_{n + 1}-x_n)$ as in \eqref{PDSA-2}. By the  Moreau decomposition formula \eqref{Moreau}, $y_{n+1}$ defined in \eqref{os-y} can be rewritten as
\be\label{os-y1}
y_{n+1}=y_{n}+\eta \left(J_{\sigma B^{-1}}( y_{n}+\sigma Kx_{n+1})+\sigma K(z_{n+1}-x_{n+1}) - y_n\right).
\ee
This expression highlights $\eta$ as a relaxation parameter. Notably, when $\theta = \eta = 1$ in \eqref{os} and $\rho=1$ in \eqref{PDSA},  $z_{n+1}$ defined in both frameworks appears to be the same.
However, direct verification shows that when $\theta = \eta = \rho = 1$, the only difference between \eqref{os} and \eqref{PDSA} is the extra term $\sigma K(z_{n+1} - x_{n+1})$ in \eqref{os-y1}.
Since $z_{n+1} \neq x_{n+1}$, our algorithm does not reduces to \eqref{PDSA}.

As can be seen from \eqref{def:parameter-set}, the range of $\tau\sigma\|K\|^2$ can be as wide as $(0,4)$.
In practical implementations, balancing the parameters $(\theta, \eta, \gamma)$ is crucial for achieving optimal performance. As evidenced by the numerical results in Section \ref{sec6.2}, favorable outcomes consistently arise from using a moderate degree of over-relaxation (i.e., $\eta > 1$ but not excessively large) in conjunction with a relatively small convex combination parameter $\theta$.

\begin{rem}
We emphasize that Algorithm \ref{alg-os} has the same major computational load as PDSA given in \eqref{PDSA}. Each iteration in both algorithms involves only one evaluation of the resolvent operators $J_A$ and $J_B$ (or $J_{B^{-1}}$, considering \eqref{Moreau}), and one evaluation of the linear operator $K$ and its adjoint $K^*$.
From \eqref{os-extra} and \eqref{os-z}, we  observe that $Kz_{n + 1}=Kx_{n+1}+  \theta (Kx_{n+1}-Kv_{n+1})/\eta$ and $Kv_{n+1}=\theta Kx_{n}+(1 - \theta)Kv_{n}$. As a result, in each iteration, we only need to evaluate the linear operator $K$ on $x_{n+1}$ once. Subsequently, $Kz_{n+1}$ and $Kv_{n+1}$ can be recursively computed without additional evaluation of $K$, requiring only vector addition and scalar multiplication operations.
Moreover, during the computation process, we only require the value of $Kz_{n+1}$, rather than the value of $z_{n+1}$ itself.
\end{rem}

\subsection{Properties of Algorithm  \ref{alg-os}   from fixed-point iteration perspective}
In this section, we analyze some properties of the sequence generated by Algorithm \ref{alg-os} from the fixed-point iteration perspective.
Let  $P\in \cB(\cH\times\cG)$ be a self-adjoint and positive definite operator with a $2\times 2$ block structure, where the $(1,1)$-block of $P$ has the same dimension as $\cH$ and the $(2,2)$-block has the same dimension as $\cG$.
Define operator $T_P: \cH\times\cG \rightarrow \cH\times\cG$ in a sequential manner as
\begin{subequations}\label{def:T}
\begin{align}
\bz&=(v,u)^\top \in \cH\times\cG, \label{def:zzz}\\
x&=J_{\tau A}(v - \tau K^*(\sigma Kv-u/\tau)), \label{def:wa} \\
w&=J_{B/\sigma} \left(Kv+Kx-u/\gamma\right),  \label{def:wb} \\
T_P (\bz) &:=\bz+ P \begin{pmatrix} \ba{c}x-v\\w-K x\ea\end{pmatrix}. \label{def:T3}
\end{align}
\end{subequations}
For the sequence $\{(x_n,v_n, y_n)\}_{n\in\N}$ generated by Algorithm \ref{alg-os}, in what follows, for $n\geq 1$, we let  $u_{n} := \gamma K v_{n}-\tau  y_{n-1}$ and $\bz_n:=(v_n,u_n)^\top$.

\begin{lem}\label{lem-a}
Let  $\{(x_n,v_n, y_n)\}_{n\in\N}$ be the sequence generated by Algorithm \ref{alg-os}.
 Then, we have
$\bz_{n+1} = T_P (\bz_n)$ for $n\geq 1$, where $T_P$ is defined  in \eqref{def:T} with  $P=\Diag(\theta I_{\cH},\eta\gamma I_{\cG})\succ0$.
\end{lem}
\proof
Recall that $\gamma = \tau\sigma$. Since  $u_{n}= \gamma K v_{n}-\tau  y_{n-1}$, it follows from \eqref{os-x} and \eqref{os-z} that
\be\label{x_n_op}
x_n=J_{\tau A}\big(v_n - \tau K^*(\sigma Kv_n-u_{n}/\tau)\big),
\ee
and thus
\be\label{z-iter}
v_{n+1} = v_{n}+\theta (x_n - v_n) = v_{n}+\theta\big[ J_{\tau A}\big(v_n - \tau K^*(\sigma Kv_n-u_{n}/\tau)\big) -v_{n}  \big].
\ee
Then, it is elementary to derive from the definition of $u_{n+1}$, \eqref{z-iter}, \eqref{os-extra} and \eqref{os-y} that
\be\label{os-u}
u_{n+1}&=&\gamma Kv_{n+1}-\tau y_{n}\nonumber\\
&=&\gamma  K \big(v_{n}+\theta (x_n - v_n)\big)  -\tau \left[
y_{n-1} + \eta \sigma  K \Big(x_n +    {\theta\over \eta} (x_n-v_{n}) \Big) - \eta \sigma J_{B/\sigma}\Big({y_{n-1}\over \sigma} +Kx_{n}\Big)
\right]
\nonumber\\
&=&u_{n}+\eta\gamma \left(J_{B/\sigma} \left(K (v_{n}+ x_n)-u_n/\gamma\right)-Kx_n\right),
\ee
where the last equality follows from combining like terms and using $y_{n-1}/\sigma  = K v_{n} -u_n/\gamma$.
The proof is completed by combining \eqref{x_n_op}-\eqref{os-u}.
\endproof

\begin{rem}
Operators analogous to $T_P$  have been implicitly explored in \cite{Automated2025Upadhyaya}, where an automated tight Lyapunov analysis for a class of first-order optimization algorithms is conducted. In this study, $T_P$ is utilized to analyze a splitting method from the perspective of generalized firmly nonexpansive operators, a viewpoint that, to the best of our knowledge, has not been considered in previous literature.
\end{rem}

Next, we present a proposition regarding the operator $T_P$ defined  in \eqref{def:T}, which plays a crucial role in the convergence analysis of Algorithm \ref{alg-os}.

\begin{prop}\label{l:T_A ne}
Let $P\in \cB(\cH\times\cG)$ be a self-adjoint positive definite operator. Suppose that $P$ has a $2\times2$ block structure, and satisfies $0\prec P\prec \Phi_K$, where $\Phi_K$ is defined as
\be\label{def:Phi}
\Phi_K:=\left[\ba{cc}2I~&~ \gamma K^*\\  \gamma K~&~2\gamma I\ea\right].
\ee
For any $\bz=(v,u)^\top\in\cH\times \cG$ and $\bar{\bz}=(\bar{v},\bar{u})^\top\in\cH\times\cG$,
we have
\be\label{eq:T_A_ne}
 \|T_P(\bz)-T_P(\bar{\bz})\|_M^2
  \leq \|\bz-\bar{\bz}\|_M^2- \|(I-T_P)(\bz)-(I-T_P)(\bar{\bz})\|_Q^2,
\ee
where  $M :=P^{-1}$ and
\be\label{def:Q}
Q:=M^*\left(\Phi_K-P\right)M\succ 0.
\ee
\end{prop}
\proof
For simplicity, denote $\bz^+=T_P(\bz)$ and $\bar{\bz}^+=T_P(\bar{\bz})$.
From  \eqref{def:T3} and $M=P^{-1}$, we have
\be\label{w-uv}
\left(\ba{c}v-x \\ Kx-w\ea\right)=M (\bz-\bz^+).
\ee
Recall that, for given $(v,u)$, $x$ and $w$ are defined in \eqref{def:wa}-\eqref{def:wb}.
We define $\bar{x}$ and $\bar{w}$ analogously by using $\bar{\bz}=(\bar{v},\bar{u})^\top$, i.e.,
$\bar x =J_{\tau A}(\bar v - \tau K^*(\sigma K\bar v-\bar u/\tau))$ and
$\bar w =J_{B/\sigma} \left(K\bar v+K\bar x-\bar u/\gamma\right)$.
Since $v - K^*(\gamma Kv-u) -x\in \tau A(x)$ and $\bar{v} - K^*(\gamma K\bar{v}-\bar{u}) -\bar{x}\in \tau A(\bar{x})$, it follows from the monotonicity of $A$ that
\be\label{eq:mono A}
 0 &\leq& \langle  x-\bar{x}, (v - K^*(\gamma Kv-u) -x)-(\bar{v} - K^*(\gamma K\bar{v}-\bar{u}) -\bar{x})\rangle\nonumber\\
 &=&\langle x-\bar{x}, (v-\bar{v})-(x-\bar{x})\rangle-\langle\gamma (Kv-K\bar{v})-(u-\bar{u}), Kx-K\bar{x}\rangle.
\ee
Analogously, we have $\gamma(Kv+Kx)-u-\gamma w\in \tau B(w)$ and $\gamma(K\bar{v}+K\bar{x})-\bar{u}-\gamma \bar{w}\in \tau B(\bar{w})$. Then, the monotonicity of $B$ yields
\be\label{eq:mono B}
 0 &\leq& \langle w-\bar{w},(\gamma(Kv+Kx)-u-\gamma w)-(\gamma(K\bar{v}+K\bar{x})-\bar{u}-\gamma \bar{w})\rangle \nonumber\\
   &=& -\langle w-\bar{w}, u-\bar{u}\rangle+\gamma \langle w-\bar{w}, (Kv+Kx)- (K\bar{v}+K\bar{x})-(w-\bar{w})\rangle.
\ee
By adding \eqref{eq:mono A} and \eqref{eq:mono B} and performing basic algebraic operations, we obtain
\be\label{eq:before key}
0&\leq& \langle x-\bar{x}, (v-\bar{v})-(x-\bar{x})\rangle+\langle u-\bar{u}, (Kx-K\bar{x})-(w-\bar{w})\rangle\nonumber\\
&&+\gamma \langle w-\bar{w}, (Kv+Kx)- (K\bar{v}+K\bar{x})-(w-\bar{w})\rangle {-\gamma\langle Kv-K\bar v, Kx-K\bar{x}\rangle.}
\ee

Next, we reformulate each term on the right-hand-side  of \eqref{eq:before key}  using equivalent transforms.
By adding the term $\langle\bar{v}-v, (v-\bar{v})-(x-\bar{x})\rangle$ to the first term in \eqref{eq:before key}, we can express it as
\be\label{eq:ip1}
\langle x-\bar{x}, (v-\bar{v})-(x-\bar{x})\rangle +\langle\bar{v}-v, (v-\bar{v})-(x-\bar{x})\rangle
= -\|(v-x) -(\bar{v}-\bar{x})\|^2.
\ee
By subtracting $\langle\bar{v}-v, (v-\bar{v})-(x-\bar{x})\rangle$ from the second term in \eqref{eq:before key} and using \eqref{def:T3}, it can be expressed as
\begin{align}
 & \langle (Kx-w)-(K\bar{x}-\bar{w}),  u-\bar{u}\rangle + \langle (v-x)-(\bar{v}-\bar{x}), v-\bar{v}\rangle \nonumber\\
 =\, &\langle (\bz-\bar{\bz})-(\bz^+-\bar{\bz}^+),\bz-\bar{\bz}\rangle_M \nonumber\\
 =\, &\frac{1}{2}\left(\|(\bz-\bz^+)-(\bar{\bz}-\bar{\bz}^+)\|_M^2+\|\bz-\bar{\bz}\|_M^2 -\|\bz^+-\bar{\bz}^+\|_M^2 \right), \label{eq:ip2}
\end{align}
where the first equality follows from \eqref{w-uv} and the fact that $M$ is self-adjoint and positive definite.
By using \eqref{id}, the third term in \eqref{eq:before key} (excluding the scalar $\gamma$) can be represented as
\be\label{eq:ip3}
 \langle w-\bar{w},   (Kv+Kx) &-& (K\bar{v}+K\bar{x})-(w-\bar{w})\rangle
=  \frac{1}{2}\|Kv-K\bar{v}\|^2+\frac{1}{2}\|Kx- K\bar{x}\|^2  \nonumber\\
&& -\frac{1}{2}\|Kx- K\bar{x}-(w-\bar{w})\|^2 -\frac{1}{2}\|w-\bar{w}-(Kv-K\bar{v})\|^2.
\ee
The last term in \eqref{eq:before key} (excluding the scalar $\gamma$) can be represented as
\be\label{eq:ip4}
- \langle K(v- \bar v), K(x- \bar{x})\rangle = \frac{1}{2}\big(
\|K(v-\bar v)- K(x-\bar{x})\|^2 - \|K(v-\bar{v})\|^2- \|K(x-\bar{x})\|^2\big).
\ee
Summing  \eqref{eq:ip3} and \eqref{eq:ip4} yields
\be\label{eq:ip34}
&& \langle w-\bar{w}, (Kv+Kx)- (K\bar{v}+K\bar{x})-(w-\bar{w})\rangle
- \langle K(v-\bar v), K(x-\bar{x})\rangle\nonumber\\
&=&\frac{1}{2}\|(Kv-K\bar{v})- (Kx-K\bar{x})\|^2-\frac{1}{2}\|Kx- K\bar{x}-(w-\bar{w})\|^2 -\frac{1}{2}\|w-\bar{w}-(Kv-K\bar{v})\|^2\nonumber\\
&=&\frac{1}{2}\|(\bz-\bz^+)-(\bar{\bz}-\bar{\bz}^+)\|_{G}^2 -\frac{1}{2}\|w-\bar{w}-(Kv-K\bar{v})\|^2,
\ee
where $G:=M^* \Diag(K^*K, -I_{\cG})M$ and the last equality follows from \eqref{w-uv}.
Finally, multiply \eqref{eq:ip34} by $\gamma$, add it to the sum of \eqref{eq:before key}-\eqref{eq:ip2}, and then multiply the result by 2. We then obtain
\be\label{eq:final}
\|\bz^+-\bar{\bz}^+\|_M^2&\leq& \|\bz-\bar{\bz}\|_M^2+\|(\bz-\bz^+)-(\bar{\bz}-\bar{\bz}^+)\|_M^2-2\|(v-x) -(\bar{v}-\bar{x})\|^2  \nonumber\\
&&+\gamma \|(\bz-\bz^+)-(\bar{\bz}-\bar{\bz}^+)\|_{G}^2  -\gamma\|w-\bar{w}-(Kv-K\bar{v})\|^2  \nonumber\\
&=&\|\bz-\bar{\bz}\|_M^2+\|(\bz-\bz^+)-(\bar{\bz}-\bar{\bz}^+)\|_M^2   - \|(\bz-\bz^+)-(\bar{\bz}-\bar{\bz}^+)\|_{\widehat{G}}^2  \nonumber \\
&& -\gamma\|w-\bar{w}-(Kv-K\bar{v})\|^2  \nonumber\\
&=&\|\bz-\bar{\bz}\|_M^2-\|(\bz-\bz^+)-(\bar{\bz}-\bar{\bz}^+)\|_{\widehat{G}-M}^2 -\gamma\|w-\bar{w}-(Kv-K\bar{v})\|^2,
\ee
where $\widehat{G} :=M^*\diag(2I_{\cH}-\gamma K^*K, \gamma I_{\cG}) M$.
Note that \eqref{w-uv} implies
\be
\|w-\bar{w}-(Kv-K\bar{v})\|^2 &=& \|(Kv-Kx+Kx-w)-(K\bar{v}-K\bar{x}+K\bar{x}-\bar{w})\|^2 \nonumber \\
&=&\|(\bz-\bz^+)-(\bar{\bz}-\bar{\bz}^+)\|_{\widehat{M}}^2, \label{jy-1}
\ee
where $\widehat{M} := M^* E_K M$ with $E_K :=\left[\ba{cc}K^*K~&~ K^*\\  K~&~I_{\cG}\ea\right]$.
Finally, \eqref{eq:T_A_ne} follows from
substituting \eqref{jy-1}  into \eqref{eq:final} and noting
\ben
\widehat{G}-M+\gamma\widehat{M} &=& M^*\left(\diag(2I_{\cH}-\gamma K^*K, \gamma I_{\cG})+\gamma E_K-M^{-1}\right)M\\
&=&M^*\left(\left[\ba{cc}2I~&~ \gamma K^*\\  \gamma K~&~2\gamma I\ea\right]-P\right)M
= M^* (\Phi_K-P) M \stackrel{\eqref{def:Q}}= Q \succ 0.
\een
The proof is completed.
\endproof

\begin{rem}\label{rem-tight}
In the proof of Proposition \ref{l:T_A ne}, apart from using the monotonicity of $A$ and $B$, we exclusively rely on identical transformations without any inequality relaxations.
Consequently, the upper bound result given in \eqref{eq:T_A_ne} cannot be improved.
\bi
\item[(i)] Given that $P\succ 0$, it is straightforward to note from the definition  of  $Q$ in \eqref{def:Q} that the upper bound condition $P \prec\Phi_K$ is both a sufficient and necessary condition for ensuring $Q\succ 0$. The condition $Q\succ 0$, which is equivalent to $0 \prec P \prec\Phi_K$, plays a crucial and indispensable role in guaranteeing the convergence of the fixed-point iteration on the operator $T_{P}$, see Theorem \ref{th:main}. Moreover, the condition $0\prec P \prec\Phi_K$ is sharp in ensuring convergence. An illustration through a concrete example can be found in Section \ref{sec:he}.

\item[(ii)] When $Q = 0$ (i.e., $P = \Phi_K \succ 0$), the operator $T_P$ is merely nonexpansive in the $M$-norm. According to Baillon's nonlinear mean ergodic theorem \cite{Baillon1975mean}, for any initial point $\bz_0 \in \mathcal{H} \times \mathcal{G}$, the Cesaro mean $\frac{1}{N} \sum_{k=0}^{N-1} T_P^k(\bz_0)$ converges weakly to a fixed-point of $T_P$ as $N \to +\infty$.

\item[(iii)] Given the property \eqref{eq:T_A_ne}, which generalizes the key property holding for classical firmly nonexpansive operators, it can be verified that the Browder's demiclosedness principle \cite[Theorem 4.27]{Bauschke2011book} and the convergence guarantees of the  Krasnosel'ski\u{i}-Mann fixed-point iteration \cite[Theorem 5.15]{Bauschke2011book} hold as well for $T_P$. These results will be useful in the proofs of our Theorems \ref{th:main} and \ref{th:limit-case}.
\ei
\end{rem}

\begin{rem}\label{rem-P}
For a general single-valued operator $T$ defined on an arbitrary Hilbert space,  $T$ is said to be extended firmly nonexpansive in the $M$-norm if it satisfies
\[\label{jy-0717}
\|T(x)-T(y)\|_M^2 \leq \|x-y\|_M^2 - \|(x-T(x)) - (y-T(y))\|_Q^2, \quad \forall x, y,
\]
where $M$ and $Q$ are bounded linear operators that are self-adjoint and positive definite.
The term ``extended" is used here to emphasize that $Q \neq M$ is permitted. In particular, if
$T$ satisfies \eqref{jy-0717} with $Q=M$, then $T$  is called firmly nonexpansive in the $M$-norm.
For the operator $T_P$ defined in \eqref{def:T}, we present the following remarks based on the key property \eqref{eq:T_A_ne}.

\bi
\item[(i)] Lemma \ref{lem-a} shows that Algorithm \ref{alg-os} can be considered as a fixed-point iteration on the operator $T_P$ as defined in \eqref{def:T} with $P = \diag(\theta I_{\cH},\eta\gamma I_{\cG})$. In this case, we have $M=M^*=P^{-1}=\diag(\frac{1}{\theta}I_{\cH}, \frac{1}{\eta \gamma} I_{\cG})\succ0$ and
\ben
Q:=M^*\left(\Phi_K-P\right)M=\left[\ba{cc}\frac{2-\theta}{\theta^2} I~&~\frac{1}{\theta \eta } K^*\\ \frac{1}{\theta \eta} K~&~\frac{2-\eta}{\eta^2\gamma} I\ea\right]\succ 0
\text{~~for any~~}(\theta,\eta, \gamma)\in\Theta.
\een
Thus, it follows from \eqref{eq:T_A_ne} that $T_P$ is extended firmly nonexpansive in the $M$-norm, and in fact, this holds whenever $P$ satisfies
$0 \prec P \prec \Phi_K$.

\item[(ii)] If  $P$ is chosen such that $0\prec P\preceq\frac{1}{2}\Phi_K$,
then from (\ref{def:Q}) and $M = P^{-1}$ we have
$Q-M = P^{-1}(\Phi_K-2P)P^{-1}\succeq0$, and thus $Q\succeq M\succ 0$.
As a result,  in this case $T_P$ is firmly nonexpansive in the $M$-norm since
by \eqref{eq:T_A_ne} we have
     \ben
 \|T_P(\bz)-T_P(\bar{\bz})\|_M^2
  \leq \|\bz-\bar{\bz}\|_M^2 - \|(I-T_P)(\bz)-(I-T_P)(\bar{\bz})\|_M^2.
\een

\item[(iii)] If $\Phi_K\succeq I$ is satisfied, then by (\ref{def:Q}) and $M=P^{-1}\succ0$ we have
\ben
Q = M^*\left(\Phi_K-P\right)M \succeq M^*\left(I-P\right)M=P^{-1}\left(I-P\right)P^{-1}.
 \een
Then it follows from  \eqref{eq:T_A_ne} and Lemma \ref{lem-ne} (ii) that $T_P$ is $P$-averaged.
Note that from \eqref{def:Phi}, the condition $ \Phi_K \succeq I $ is equivalent to $ \gamma^2\|K\|^2 \leq 2\gamma - 1 $; this condition is satisfiable for appropriately chosen $ \gamma $ if and only if $ \|K\| \leq 1 $.
\ei
\end{rem}

\section{Convergence analysis of Algorithm \ref{alg-os}}
\label{sec:Convergence analysis}

\subsection{Convergence}
Next,  we characterize the set of fixed points of the operator $T_P$ by using the primal-dual solutions given in \eqref{cond-solver}.
Again, recall that $\gamma=\tau\sigma$.
\begin{lemma}\label{l:fixed points}
Let the operator $T_P$ be defined as in \eqref{def:T}. Then the following statements hold.
\begin{enumerate}[(a)]
\item  If $(v^\star,u^\star)\in\Fix T_P$, then $(v^\star,\sigma K v^\star-u^\star/\tau)\in\Omega$.
\item  If $(x^\star,y^\star)\in\Omega$, then $(x^\star,\gamma Kx^\star-\tau y^\star)\in\Fix T_P$ and $x^\star\in \zer\left(A+K^*BK\right)$.
\end{enumerate}
Consequently, $\Fix T_P \neq\varnothing$ if and only if $\Omega\neq\varnothing$, and $\zer\left(A+K^*BK\right)\neq\varnothing$ if either $\Fix T_P$ or
$\Omega$ is nonempty.
\end{lemma}
\proof
(a) Let $(v^\star,u^\star)\in\Fix T_P$ and set $x^\star=J_{\tau A}(v^\star - \tau K^*(\sigma Kv^\star-u^\star/\tau))$.
Since $P\succ0$, \eqref{def:T} implies that $v^\star=x^\star=J_{\tau A}(v^\star - \tau K^*(\sigma Kv^\star-u^\star/\tau))$
and $J_{B/\sigma} \left(Kv^\star+Kx^\star-u^\star/\gamma\right)=Kx^\star$. Let $y^\star = \sigma Kv^\star-u^\star/\tau$. Then,  we have
\ben
A(v^\star)\ni- K^*(\sigma Kv^\star-u^\star/\tau)=-K^*y^\star \text{~~and~~}Kv^\star=Kx^\star = J_{B/\sigma} \left(Kv^\star+y^\star/\sigma\right),
\een
which implies $y^\star\in B(Kv^\star)$. Together with $A(v^\star)\ni - K^*y^\star$, we obtain $(v^\star,y^\star)\in\Omega$.
Recall that $y^\star = \sigma Kv^\star-u^\star/\tau$. The result $(v^\star,\sigma Kv^\star-u^\star/\tau)\in\Omega$ follows.

(b) Let $(x^\star,y^\star)\in\Omega$ and set $u^\star=\gamma Kx^\star-\tau y^\star$.
It follows from \eqref{cond-solver} that
\be\label{lem-in1}
(I+\tau A)(x^\star)\ni x^\star-\tau K^*y^\star=x^\star-\tau K^*(\sigma Kx^\star-u^\star/\tau).
\ee
Thus, we have $x^\star=J_{\tau A}(x^\star - \tau K^*(\sigma Kx^\star-u^\star/\tau))$. Furthermore, from \eqref{cond-solver} we have
\be\label{lem-in2}
Kx^\star+ \sigma^{-1} B(Kx^\star) \ni Kx^\star+ y^\star/\sigma =2Kx^\star-  u^\star/\gamma.
\ee
This further implies $Kx^\star=J_{B/\sigma} \left(2Kx^\star-u^\star/\gamma\right)$ and
$(x^\star,u^\star)=(x^\star,\gamma Kx^\star-\tau y^\star)\in\Fix T_P$.
Using \eqref{lem-in1}-\eqref{lem-in2} and $y^\star = \sigma Kx^\star-u^\star/\tau$, we obtain
\ben
A(x^\star)\ni  -K^*(\sigma Kx^\star-u^\star/\tau)~~\mbox{and}~~
B(Kx^\star)\ni \sigma Kx^\star-u^\star/\tau.
\een
This implies $K^*B(Kx^\star)\ni K^*(\sigma Kx^\star-u^\star/\tau)$, and thus $0\in A(x^\star)+ K^*B(Kx^\star)$.
The rest of the theorem's statements follow directly.
\endproof

\begin{theorem}\label{th:main}
Let $P$ be a self-adjoint operator  that satisfies $0\prec P\prec \Phi_K$, and  let $ T_P$ be defined as in \eqref{def:T} so that   $\Fix  T_P \neq \varnothing$. Let $\{\bz_n:=(v_n,u_n)^\top\}_{n \in\N}$ be generated by
$\bz_{n+1} = T_{P} (\bz_n)$ from any $\bz_0 \in \cH\times \cG.$
Then the following assertions hold.
\begin{enumerate}[(a)]
\item The sequence $\{\bz_n\}_{n \in\N}$ converges weakly to a point $\bz^\star\in\Fix T_P$.
\item The sequence $\{v_n\}_{n \in\N}$ converges weakly to {a point   $v^\star \in\zer\left(A+K^*BK\right)$.}
\end{enumerate}
\end{theorem}
\begin{proof}
(a). Since $\Fix T_P\neq\varnothing$ and $0\prec P\prec \Phi_K$,  for every $\bz^\star\in \Fix  T_P$, Proposition~\ref{l:T_A ne} implies that
\ben
 \|\bz_{n+1}-\bz^\star\|_M^2 + \|\bz_{n+1}-\bz_n\|_Q^2
  \leq \|\bz_n-\bz^\star\|_M^2,
\een
from which we can easily derive $\sum_{n \in\N}\|\bz_{n+1}-\bz_n\|_Q^2\leq \|\bz_0-\bz^\star\|_M^2$, which implies $\lim_{n\to\infty}\|\bz_{n+1}-\bz_n\|=0$ from $Q\succ0$. Recall that the convergence results of Krasnosel'ski\u{i}-Mann fixed-point iteration can be generalized to the operator $T_P$, see Remark \ref{rem-tight} (iii). Then, by applying the similar procedure as in \cite[Theorem~5.15]{Bauschke2011book}, we therefore deduce that $\{\bz_n\}_{n \in\N}$ converges weakly to a point $\bz^\star\in\Fix T_P$.

(b). Recall $\Fix T_P\neq\varnothing$, Lemma~\ref{l:fixed points} implies $\zer\left(A+K^*BK\right)\neq\varnothing$. Note that $\bz_n=(v_n,u_n)^\top$ and (a), we observe $\{(v_n,u_n)\}_{n\in\N}$  converges weakly to a point $(v^\star,u^\star)\in \Fix T_P$. This together with Lemma~\ref{l:fixed points} implies that $v^\star \in\zer\left(A+K^*BK\right)$.
\end{proof}

\begin{cor}
Let $\{(x_n,y_n)\}_{n \in\N}$ be generated by Algorithm \ref{alg-os}. Then the sequence $\{x_n\}_{n \in\N}$ converges weakly to {a point   $v^\star \in\zer\left(A+K^*BK\right)$.}
\end{cor}
\proof
From Lemma \ref{lem-a}, Algorithm \ref{alg-os}   can be regarded as a fixed-point iteration of $T_P$ with $P=\Diag(\theta I_{\cH},\eta\gamma I_{\cG})\succ0$.
For any $(\theta,\eta,\gamma)\in \Theta$ defined in \eqref{def:parameter-set}, we have
\ben
\Phi_K-P=\left[\ba{cc}(2-\theta)I~&~ \gamma K^*\\  \gamma K~&~(2-\eta)\gamma I\ea\right]\succ0.
\een
Consequently, we obtain $0\prec P\prec\Phi_K$, which satisfies the condition stated in Theorem \ref{th:main}.
By Theorem \ref{th:main}(b), $\{v_n\}_{n \in\N}$ converges weakly to a point $v^\star \in\zer\left(A+K^*BK\right)$.
The sequence $\{x_n\}_{n \in\N}$ converges weakly to $v^\star$ as well since $x_n= (v_n-(1-\theta) v_{n-1})/\theta$.
\endproof

\subsection{Convergence rate when $0\prec P\preceq\frac{1}{2}\Phi_K$}
For any $\bz^\star\in \Fix T_P$, it follows from $\bz_{n+1}=T_P(\bz_{n})$ and Proposition~\ref{l:T_A ne} that
\ben
 \|\bz_{n+1}-\bz^\star\|_M^2\leq \|\bz_{n}-\bz^\star\|_M^2- \|\bz_{n+1}-\bz_{n}\|_Q^2, \quad \forall n\geq 1.
\een
Consequently,  $\|\bz_{n+1}-\bz_{n}\|_Q^2\leq \|\bz_{n}-\bz^\star\|_M^2$ for all $n\geq 1$, the sequence $\{\|\bz_{n+1}-\bz^\star\|_M^2\}_{n\in\N}$ is monotonically nonincreasing, and
\be\label{sum-z}
\sum_{n=0}^{N-1}\|\bz_{n+1}-\bz_{n}\|_Q^2\leq\|\bz_{0}-\bz^\star\|_M^2, \quad \forall N\geq1.
\ee
Next, we  derive an asymptotic convergence rate of the fixed point residual $\|\bz_{n-1} - T_P(\bz_{n-1})\|$, which is identical to   $\|\bz_{n-1} - \bz_{n}\|$, under the more demanding condition $0\prec P\preceq\frac{1}{2}\Phi_K$.

\begin{theorem}\label{th:rate}
Assume $0\prec P\preceq\frac{1}{2}\Phi_K$, and let
$\{\bz_n\}_{n \in\N}$ be the sequence generated by $\bz_{n+1}=T_P(\bz_{n})$ from any initial point   $\bz_0 \in \cH\times \cG$. Then, for any $N\geq1$ and  $\bz^\star\in \Fix T_P$, we have
$\|\bz_{N}-\bz_{N-1}\|_M^2 \leq \|\bz_{0}-\bz^\star\|_M^2/N$. Furthermore, there holds $\|\bz_{N}-\bz_{N-1}\|_M^2= o(1/N)$ as $N\rightarrow\infty$.
\end{theorem}
\proof
Based on Remark \ref{rem-P} (ii), under the condition $0\prec P\preceq\frac{1}{2}\Phi_K$, $T_P$ is firmly nonexpansive in the $M$-norm. Then, it is elementary to show that $\|\bz_{n+1}-\bz_{n}\|_M$ is monotonically nonincreasing and $\sum_{n=0}^{N-1}\|\bz_{n+1}-\bz_{n}\|_M^2\leq \|\bz_{0}-\bz^\star\|_M^2$, from which it follows
$\|\bz_{N}-\bz_{N-1}\|_M^2 \leq \|\bz_{0}-\bz^\star\|_M^2/N$ for any $N\geq 1$. Moreover, the sequence $\{\|\bz_{n}-\bz_{n-1}\|_M^2\}_{n\in\N}$ fulfills
the conditions of Lemma \ref{lem-con-rate}, and thus it holds that $\|\bz_{N}-\bz_{N-1}\|_M^2= o(1/N)$ as $N\rightarrow\infty$.
\endproof

\subsection{Convergence rate for composite convex optimization problem}
In this section, we apply Algorithm \ref{alg-os} to the composite convex optimization problem \eqref{two-block}, which corresponds to \eqref{inclusion} with  $A = \partial g$ and $B=\partial f$. In this case, we have $J_{\tau A} = \prox_{\tau g}$ and $J_{B/\sigma} = \prox_{f/\sigma}$, where $\prox_{\tau g}$ and $\prox_{f/\sigma}$ are the proximal operators\cite[Definition 12.23]{Bauschke2011book} of $\tau g$ and $f/\sigma$, respectively. For any $x\in \cH$, it holds that $z = \prox_{\tau h}(x)$ if and only if
\be\label{property-p}
\tau \big(h(y) - h(z)\big) \geq  \langle x-z, y-z\rangle,~~ \forall y\in \cH.
\ee
By introducing an auxiliary variable $w\in\cG$, we can rewrite \eqref{two-block} equivalently as
\be\label{two-block2}
\min\nolimits_{x,w}  \{ g(x)+f(w) ~\mid~ Kx-w=0, \, x\in\cH, \, w\in\cG\}.
\ee
We shall establish ${\cal O}(1/N)$ ergodic convergence rate of Algorithm \ref{alg-os} measured by the function value gap and constraint violations of the equivalent problem \eqref{two-block2}.

Let $y\in \cG$ be the Lagrange multiplier. The objective function and the Lagrangian associated to \eqref{two-block2} are denoted respectively by
\begin{align*}\label{def:PhiL}
        \Phi(x,w)     :=   g(x) + f(w) \text{~~and~~}
      \cL(x,w,y)      :=   \Phi(x,w) + \langle y, Kx - w\rangle.
\end{align*}
We make the following blanket assumptions.
\begin{asmp}\label{asmp-2}
Assume that the set of solutions of \eqref{two-block}, and hence \eqref{two-block2}, is nonempty and,
in addition, there exists $\tilde x\in\text{ri}(\text{dom}(g))$ such that $K\tilde x\in\text{ri}(\text{dom}(f))$.
\end{asmp}

Under Assumption \ref{asmp-2}, it follows from \cite[Corollaries 28.2.2 and 28.3.1]{Rockafellar1970Convex} that
$({x^{\star}}, {w^{\star}}) \in \cH\times \cG$ is a solution of \eqref{two-block2} if and only if there exists an optimal solution ${y^{\star}}\in \cG$  to its dual problem
such that $({x^{\star}}, {w^{\star}}, {y^{\star}})$ is a saddle point of $\cL(x,w,y)$, i.e.,
\ben
\cL({x^{\star}},{w^{\star}},y) \leq \cL({x^{\star}},{w^{\star}},{y^{\star}}) \leq\cL(x,w,{y^{\star}}) \text{~~for all~~} (x, w, y) \in \cH\times \cG \times \cG.
\een
We denote the set of saddle points of $\cL(x,w,y)$ by $\widetilde{\Omega}$, which is nonempty under  Assumption \ref{asmp-2} and is given by
\[\label{def:Omega}
\widetilde{\Omega} = \{({x^{\star}}, {w^{\star}}, {y^{\star}}) \in \R^q\times \R^p \times \R^p ~\mid~ -K^\top {y^{\star}} \in \partial g({x^{\star}}), \,\, {y^{\star}} \in \partial f({w^{\star}}), \,\, K{x^{\star}} = {w^{\star}}\}.
\]
Furthermore, for any $(x,w,y)\in \cH\times \cG \times \cG$, we define
\begin{align}\label{def:G}
  J(x,w,y)
:= \cL(x,w,y)  - \cL({x^{\star}},{w^{\star}}, y)
= \Phi(x,w)+\langle y, Kx-w\rangle  -  \Phi({x^{\star}},{w^{\star}}).
\end{align}

Recall that Algorithm \ref{alg-os} can be viewed as a fixed-point iteration in terms of the operator $T_P$ with $P=\Diag(\theta I_{\cH},\eta\gamma I_{\cG})$. We have the following result.

\begin{lemma}\label{lem-opt}
Let  $\{(x_n,v_n, y_n)\}_{n\in\N}$ be the sequence generated by Algorithm \ref{alg-os} with $A=\partial g$ and $B=\partial f$.
Set $\bz_n=(v_n, u_n)=(v_n, \gamma Kv_n-\tau y_{n-1})$ and $w_n=\prox_{f/\sigma}((Kv_n+Kx_n)-\frac{1}{\gamma}u_n)$.
 Then, for any $y\in \cG$ we have
\begin{align}\label{sec43lem}
\|\bz_{n+1}-\bz^\star(y)\|_M^2+2 \tau J(x_n,w_n,y) \leq  \|\bz_{n}-\bz^\star(y)\|_M^2-\|\bz_{n+1}-\bz_{n}\|_Q^2,
\end{align}
where $\bz^\star(y) :=(v^\star, u^\star(y))^\top $ with $u^\star(y) :=\gamma Kv^\star-\tau y$, $M = P^{-1}$ and $Q$ is defined in \eqref{def:Q} as in Proposition \ref{l:T_A ne}.
\end{lemma}

The proof of this lemma is analogous to that of Proposition \ref{l:T_A ne}. For completeness, we do not omit it but instead relegate it  to Appendix \ref{proof:lem-opt}.
Based on this lemma, we next establish an ${\cal O}(1/N)$ ergodic convergence rate result for Algorithm \ref{alg-os} in solving \eqref{two-block}, or equivalently \eqref{two-block2}.

\begin{theorem}[Sublinear convergence rate]\label{thm22}
Let  $\{(x_n,v_n, y_n)\}_{n\in\N}$ be the sequence generated by Algorithm \ref{alg-os} with $A=\partial g$ and $B=\partial f$.
Then, there exists a constant $C> 0$ such that for any $N\geq 1$ we have $|\Phi({\hat x}_N,{\hat w}_N)-\Phi({x^{\star}},{w^{\star}})| \leq C/N$
and $\|K{\hat x}_N-{\hat w}_N\| \leq   (2C/c) /N$,
where $c>0$ is a constant satisfying  $c \geq 2\|{y^{\star}}\|$, ${\hat x}_N :=\frac{1}{N}\sum_{n=1}^N x_n$ and
${\hat w}_N :=\frac{1}{N} \sum_{n=1}^N w_n$.
\end{theorem}
\begin{proof}
Recall that $Q\succeq 0$. Thus, $\|\bz_{n+1}-\bz_{n}\|_Q^2 \geq 0$ for all $n\geq 1$. For any $y\in \cG$, Lemma \ref{lem-opt} implies that
$2 \tau J(x_n,w_n,y)\leq \|\bz_{n}-\bz^\star(y)\|_M^2 - \|\bz_{n+1}-\bz^\star(y)\|_M^2$, a sum of which over
$n = 1, \ldots, N$ yields
\be
 2\tau\sum\nolimits_{n=1}^N  J(x_n,w_n,y) \leq \|\bz_{1}-\bz^\star(y)\|_M^2 - \|\bz_{N+1}-\bz^\star(y)\|_M^2 \leq
 \|\bz_{1}-\bz^\star(y)\|_M^2.
\label{thm2-1}
\ee
Since $J(x,w,y)$ is convex in $x$ and $w$, it follows from the definition of $({\hat x}_N, {\hat w}_N)$ and Jensen's inequality that
$J\big({\hat x}_N,{\hat w}_N,y\big) \leq {1  \over N} \sum_{n=1}^N  J(x_n,w_n,y)$.
Combining this with \eqref{thm2-1} and considering the definition  of $J(\cdot)$ in \eqref{def:G}, we obtain
\be\label{ineq_them3.2}
\Phi({\hat x}_N,{\hat w}_N)+\langle y, K {\hat x}_N-{\hat w}_N\rangle-\Phi({x^{\star}},{w^{\star}})\leq \|\bz_{1}-\bz^\star(y)\|_M^2 / (2\tau N).
\ee
By taking the maximum of both sides of \eqref{ineq_them3.2} over $\|y\|\leq c$ and defining $C := \sup_{y} \{\|\bz_{1}-\bz^\star(y)\|_M^2: \,  \|y\|\leq c\} / (2\tau) > 0$, we obtain
\be\label{key-result}
\Phi({\hat x}_N,{\hat w}_N)+c\|K {\hat x}_N-{\hat w}_N\|-\Phi({x^{\star}},{w^{\star}})\leq C/N,
\ee
which implies $\Phi({\hat x}_N,{\hat w}_N)-\Phi({x^{\star}},{w^{\star}})\leq C/N$.
Furthermore, since ${\cal L}({x^{\star}},{w^{\star}},{y^{\star}}) \leq {\cal L}({\hat x}_N,{\hat w}_N,{y^{\star}})$, $K{x^{\star}} = {w^{\star}}$ and $\|{y^{\star}}\| \leq c/2$, we have
\begin{eqnarray}\label{jy-5}
\Phi({x^{\star}},{w^{\star}})-\Phi({\hat x}_N,{\hat w}_N)&\leq&\langle {y^{\star}}, K{\hat x}_N-{\hat w}_N\rangle
 \leq (c/2)\|K{\hat x}_N-{\hat w}_N\|,
\end{eqnarray}
which together with \eqref{key-result} implies
\begin{eqnarray*}
c\|K{\hat x}_N-{\hat w}_N\|&\leq& \Phi({x^{\star}},{w^{\star}}) - \Phi({\hat x}_N,{\hat w}_N)+  C / N
\leq  (c/2)\|K{\hat x}_N-{\hat w}_N\|+  C / N.
\end{eqnarray*}
As a result, we derive $\|K{\hat x}_N-{\hat w}_N\|\leq 2C/(cN)$. It then follows from \eqref{jy-5} that
$\Phi({x^{\star}},{w^{\star}})-\Phi({\hat x}_N,{\hat w}_N)\leq C/N$, and thus
$|\Phi({\hat x}_N,{\hat w}_N)-\Phi({x^{\star}},{w^{\star}})| \leq C/N$.
The proof is completed.
\end{proof}

\section{Further discussions}
\label{sec:discussion}
Algorithm \ref{alg-os} requires that $\gamma\|K\|^2 < (2-\theta)(2-\eta)$.
This section investigates the limiting case of the parameters in Algorithm \ref{alg-os}, where $\gamma\|K\|^2=(2-\theta)(2-\eta)$, under the uniform monotonicity assumption of operator $A$. We extend the analysis of operator $P$ in \eqref{def:T} by considering its transformation from a diagonal structure $P=\diag(\theta I_{\cH}, \eta\gamma I_{\cG})$ to a nondiagonal form, leveraging the flexibility provided by Proposition \ref{l:T_A ne}. Furthermore, we explore heuristic strategies for choosing parameters $\theta$ and $\eta$ to improve numerical performance.

\subsection{Limiting case analysis under uniform monotonicity of $A$}
\label{sec: limiting case}
In this subsection, we analyze the limiting case  where   $\gamma\|K\|^2=(2-\theta)(2-\eta)$.
In this case, the operator $Q$ loses its positive definiteness and only remains positive semidefinite. Under this scenario, we cannot deduce $\lim_{n\to\infty}\|\bz_{n+1}-\bz_n\|=0$ from $\sum_{n \in\N}\|\bz_{n+1}-\bz_n\|_Q^2\leq \|\bz_0-\bz^\star\|_M^2$. In the remainder of this subsection, we assume that the operator $A$ is uniformly monotone with modulus $\phi: \R_+ \rightarrow [0,+\infty]$, i.e., it satisfies
\ben
\langle x-y,  u-v\rangle \geq \phi (\|x-y\|), ~~\forall (x,u), (y,v)\in \gra A,
\een
where $\phi$ is strictly increasing and $\phi(0)=0$.
With this extra condition, we analyze convergence of Algorithm \ref{alg-os} under the limiting case:
\be\label{limit-case}
\theta\in(0,2), ~\eta\in(0,2)~~\mbox{and}~\gamma\|K\|^2=(2-\theta)(2-\eta).
\ee

\begin{theorem}\label{th:limit-case}
Let $ T_P$ be defined as in \eqref{def:T} with $P=\diag(\theta I_{\cH}, \eta\gamma I_{\cG})\succ0$ and let operator $A$ be uniformly monotone with modulus $\phi$.
Assume that   $\Fix  T_P \neq \varnothing$ and the condition \eqref{limit-case} is satisfied. Let $\{\bz_n:=(v_n,u_n)^\top\}_{n \in\N}$ be generated by
$\bz_{n+1} = T_{P} (\bz_n)$ with any $\bz_0 \in \cH\times \cG.$
Then the following assertions hold.
\begin{enumerate}[(a)]
\item The sequence $\{v_n\}_{n \in\N}$ converges strongly to {a point   $v^\star \in\zer\left(A+K^*BK\right)$.}
\item The sequence $\{\bz_n\}_{n \in\N}$ converges weakly to a point $\bz^\star\in\Fix T_P$.
\end{enumerate}
\end{theorem}

\begin{proof}
(a). Recall that $P=\diag(\theta I_{\cH}, \eta\gamma I_{\cG})\succ0$, $M=P^{-1} \succ0$, and $\Phi_K$ is defined in \eqref{def:Phi}.
Following the proof of Proposition \ref{l:T_A ne} while incorporating the uniform monotonicity of $A$ instead of monotonicity, we obtain
\ben
 \|T_P(\bz)-T_P(\bar{\bz})\|_M^2 + \|(I-T_P)(\bz)-(I-T_P)(\bar{\bz})\|_Q^2+2\phi(\|x-\bar{x}\|)
  \leq \|\bz-\bar{\bz}\|_M^2,
\een
from which we can easily derive
\be\label{um-z}
 \|\bz_{n+1}-\bz^\star\|_M^2 + \|\bz_{n+1}-\bz_n\|_Q^2+2\phi(\|x_n-x^\star\|)
  \leq \|\bz_n-\bz^\star\|_M^2.
\ee
It then follows that $\{\|\bz_n-\bz^\star\|_M^2\}_{n \in\N}$ is nonincreasing and hence convergent.
Since $\phi$ is nonnegative, strictly increasing, and vanishes only at $0$, it is elementary to derive from \eqref{um-z} that $x_n\rightarrow x^\star\in\zer\left(A+K^*BK\right)$ as $n\rightarrow +\infty$.  That is,  the sequence $\{x_n\}_{n \in\N}$ converges to $x^\star$ strongly.
Combined with the relation $v_{n} =\theta x_{n-1}+(1-\theta)v_{n-1}$, this implies the strong convergence of $\{v_n\}_{n \in\N}$ to $x^\star$ as well.

(b). From \eqref{um-z}, we obtain $\sum_{n=0}^\infty \|\bz_{n+1}-\bz_n\|_Q^2\leq \|\bz_0-\bz^\star\|_M^2$, which implies
$\|\bz_{n+1}-\bz_n\|_Q^2\rightarrow 0$. Combining this with $Q=M^*\left(\Phi_K-P\right)M$ and $M = P^{-1}$, we obtain
\be\label{jy-2}
(\Phi_K-P)M(\bz_{n+1}-\bz_n)=\left[\ba{ccc}(2-\theta)I&~&\gamma K^*\\ \gamma K&~&\gamma(2-\eta)I\ea\right]\left(\ba{l}\frac{1}{\theta}(v_{n+1}-v_n)\\ \frac{1}{\eta\gamma}(u_{n+1}-u_n)\ea \right)\rightarrow 0.
\ee
Relation \eqref{jy-2} yields $\frac{\gamma}{\theta} K(v_{n+1}-v_n)+\frac{2-\eta}{\eta}(u_{n+1}-u_n)\rightarrow 0$.
Since  $v_{n+1}-v_n\rightarrow 0$, it follows that $u_{n+1}-u_n\rightarrow 0$.
Therefore, we obtain $(I-T_P)(\bz_n)\rightarrow 0$ strongly.
Note that $T_P$ is nonexpansive in the $M$-norm due to Proposition \ref{l:T_A ne}.
 Let $\bz'$ be any  weak cluster point of the bounded sequence $\{\bz_n\}_{n \in\N}$.
According to the Browder's demiclosedness principle \cite[Theorem 4.27]{Bauschke2011book}, which also holds for $T_P$   (see Remark \ref{rem-tight} (iii)),
we deduce that $I-T_P$ is demiclosed.
Thus, $(I-T_P)(\bz_n)\rightarrow 0$ implies that $\bz'\in \Fix T_P$.
By applying \cite[Theorem 5.5]{Bauschke2011book}, we conclude that the sequence $\{\bz_n\}_{n \in\N}$ converges weakly to an element in $\Fix T_P$.
\end{proof}

\subsection{A nondiagonal choice of $P$}
As established in Lemma \ref{lem-a}, Algorithm \ref{alg-os} corresponds to a fixed-point iteration of the operator $T_P$ defined in \eqref{def:T} with
$P=\Diag(\theta I_{\cH},\eta\gamma I_{\cG})$. While the diagonal structure of $P$ is commonly used, the extended firmly nonexpansive property of $T_P$ given in \eqref{eq:T_A_ne} -- crucial for the convergence results in  Theorem \ref{th:main} -- depends only on the condition $0\prec P\prec \Phi_K$, as demonstrated by Proposition \ref{l:T_A ne} and Theorem \ref{th:main}.
This broader condition offers significant flexibility in selecting $P$ to ensure convergence of the iteration.
Motivated by this observation, in this subsection we investigate the following $2$-by-$2$ but nondiagonal choice for $P$:
\be\label{def:P}
P=\left[\ba{cc}\theta I~&~ \gamma K^* \\  \gamma K ~&~\eta\gamma I\ea\right].
\ee
Analogous to the parameter set $\Theta$ defined in \eqref{def:parameter-set}, we introduce $\widehat{\Theta}$ as follows:
\ben
\widehat{\Theta}:= \{(\theta,\eta, \gamma)~|~ \theta \in (0,2), \, \eta  \in (0,2), \, \gamma\in (0,\infty), \, \gamma \|K\|^2<\theta\eta\}.
\een
It is easy to observe that $0\prec P\prec \Phi_K$  for any $(\theta,\eta,\gamma)\in\widehat{\Theta}$. When the nondiagonal operator $P$ defined in \eqref{def:P} is adopted in the fixed-point iteration of $\bz_{n+1} = T_P (\bz_n)$, where $T_P$ is defined in \eqref{def:T}, we obtain the following modified PDSA. We emphasize that Algorithm 5.1 is a special case of the fixed-point iteration $\bz_{n+1} = T_P (\bz_n)$ for $n\geq 0$, and the results in Lemma \ref{l:fixed points} hold as well for the choice of $P$ in \eqref{def:P}.

\vskip5mm
\hrule\vskip2mm
\begin{algo}
[A modified PDSA]\label{alg-os-imp}
{~}\vskip 1pt {\rm
\begin{description}
\item[{\em Step 0.}] Choose $(\theta,\eta,\gamma)\in \widehat{\Theta}$, $\tau>0$, and set $\sigma=\gamma/\tau$. Initialize  $v_{0}\in \cH$, $u_0\in \cG$ and $n=0$.
\item[{\em Step 1.}] Compute $x_{n}$, $w_{n}$, $v_{n+1}$ and $u_{n+1}$ sequentially as follows:
\ben
x_n&=&J_{\tau A}(v_n - \tau K^*(\sigma Kv_n-u_{n}/\tau)), \\
w_n&=&J_{B/\sigma} \left(Kv_n+Kx_n-u_n/\gamma\right), \\
{v_{n+1}} &=& v_{n}+\theta(x_n-v_{n})+\gamma K^*(w_n-Kx_n), \\
u_{n+1}&=&u_{n}+\gamma K(x_n-v_{n}) +\eta\gamma \left(w_n-Kx_n\right).
\een
\item[{\em Step 2.}] Set $n\leftarrow n + 1$ and return to Step 1.
  \end{description}
}
\end{algo}
\vskip1mm\hrule\vskip5mm

For the operator $P$ defined in \eqref{def:P},  when the triple of parameters $(\theta,\eta,\gamma)$ lies within $\widehat{\Theta}$, the strict inequalities $0\prec P\prec \Phi_K$  hold. Consequently, the convergence results presented in Theorem \ref{th:main} are also applicable to Algorithm \ref{alg-os-imp}.
Compared to Algorithm \ref{alg-os}, this modified PDSA introduces additional  terms $\gamma K^*(w_n-Kx_n)$ and $\gamma K(x_n-v_{n})$ to the updates of $v_{n+1}$ and $u_{n+1}$, respectively, requiring one additional evaluation of the adjoint operator $K^*$ per iteration.   However, as demonstrated in   Figure \ref{Fig:LASSO-P}, this nondiagonal structure $P$ enables the algorithm converges faster, requiring fewer iterations.

\subsection{Adaptive adjustment heuristics for $\theta$ and $\eta$}
The convergence-guaranteeing condition presented in \eqref{def:parameter-set} is much broader than the conditions in CP-PDHG  \cite{Chambolle2011A} and GRPDA \cite{ChY2020Golden,ChY2022relaxed}. Specifically, both algorithms require  $\tau\sigma\|K\|^2$ be upper bounded. For CP-PDHG, the upper bound is $1$. For GRPDA, the upper bound is $(\sqrt{5}+1)/2$ in \cite{ChY2020Golden} and $4(\sqrt{2}-1)$ in \cite{ChY2022relaxed}. In contrast,  as indicated from the definition of $\Theta$ in \eqref{def:parameter-set}, this upper bound can be as large as $4$.
This flexible and improved condition is advantageous as it enables the joint tuning of the relaxation parameter $\eta$, the convex combination parameter $\theta$, and the step size parameter $\gamma$. Experimental results demonstrate that moderate over-relaxation paired with a relatively small $\theta$ is most efficient. Notably,  once $\eta$ and $\theta$ are determined, the optimal performance is mostly achieved by using the largest feasible step size $\gamma$.

As established in Proposition \ref{l:T_A ne}, the operator $T_P$ is
nonexpansive with respect to the $M$-norm. While this choice of norm preserves the standard convergence proof for the fixed-point algorithm, the selection of $P$ can significantly impact the numerical performance of Algorithm \ref{alg-os}.
To illustrate this point,  let us examine the technical details. Recall that for Algorithm \ref{alg-os} we have $M=P^{-1}=\diag(\theta^{-1}I_{\cH}, (\eta\gamma)^{-1}I_{\cG})$ and
\be\label{jy-3}
\|\bz_{n+1}-\bz_{n}\|_M^2=  \|v_{n+1}-v_n\|^2 / \theta +  \|u_{n+1}-u_n\|^2/(\eta\gamma).
\ee
This quantity converges to $0$ monotonically as the algorithm proceeds.
However,  for given $\gamma$ and $\eta$, if $\theta$ is very small, the weight of the first term on the right-hand-side of \eqref{jy-3}
would be much larger than that of the second. As a result, a small decrease in $\|v_{n+1}-v_n\|$ will readily imply a decent decrease in the norm  $\|\bz_{n+1}-\bz_{n}\|_M$ even if $\|u_{n+1}-u_n\|$ does not decrease much.
The same line of reasoning holds when, for a given $\gamma$, $\eta$ is much smaller than $\theta$.
In practice, adaptively selecting $\theta$ and $\eta$ such that the primal and dual residuals are balanced appears to be of utmost importance for attaining favorable numerical performance.

We will now outline our strategy for dynamically adjusting $\theta$ and $\eta$. Although we have not developed a comprehensive theory for this strategy, we have witnessed its effectiveness in diverse applications, as illustrated in Figure \ref{Fig:LASSO}.
Define the variable residuals  by $\texttt{vres}_n:=\|v_{n}-v_{n-1}\|$ and $\texttt{ures}_n:=\|u_{n}-u_{n-1}\|$, and set $r_n := \texttt{vres}_n/\texttt{ures}_n$. Then, we can dynamically adjust the values of $(\theta, \eta)$ to ensure that $\texttt{vres}_n$ and $\texttt{ures}_n$ are balanced. Specifically, we have implemented the following adaptive rule for determining $(\theta_n, \eta_n)$, which represents the value of $(\theta, \eta)$ at the $n$th iteration:
\bi
\item (i) If  $r_n \leq 4/5$, then set $\theta_n=\min\{5\theta_{n-1}/4,~ \overline{\theta}\}$ and $\eta_n=\varepsilon(2-\gamma \|K\|^2/(2-\theta_n))$.
\item (ii) If $r_n \geq 5/4$, then set $\eta_n=\min\{5\eta_{n-1}/4, ~\overline{\eta}\}$ and $\theta_n=\varepsilon(2-\gamma\|K\|^2/(2-\eta_n))$.
\item (iii) If $r_n \in (4/5, ~5/4)$, then set $(\theta_n,~ \eta_n)=(\theta_{n-1},~ \eta_{n-1})$.
\ei
Here, $\varepsilon\in (0,1)$, and $\overline{\theta}$ and $\overline{\eta}$ are upper bounds on $\theta$ and $\eta$, respectively.
In our experiments, we  set $\varepsilon=0.99$ and $\overline{\theta} = \overline{\eta}= 1.99$.
Note that an adaptive PDHG algorithm was proposed in \cite{Goldstein2015Adaptive}, which adaptively adjusts the step sizes $\tau$ and $\eta$ by balancing the primal and dual residuals. In contrast, our strategy tunes the convex combination parameter $\theta$ and the relaxation parameter $\eta$ by balancing the variable residuals $\texttt{vres}_n$ and $\texttt{ures}_n$.  While both approaches choose parameters adaptively, they tune different parameters using different rules.

\section{Numerical results}
\label{sec-experiments}
In this section, numerical experiments for three different cases are presented to assess the performance of the proposed algorithms: a benchmark example by He and Yuan \cite{He2014On}, a discrete total-variation regularized image denoising and inpainting problem, and a bilinear saddle point problem covering minimax matrix-game and LASSO problems. The experiments were carried out on a 64-bit Windows system with an Intel(R) Core(TM) i5-4590 processor (3.30 GHz) and 8 GB RAM, and all the results are reproducible by specifying the \texttt{seed} of the random number generator in the code accessible at \url{https://github.com/xkchang-opt/PDSA-ec}.

\subsection{He and Yuan's example}
\label{sec:he}
He and Yuan \cite{He2014On, He22On} demonstrated through numerous examples that the Arrow-Hurwicz splitting (AHS) method \cite{Uzawa58} may not converge, even when the step sizes are fixed as small constants. Here, we consider one such example in the form of the following primal-dual saddle point problem:
\begin{equation}\label{he_example}
\min_{x\in\mathbb{R}}\max_{y\in\mathbb{R}} ~~xy,
\end{equation}
which is apparently a special case of \eqref{saddle-point} with $g=0$, $f^*=0$ and $K=I$.
For this example, the AHS method does not converge, see \cite{He2014On} and also Figure \ref{Fig:example} below.
It is straightforward to observe that problem \eqref{he_example} has the unique saddle point $(x^*,y^*)=(0,0)$. When using the CP-PDHG method, the solution point $(0,0)$ can be obtained after just one iteration.
When we apply Algorithm \ref{alg-os} to the problem \eqref{he_example}, we obtain the following iterative scheme:
\ben
\left\{\ba{rcl}
v_{n} &=&\theta x_{n-1}+(1-\theta)v_{n-1},\\
x_n &=& v_{n} -\tau y_{n-1},\\
y_{n}&=&y_{n-1}+\eta\sigma x_{n}+\sigma\theta (x_n-v_n),
\ea\right.\een
where $\tau, \sigma>0$ and $\tau\sigma<(2-\theta)(2-\eta)$.

In Figure \ref{Fig:example}(a), we display the convergence behavior of AHS, CP-PDHG, PDAc, and Algorithm \ref{alg-os}, using the initial point $(x_0,y_0)=(1,1)$.
For PDAc, we set $\psi = 1.6$, $\tau = 1$, and $\sigma = 1.6$. For Algorithm \ref{alg-os}, we set $\theta = 1$ and $\eta = 0.99$.
It can be observed that Algorithm \ref{alg-os} converges to the solution point $(0,0)$ for the special example \eqref{he_example}.
By simple deduction, we also note that when $\theta=\eta=\tau=\gamma = 1$ for Algorithm \ref{alg-os}, its iteration point is the same as that of CP-PDHG. That is, starting from $(x_0,y_0)=(1,1)$, we have $(x_1,y_1)=(0,0)$, despite their different iteration schemes.

\begin{figure}[htp]
\centering
\subfigure[Results obtained from different methods.]{
\includegraphics[width=0.45\textwidth]{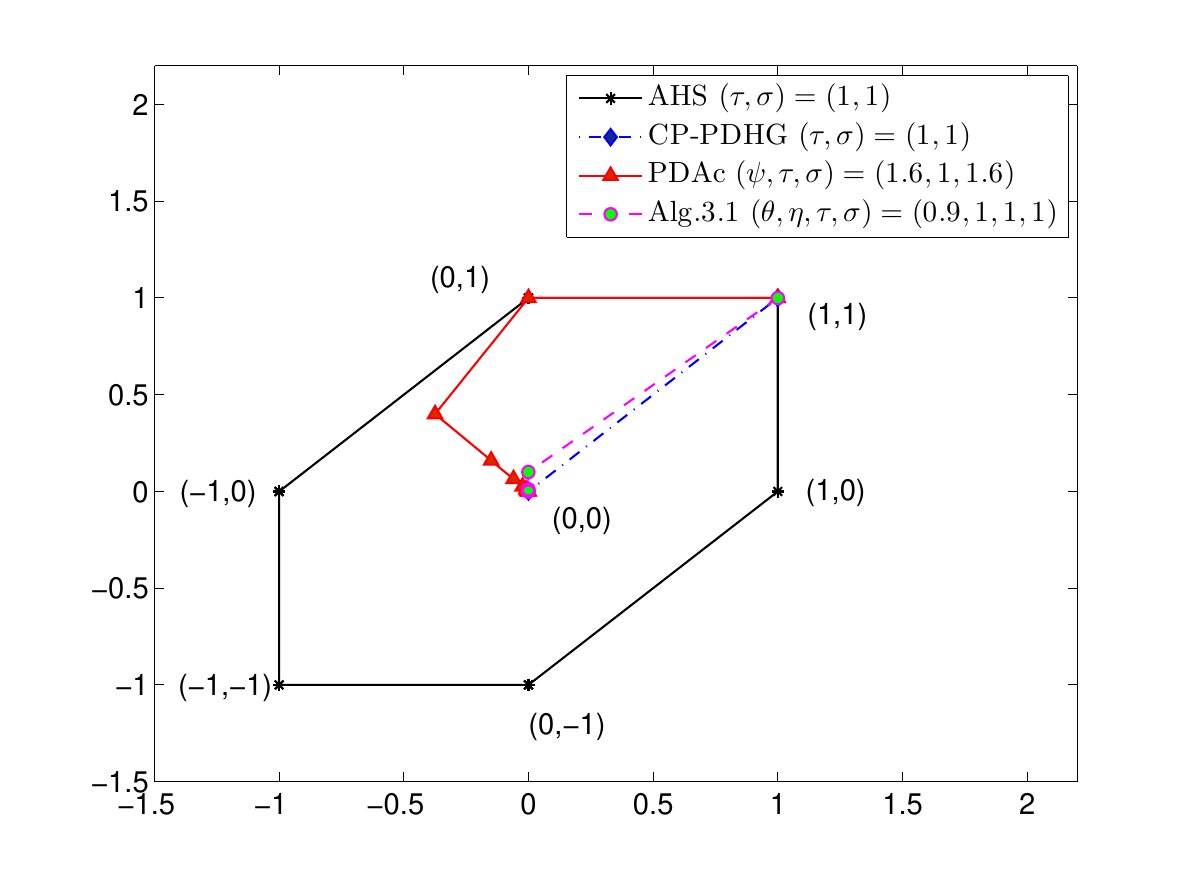}}
\subfigure[Results of Algorithm  \ref{alg-os-imp} for different combinations of $(\theta,\eta,\gamma)$.]{
\includegraphics[width=0.45\textwidth]{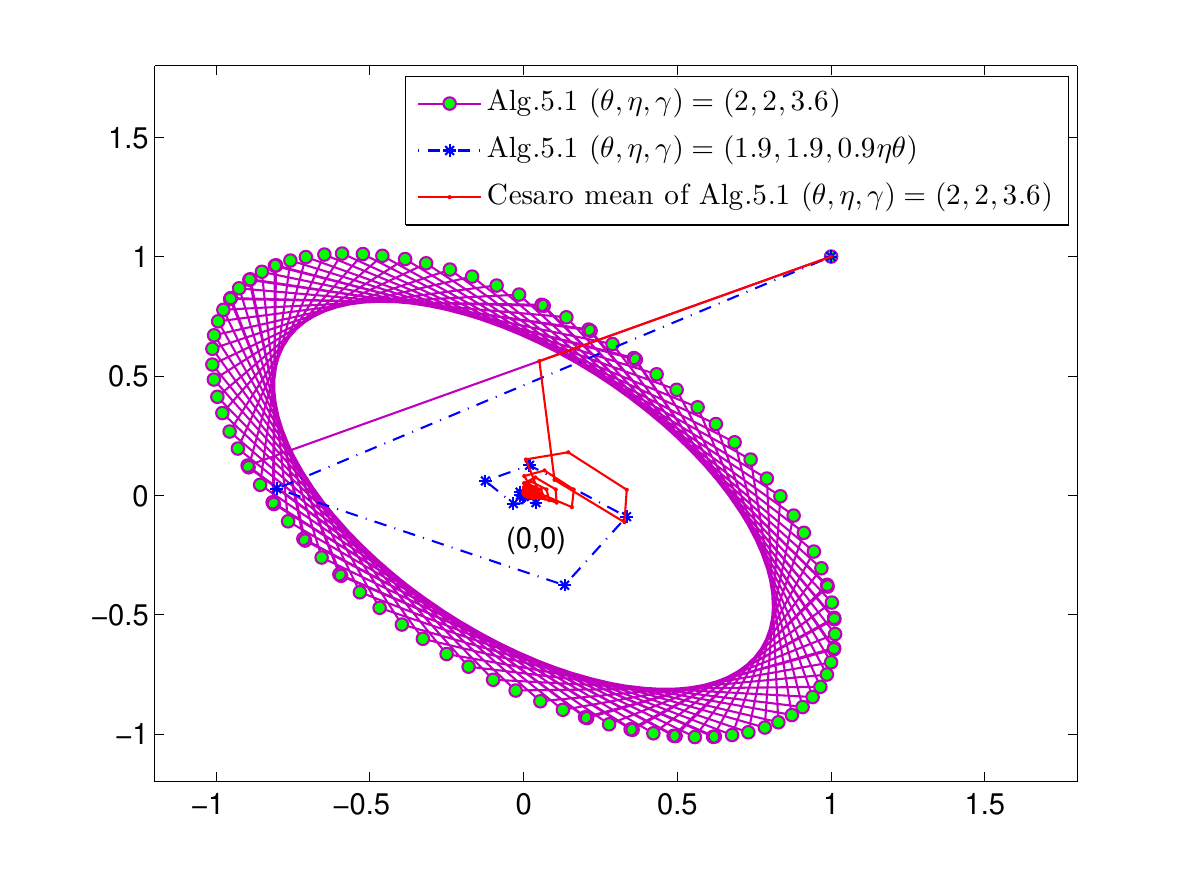}}
\caption{Results for the example \eqref{he_example} from He and Yuan \cite{He2014On}. }
\label{Fig:example}
\end{figure}

Recall that $\Phi_K$ is defined in \eqref{def:Phi}, and Algorithm \ref{alg-os-imp} corresponds to a fixed-point iteration of $T_P$ with $P$
defined in \eqref{def:P}.
We further tested Algorithm \ref{alg-os-imp} with $\tau = \sigma$ for solving problem \eqref{he_example}. The algorithm can be implemented as follows:
\ben
\left\{\ba{l}
x_{n}=v_{n}-\tau y_{n-1},\\
w_n= 0,\\
y_{n}=y_{n-1}+ \sigma [(\eta+\theta-1-\gamma) x_n+(1-\theta)v_{n}-(\eta-\gamma) w_n],\\
v_{n+1} =v_n+ (\theta-\gamma) x_{n}-\theta v_{n}+\gamma w_n.
\ea\right.
\een
The initial point was set to $(v_0,y_0)=(1,1)$, and two different sets of parameters were tested:
\begin{itemize}
    \item (i) $(\theta,\eta,\gamma)=(2,2,3.6)$, which satisfies $0\prec P=\Phi_K$.
    \item (ii)$(\theta,\eta,\gamma)=(1.9,1.9,0.9\theta\eta)$, which satisfies $0\prec P\prec\Phi_K$.
\end{itemize}
As is evident from Figure \ref{Fig:example}(b), under case (i) we have $Q=0$, and the sequence $\{(x_n,y_n)\}$ generated by Algorithm \ref{alg-os-imp} does not converge. In fact, we  conducted tests on Algorithm \ref{alg-os-imp} using different values of $\gamma$ within the interval $(0,4)$ and observed that, the iterates revolve outside an ellipse and display periodic behavior. However, under the condition in case (ii), Algorithm \ref{alg-os-imp} converges to the solution $(0, 0)$. Moreover,
 we plot the Cesaro mean of the sequence $\{(x_n,y_n)\}$ under case (i), which demonstrates converge to the solution $(0, 0)$.

Next, we provide an explanation regarding the aforementioned periodic behavior of Algorithm \ref{alg-os-imp} when $P=\Phi_K$.
Recall that Algorithm \ref{alg-os-imp} is a realization of the iteration
$\bz_{n+1}=T_P(\bz_n)$ with $T_P$  and $P$ defined in  \eqref{def:T} and \eqref{def:P}, respectively.
Recall that $\Phi_K$ is defined in \eqref{def:Phi}. Then, the condition $P=\Phi_K$ is equivalent to setting $\theta=\eta=2$ in \eqref{def:P}.
For problem \eqref{he_example}, $T_P$ is linear and can be represented by the following matrix:
\ben
T_P=\left(\ba{cc}
\gamma^2-3\gamma+1, \; & 2-\gamma\\
\gamma^2-2\gamma,   &  1-\gamma\ea\right).
\een
The two eigenvalues of $T_P$ are given by $\lambda(T_P):= {\gamma^2/2}-2\gamma+1\pm({\gamma/2}-1)\sqrt{\gamma(\gamma-4)}$.
Then, it is elementary to verify  that $|\lambda(T_P)|=1$ for any $\gamma\in(0,4)$.
This provides a reasonable explanation on the periodic behavior of the iterates generated by Algorithm \ref{alg-os-imp} with $P=\Phi_K$.
Furthermore, this also certifies the sharpness of the condition  $0\prec P\prec\Phi_K$ in ensuring the convergence of the fixed-point iteration $\bz_{n+1}=T_P(\bz_n)$.

\subsection{Discrete total-variation regularized image denoising problem}\label{sec6.2}
In this section, we consider image denoising problem.  As usual,  we denote images by vectors, instead of matrices.
Given a noisy image $f_0$ and a regularization parameter $\alpha > 0$, the
discrete total-variation (TV) regularized image denoising problem \cite{ROF1992}  can be formulated as
\be\label{dtv-pro}
\min_{x}\frac{1}{2}\|x-f_{0}\|_{2}^{2}+\alpha\|D x\|_{1},
\ee
where $D$ represents the discrete finite difference operator, see, e.g., \cite{WYYZ08}.
Problem \eqref{dtv-pro} can be equivalently recast as the saddle point problem \eqref{saddle-point}, with
\ben
g(x)=\frac{1}{2}\|x-f_0\|_{2}^{2}, ~~K=D~~\mbox{and}~~ f^*(y)=\iota_{\{\|y\|_\infty\leq\alpha\}}(y).
\een
Here, $\iota_C$ represents the indicator function of the set $C$. That is, $\iota_C(x) = 0$ if $x\in C$, and $\iota_C(x)=\infty$ otherwise.
The proximal operator of $\tau g$ is given by $\prox_{\tau g}(x) = (x+\tau f_{0})/(1+\tau)$ and $\prox_{f^*}$ is given by the Euclidean projection on to the $\ell_{\infty}$-norm ball $\|y\|_\infty\leq\alpha$.
In this experiment, we measure the optimality using the primal-dual gap function as defined in \cite{Bredies2015preconditioned,Chambolle2011A}, which in this context appears as
\begin{align*}
G(x,y)=\frac{1}{2}\|x - f_{0}\|_{2}^{2}+\alpha\|D x\|_{1}+\iota_{\{\|y\|_{\infty}\leq\alpha\}}(y)+\frac{1}{2}\|D^\top y\|_{2}^{2}+\langle D^\top y,f_{0}\rangle.
\end{align*}
To make the results independent of image size, in our experiments, we use the normalized primal-dual gap $ G(x,y)/(N_xN_y)$, where $N_x$ and $N_y$ are the dimensions of the image. The iteration process continues until the normalized primal-dual gap is less than a given $\epsilon>0$ or the number of iterations $n$ reaches $n_{\max}$, with $n_{\max}$ being the maximum number of allowed iterations. In this section, we set $n_{\max}=5\times10^{4}$.

We conducted tests on the Butterfly image (with dimensions $512\times768$) and the Barbara image (with dimensions $512\times512$). The original images were corrupted by Gaussian noise with a zero mean and a variance of $0.05$. The regularization parameters were chosen as $\alpha = 0.2$ and $\alpha = 0.5$, respectively.
For CP-PDHG, we set $\tau=\sigma$ and $\tau\sigma=1/L^2$, as in \cite{Bredies2015preconditioned}, where we set $L = \sqrt{8}$ as in \cite{Chambolle2011A}. We also tested the relaxed CP-PDHG given in \eqref{PDSA} with $\rho=1.5$.
For PDAc, we set $\psi = 1.6$, $\tau=\sigma$ and $\tau\sigma=1.6/L^2$.

To determine the parameters for Algorithm \ref{alg-os}, we initially conducted tests on various combinations of the relaxation parameter $\eta$ and the convex combination parameter $\theta$, both within the open interval $(0,2)$. Given that the function $g(x) = \frac{1}{2}\|x - f_0\|_2^2$ is strongly convex, its gradient operator $\nabla g$ is uniform monotone. Consequently, in accordance with the theoretical results presented in Section \ref{sec: limiting case}, we set $\gamma= (2 - \theta)(2 - \eta)/L^{2}$ and $\tau = \sigma = \sqrt{\gamma}$, ensuring the algorithm's convergence.
For this specific test, we performed computations on the Butterfly image, with the regularization parameter $\alpha$ set to $0.2$. Table \ref{table-th-eta} summarizes the iteration counts required for Algorithm \ref{alg-os} to satisfy the criterion $G(x_n,y_n) < 10^{-6}$ under different $\theta$ and $\eta$ values.

The experimental results reveal two key observations: First, when $\theta$ is fixed, the algorithm performs optimally with a moderate degree of over-relaxation, i.e., $\eta > 1$, but the value of $\eta$ should not be excessively large. Second, for a fixed $\eta$, the algorithm demonstrates a preference for choosing the convex combination parameter $\theta$ in the range of approximately $1/7$ to $1/5$. We further validated these findings by testing additional images with varying regularization parameters and stopping criteria, and the results remained consistent across different scenarios.
Based on these observations, we set $\theta = 1/5$, $\eta = 7/6$, $\gamma= (2 - \theta)(2 - \eta)/L^{2}= 1.5/L^2$, and $\tau=\sigma=\sqrt{\gamma}=\sqrt{1.5}/L$ or $(\tau, \sigma) = (1/L, 1.5/L)$ for Algorithm \ref{alg-os} in this TV denosing problem.  For cases where the uniform monotonicity of $\partial g$ is not guaranteed, we adjust the parameter $\theta$ to $\theta = 0.99/5$ and remain the choice $\gamma= 1.5/L^2$, ensuring that the convergence condition $\gamma L^2 < (2-\theta)(2-\eta)$ is met.

\begin{table}[htpb]
\caption{Iteration counts required for Algorithm \ref{alg-os} to achieve $ G(x_n,y_n) < 10^{-6}$ under different $\theta$ and $\eta$ for the TV image denoising problem (\ref{dtv-pro}). Here, $ \gamma =  (2 - \theta)(2 - \eta) / L^2$, $\tau = \sigma = \sqrt{\gamma}$, the regularization parameter is set as $\alpha = 0.2$, and the tested image is Butterfly. In the table, ``---" represents that the algorithm reached $2000$ iterations without satisfying the stopping condition.}\label{table-th-eta}\label{table-th-eta}
\center
\small
\begin{tabular}{c|cccccccccc}
\hline
\multirow{2}{*}{$\eta$}&\multicolumn{9}{c}{$\theta$}\\
&1/7 &  1/6 &  1/5  & 1/4  & 1/3  & 1/2 & 1& 7/6 &9/6 \\
\cline{1-10}
1/2 & 1871	& 1874	& 1881	& 1894	& 1924	& ---	& ---	& ---	& --- \\
1 & 1177	& 1178	& 1180	& 1186	&1202	& 1247	& 1489	& 1623	& ---\\
7/6 & \textbf{1115}	& \textbf{1115}	& \textbf{1117}	& 1122	& 1136	& 1178	& 1405	& 1530	& 1957	\\
9/6 & 1136	& 1136	& 1138	& 1143	& 1157	& 1198	& 1425	& 1550	& 1979	\\
11/6 & 1625	&  1628	& 1632	& 1640	& 1662	& 1718	& ---	& ---	& ---	\\
\hline
\end{tabular}
\end{table}

\begin{table}[htpb]
\caption{Numerical results for the TV image denoising problem (\ref{dtv-pro}). The CPU time consumed (denoted as Time) is measured in seconds, and the iteration number is denoted as Iter.}\label{table1}
\center
\small
\begin{tabular}{cccccccccccc}
\hline
\multirow{4}{*}{$\alpha$}&\multirow{4}{*}{$\epsilon$} & \multicolumn{10}{c}{Butterfly image }\\
\cline{3-12}
&&\multicolumn{2}{c}{CP-PDHG}&\multicolumn{2}{c}{CP-PDHG}& \multicolumn{2}{c}{PDAc} &\multicolumn{2}{c}{Alg. \ref{alg-os}} &\multicolumn{2}{c}{Alg. \ref{alg-os}} \\
&&\multicolumn{2}{c}{}&\multicolumn{2}{c}{(relaxed)}& \multicolumn{2}{c}{} &\multicolumn{2}{c}{($\tau=\sigma=\sqrt{1.5}/L$)} &\multicolumn{2}{c}{($\tau=1/L, \sigma=1.5/L$)} \\
\cline{3-12}
&  & Iter &Time & Iter &Time &Iter &Time&Iter &Time&Iter &Time\\
\hline
\multirow{3}{*}{$0.2$}& $10^{-5}$ & 337 &15.5  & 227 & 11.4 &281 &12.6&267 &12.1 & 226 & 9.8\\
&$10^{-6}$	& 1478 &67.9  & 995   &    48.5  & 1229 &55.1   & 1129  & 54.1  & 951   &    43.3  \\
&$10^{-7}$	&6747 & 316.3  &4517 &215.6    &5653 &260.0 & 5144   &    232.2 & 4286 &  197.8 \\
\hline
\multirow{3}{*}{$0.5$}& $10^{-4}$&619 &27.2  &412  & 17.7  &489 &20.6&428  & 19.5 & 351        & 14.9  \\
&$10^{-5}$	& 2679 &118.5   &1792 & 83.7  &2117  &92.9  &1894  & 80.8   & 1549  &68.5  \\
&$10^{-6}$	& 11284 &  474.5  & 7523 &  327.6 & 8919  &  395.5 & 7839 & 335.8 & 6552  & 282.2 \\
\hline
\hline
\multirow{4}{*}{$\alpha$}&\multirow{4}{*}{$\epsilon$}& \multicolumn{10}{c}{Barbara image}\\
\cline{3-12}
&&\multicolumn{2}{c}{CP-PDHG}&\multicolumn{2}{c}{CP-PDHG}& \multicolumn{2}{c}{PDAc} &\multicolumn{2}{c}{Alg. \ref{alg-os}} &\multicolumn{2}{c}{Alg. \ref{alg-os}}\\
&&\multicolumn{2}{c}{}&\multicolumn{2}{c}{(relaxed)}& \multicolumn{2}{c}{} &\multicolumn{2}{c}{($\tau=\sigma=\sqrt{1.5}/L$)} &\multicolumn{2}{c}{($\tau=1/L, \sigma=1.5/L$)} \\
\cline{3-12}
&  & Iter &Time & Iter &Time &Iter &Time& Iter &Time &Iter &Time\\
\hline
\multirow{3}{*}{$0.2$}& $10^{-5}$  &331 &10.3  & 222 & 6.9  &276 &8.2 &256  & 7.7 & 220 &        6.5\\
&$10^{-6}$  &1405 &43.7   & 954  &  30.0  &1181 &   34.8  &1077  & 31.9  & 901 & 26.2\\
&$10^{-7}$  &6455 &199.3  & 4313 & 135.2  &5329 &   158.9& 4804 & 150.9 &4085  &  119.5 \\
\hline
\multirow{3}{*}{$0.5$}& $10^{-4}$&529 &15.7  & 349 & 10.8  &419&11.9&370 & 10.7& 305 & 8.9\\
&$10^{-5}$	&2341&70.5  & 1520  &  45.9 &1849 &54.8&1602   & 45.2& 1311  & 38.9 \\
&$10^{-6}$ &8589 &259.1  & 5684 &  171.5   &6794  & 202.2&5928 & 171.1 &4962 & 146.3\\
\hline
\end{tabular}
\end{table}

\begin{figure}[htbp]
\centering
\includegraphics[width=0.23\textwidth]{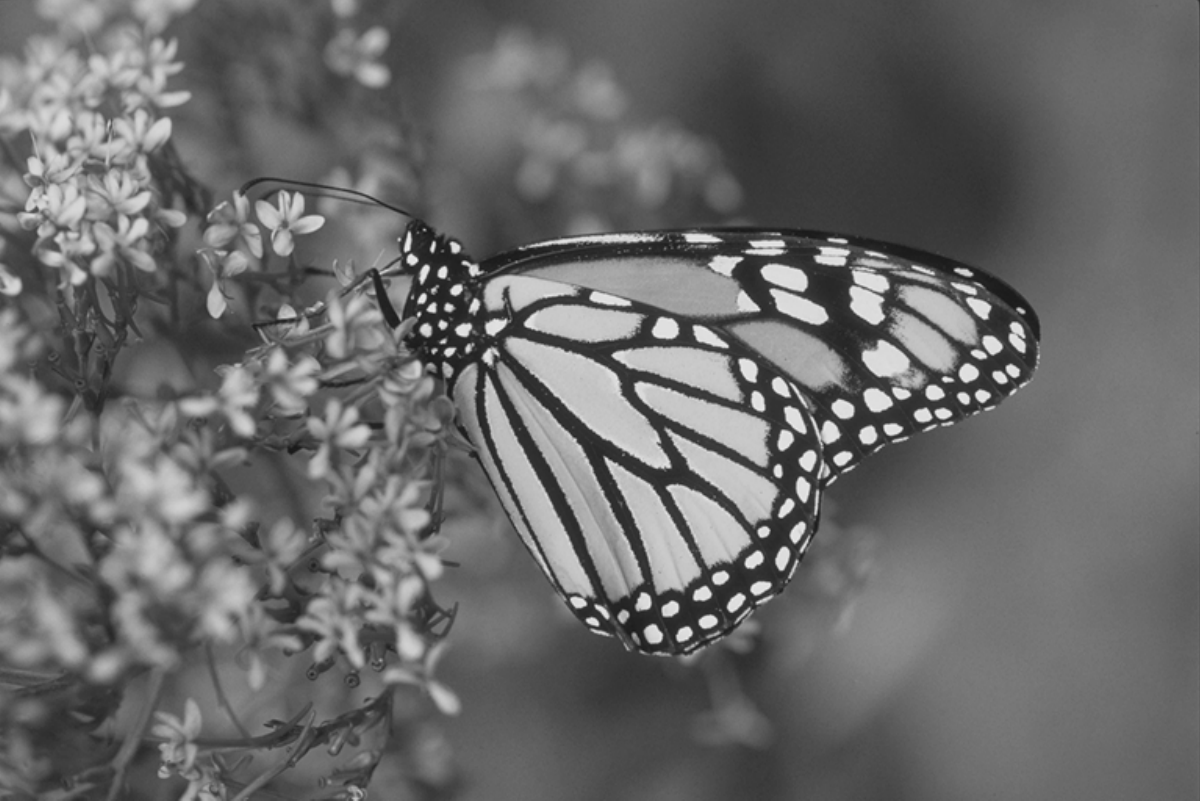}
\includegraphics[width=0.23\textwidth]{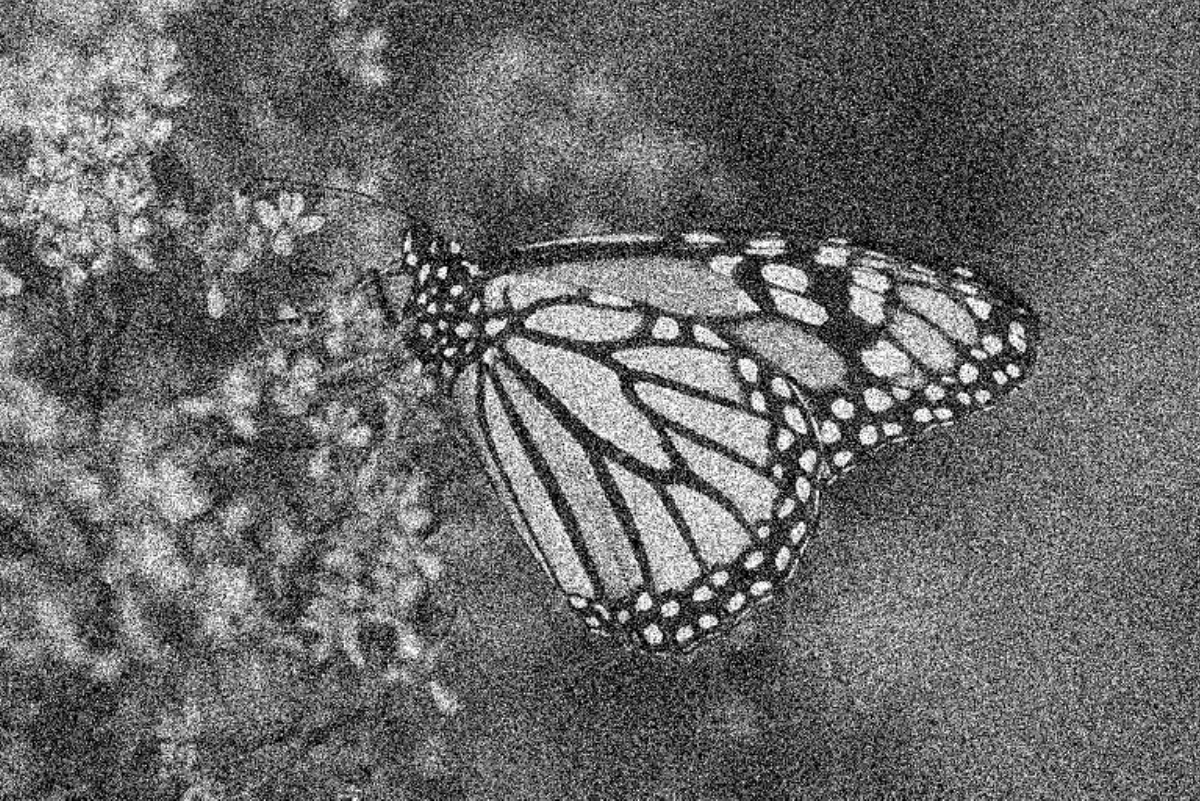}
\includegraphics[width=0.23\textwidth]{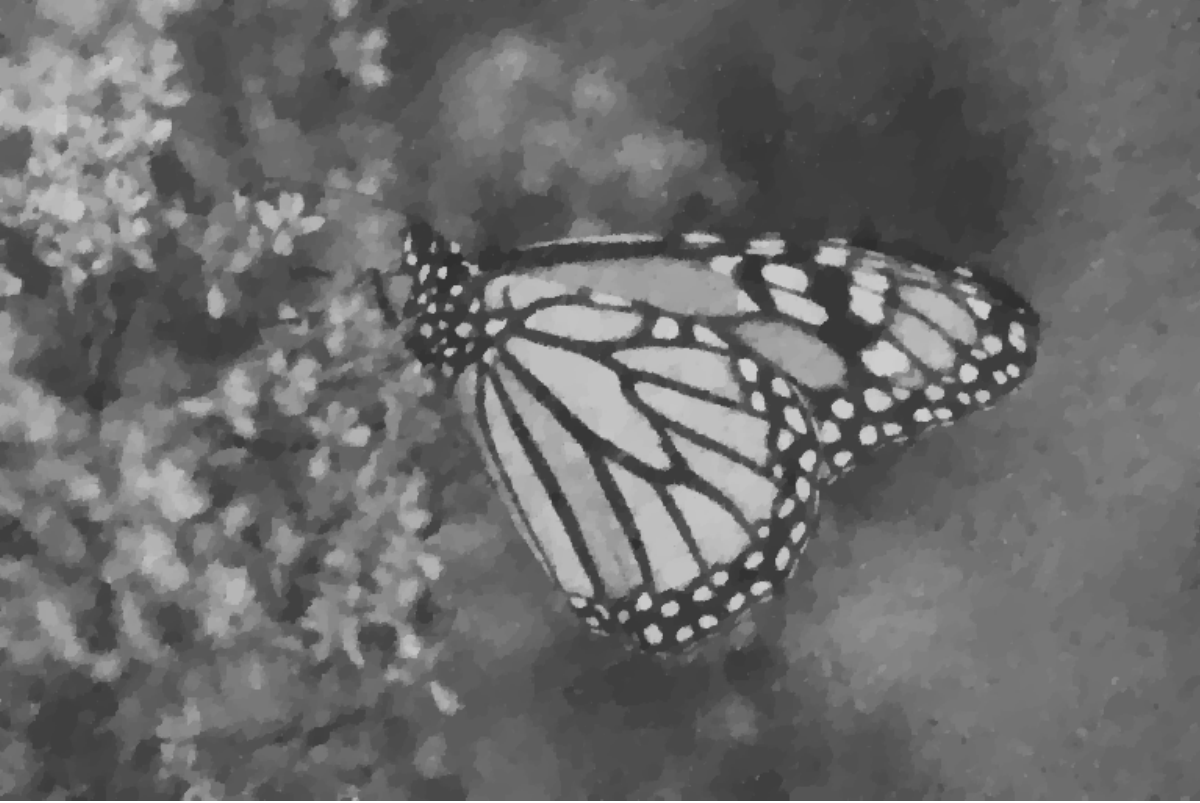}
\includegraphics[width=0.23\textwidth]{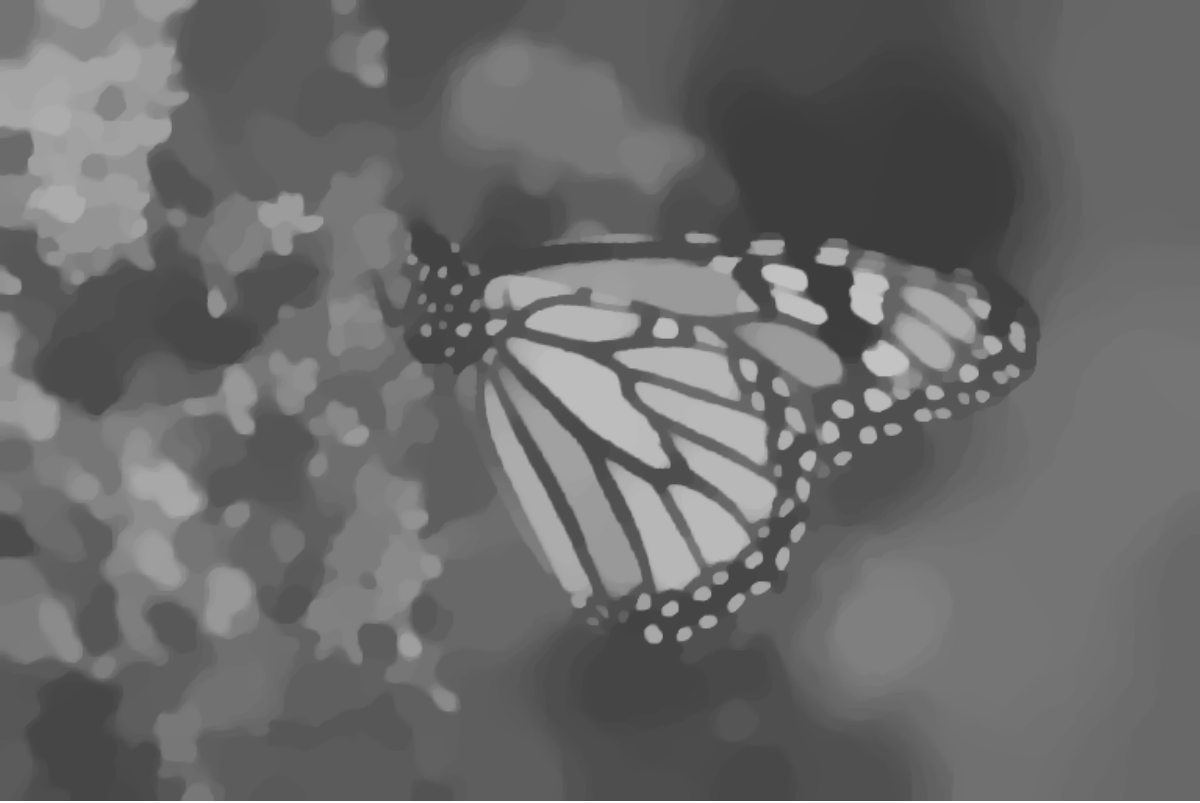}
\includegraphics[width=0.23\textwidth]{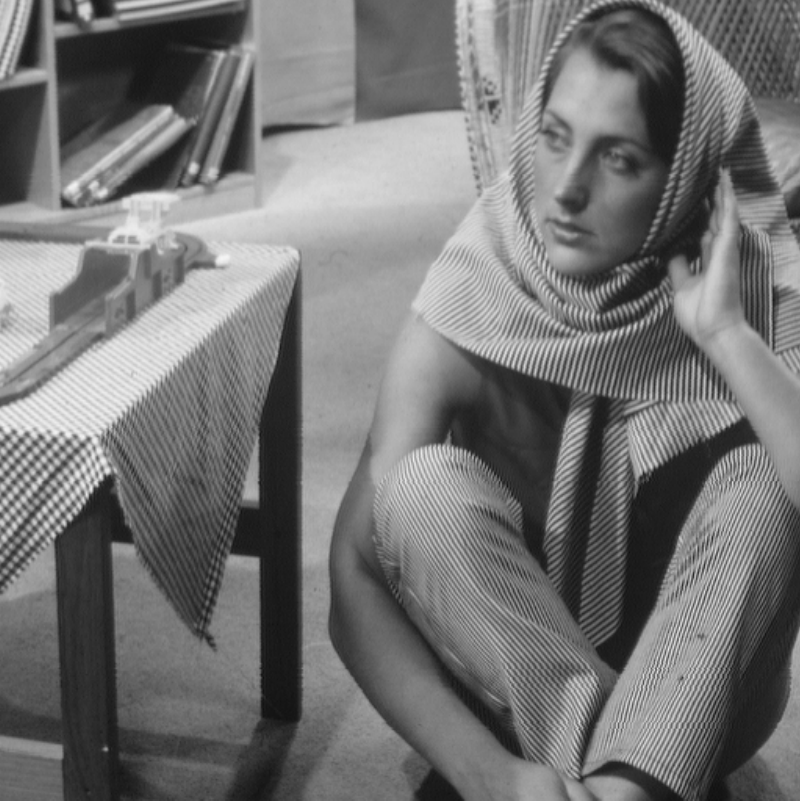}
\includegraphics[width=0.23\textwidth]{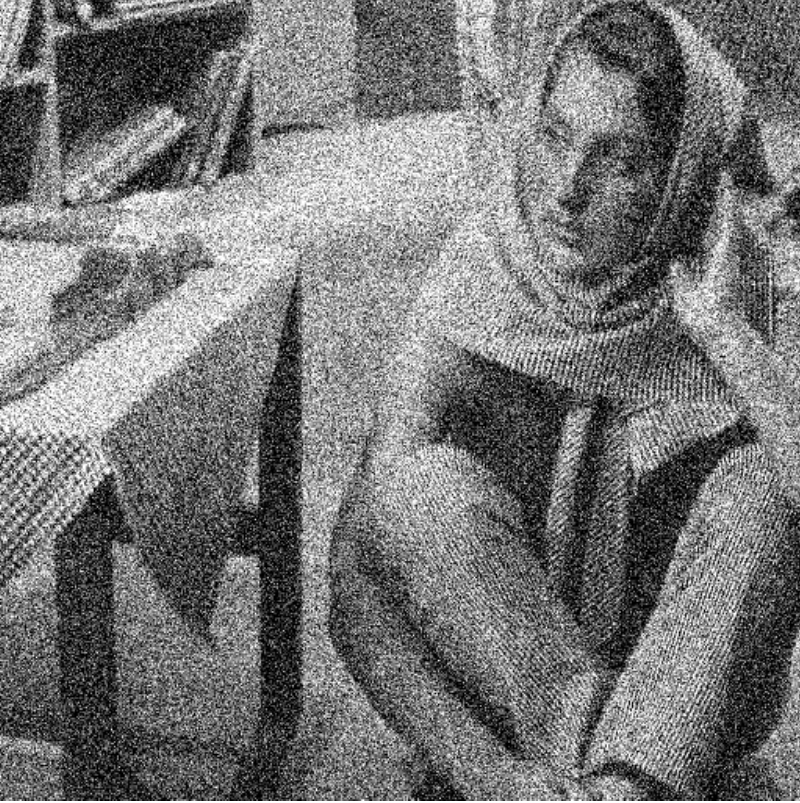}
\includegraphics[width=0.23\textwidth]{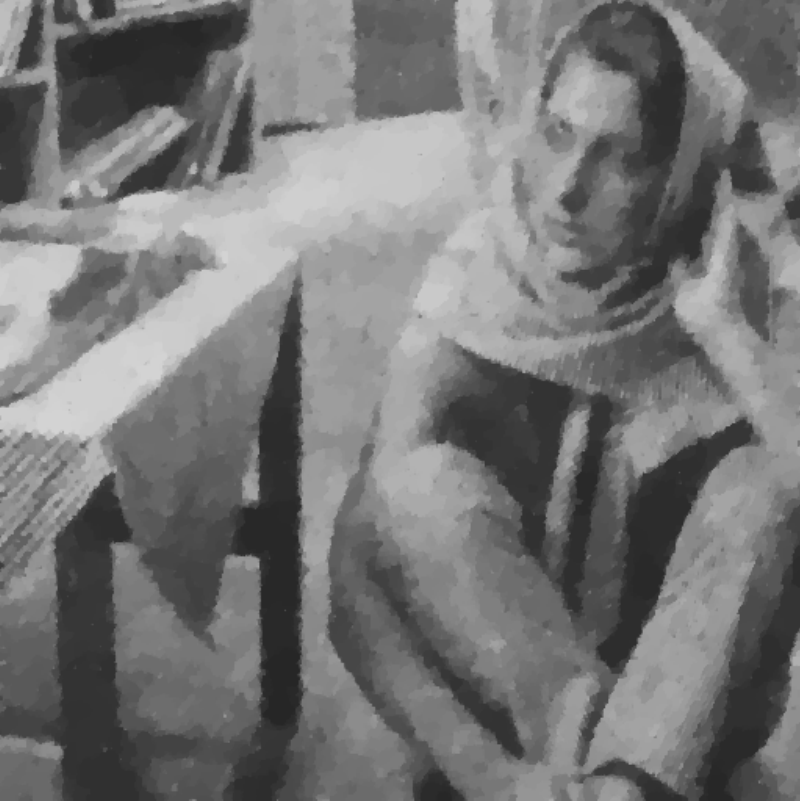}
\includegraphics[width=0.23\textwidth]{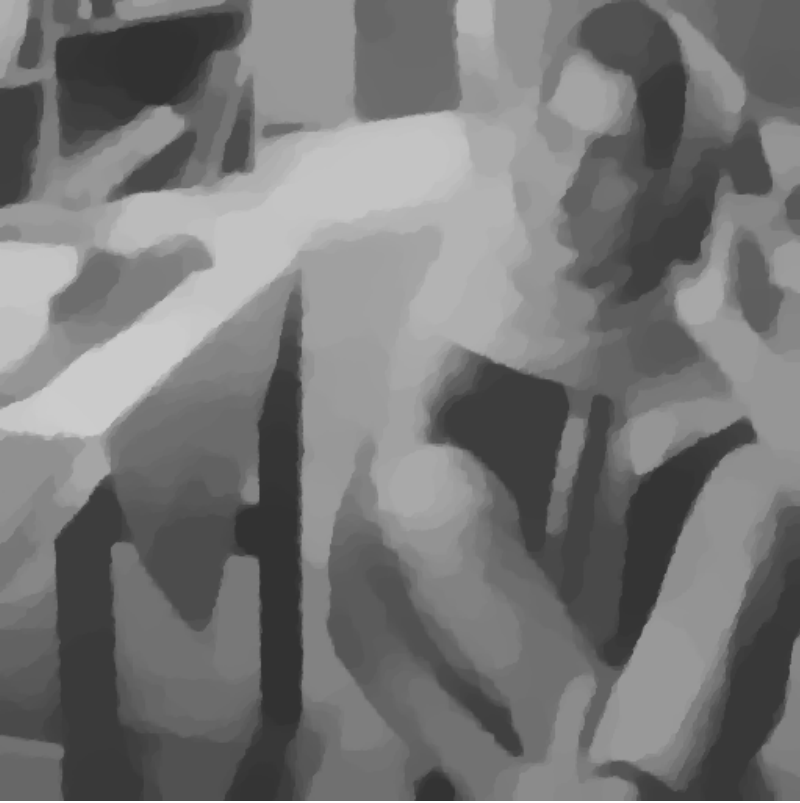}
\caption{The first column shows the original images, while the second column presents the input noisy images.
The final two columns, in turn, display the output images reconstructed by Algorithm \ref{alg-os} with
$\alpha$ values set to $0.2$ and $0.5$ in \eqref{dtv-pro}, respectively.}
\label{Fig:TV}
\end{figure}

\begin{figure}[htbp]
\centering
\subfigure[$\alpha = 0.2$]{
\includegraphics[width=0.43\textwidth]{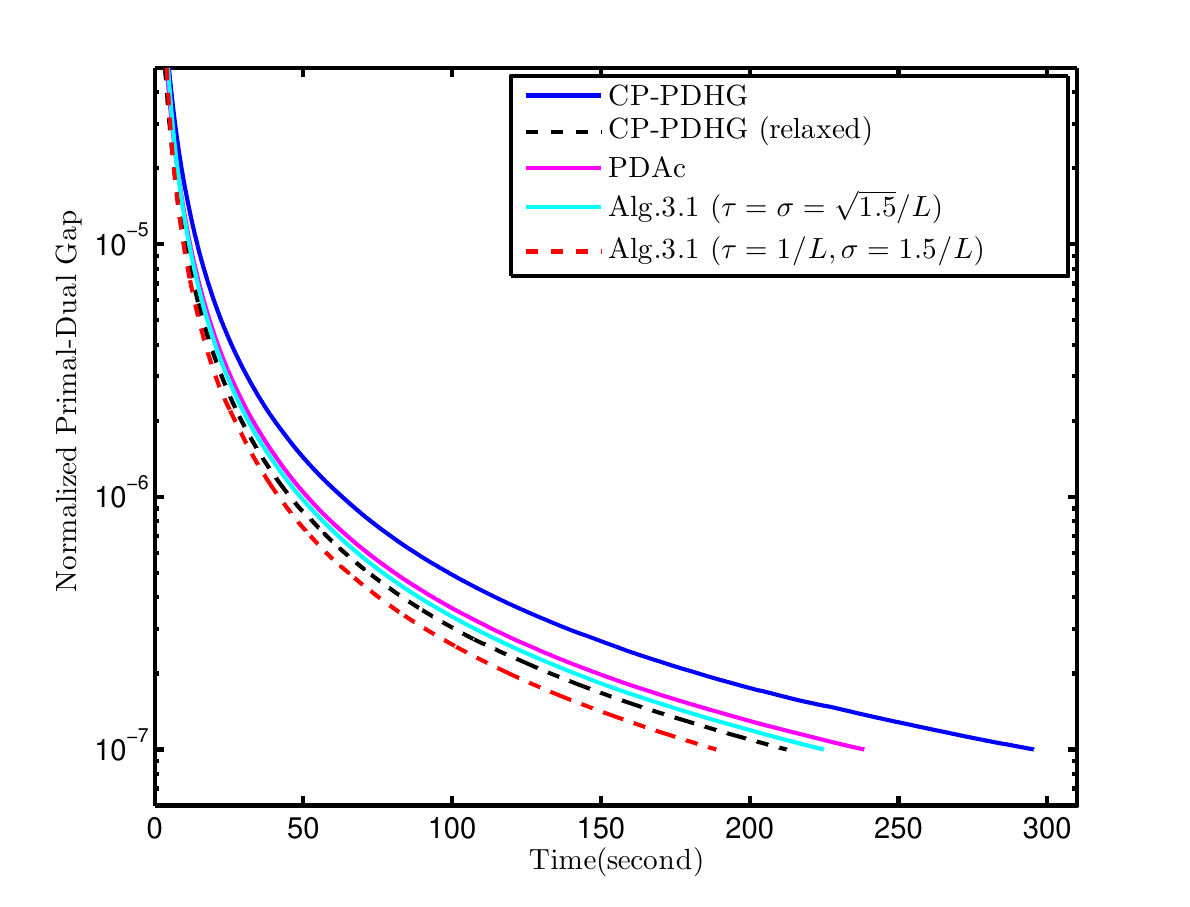}}
\subfigure[$\alpha = 0.5$]{
\includegraphics[width=0.43\textwidth]{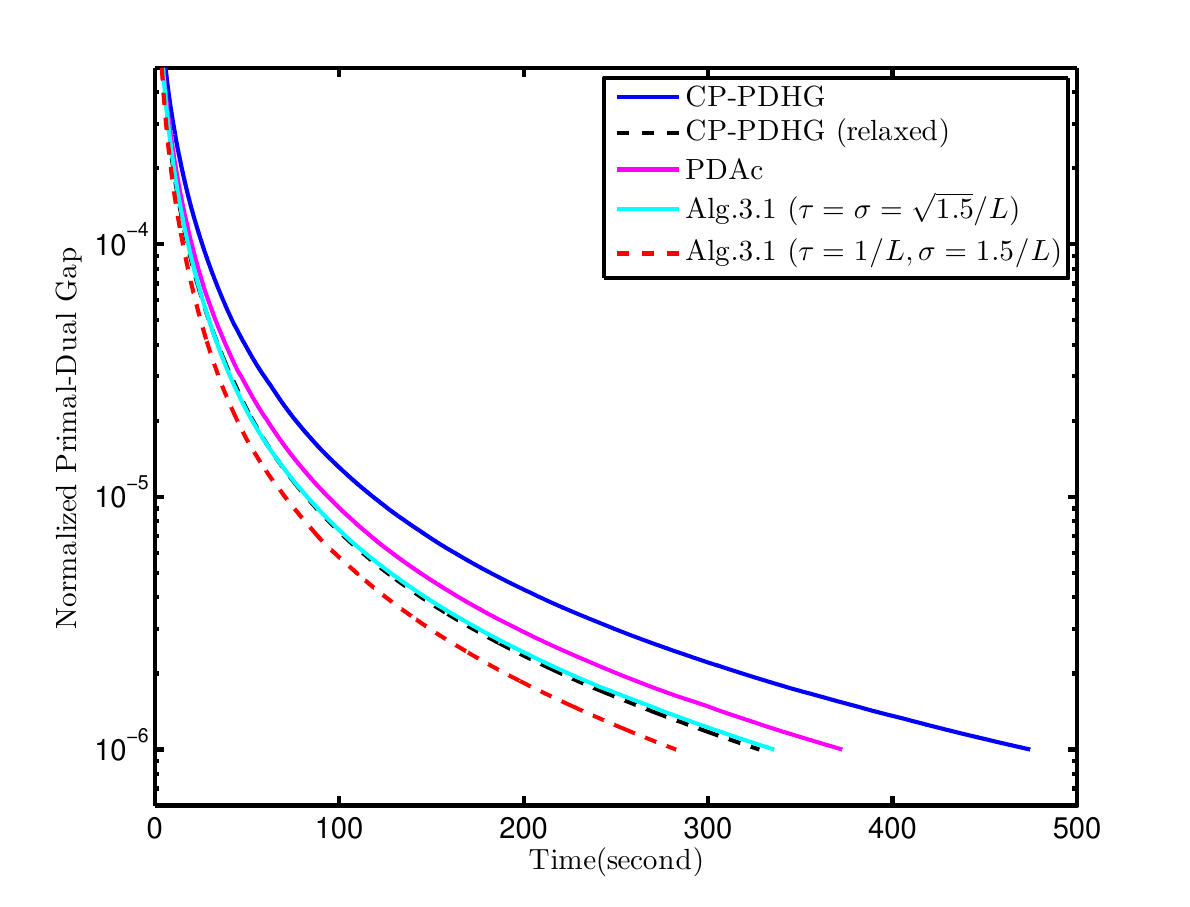}}
\caption{Plots of the normalized primal-dual gap against CPU time for the Butterfly image, with different regularization parameters $\alpha$. The results for the Barbara image are similar.}
\label{Fig:TV-result}
\end{figure}

Table \ref{table1} compiles the numerical results regarding the number of iterations (Iter) and CPU time (Time, measured in seconds) for the image denoising problem, obtained from CP-PDHG, PDAc, and Algorithm \ref{alg-os}. Figure \ref{Fig:TV} showcases the test images, noisy images, and denoised images corresponding to different regularization parameters. Figure \ref{Fig:TV-result} depicts the comparison results of the normalized primal-dual gap plotted against CPU time for the Butterfly image. The corresponding results for the Barbara image are similar and are thus omitted.
As can be observed from Table \ref{table1} and Figure \ref{Fig:TV-result}, Algorithm \ref{alg-os} and PDAc outperform CP-PDHG in terms of the number of iterations and CPU time. Moreover, relaxed CP-PDHG and Algorithm \ref{alg-os} with $\tau=1/L$ and $\sigma=1.5/L$ show slightly better performance compared to others.

Finally, we tested Algorithm \ref{alg-os} and CP-PDHG with the adaptive step sizes, proposed by Goldstein \textit{et al.} \cite{Goldstein2015Adaptive}. Figure \ref{Fig:TV-result-adap} depicts the comparison results against the iteration numbers for the Butterfly image. We see that, the adaptive step sizes is very efficient for both methods.

\begin{figure}[htbp]
\centering
\subfigure[$\alpha = 0.2$]{
\includegraphics[width=0.43\textwidth]{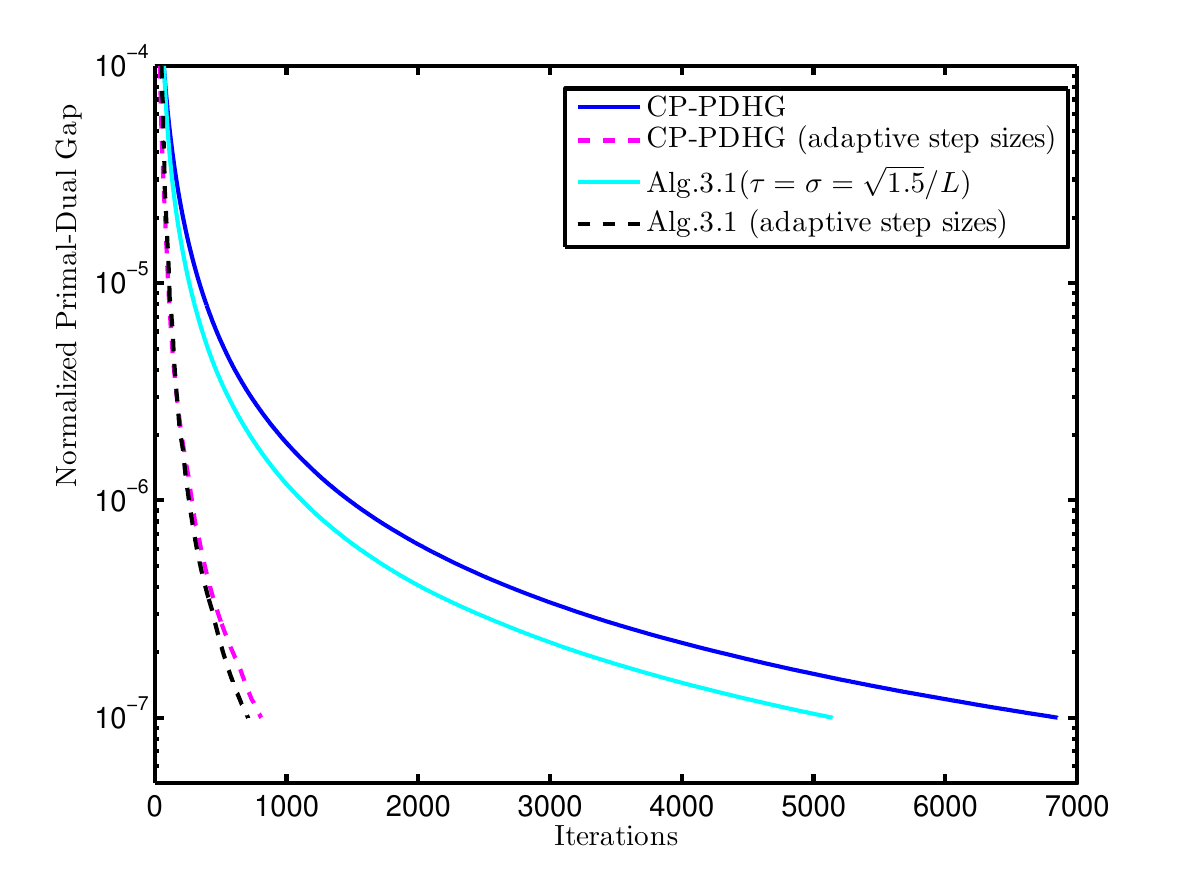}}
\subfigure[$\alpha = 0.5$]{
\includegraphics[width=0.43\textwidth]{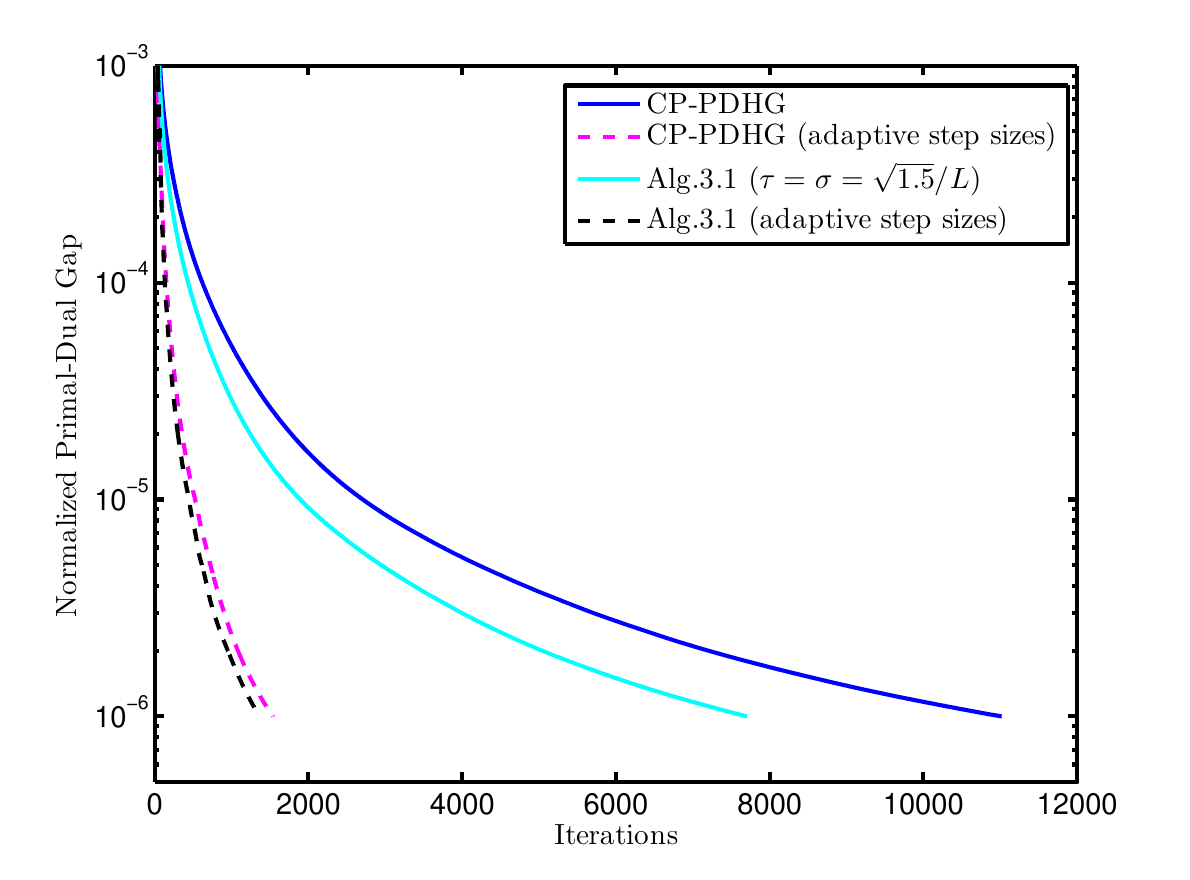}}
\caption{Plots of the normalized primal-dual gap against the iteration numbers for the Butterfly image, with different regularization parameters $\alpha$. The results for the Barbara image are similar.}
\label{Fig:TV-result-adap}
\end{figure}

\subsection{TV image inpainting problem}
Image inpainting refers to the task of filling in missing or damaged parts of an image to make it appear complete and natural.
Let $f_0$ denote a given image in which certain pixel values are missing (by default, these missing values are represented as $0$) and that is contaminated by some random noise. Below, we will refer to $f_0$ as the noisy image. Furthermore, as before, all images are represented by vectors, instead of matrices.
The TV image inpainting problem  \cite{He2012Convergence,Ma2025} can be formulated as:
\be\label{dtv-inp}
\min_{x}~~F(x):=\frac{\lambda}{2}\|Hx-f_0\|_2^2+\|D x\|_1,
\ee
where $H$ is a mask operator,  and $\lambda > 0$ is a regularization parameter.
Specifically, the mask operator $H$ is a diagonal matrix, whose diagonal elements are set to $1$ for the positions corresponding to existing pixels and to $0$ for those corresponding to absent pixel values.
Apparently,  problem \eqref{dtv-inp} represents a special case of \eqref{two-block} with
$g(x)=\frac{\lambda}{2}\|Hx-f_0\|_{2}^{2}$, $f(y)=\|y\|_1$ and $K=D$. As a result, it can be solved using the algorithms discussed/developed in this paper.

In this experiment, we set $\lambda = 100$ as in \cite{Ma2025} and consider two test images: Boats ($512 \times 512$) and Pepper ($512 \times 512$). The noisy input images are generated in accordance with \cite{He2012Convergence,Ma2025}.
For the Boats image, the mask operator is generated using a character mask, causing approximately 15\% of the pixels in the image to be lost. Regarding the Pepper image, the mask operator is generated as described in \cite{He2012Convergence}. Specifically, only one row out of every eight rows is retained, resulting in about 87\% of the pixels being missing. Subsequently, Gaussian noise with a mean of 0 and a standard deviation of 0.02 is added to both images.

We applied the CP-PDHG, PDAc, and Algorithm \ref{alg-os} to solve problem \eqref{dtv-inp}. The parameter settings for these algorithms are as follows: For CP-PDHG, $\tau\sigma = 1/L^2$ and $\rho=1.5$; for PDAc, $\psi = 1.6$ and $\tau\sigma=1.6/L^2$; for Algorithm \ref{alg-os}, $\theta = 0.99/5$, $\eta = 7/6$, and $\gamma = 1.5/L^2$. Here, $L$ denotes the spectral norm of $K$ and is chosen to be $\sqrt{8}$ as in \cite{Chambolle2011A}. In this section, for all the algorithms under test, we set $\tau = r^2\sigma$ with $r = 0.4$.
All of the algorithms initiate their iterations with  $x_0 = f_0$ and $y_0 = D x_0$. The iteration process is terminated when the relative function value residual
$e_{\text{obj}}(x_n):= {|F(x_n)-F_{\text{opt}}|/|F_{\text{opt}}|}<\epsilon$ for given $\epsilon>0$. Here, $F_{\text{opt}}$ represents the optimal objective function value of problem \eqref{dtv-inp}. As the value of $F_{\text{opt}}$ is not known a priori, we execute the CP-PDHG algorithm for a sufficiently large number of iterations to acquire the ground-truth solution $x^\star$. Subsequently, we set $F_{\text{opt}} = F(x^\star)$.
The quality of the inpainted images is evaluated using the Signal-to-Noise Ratio (SNR). The SNR of the recovered image $x$ is defined as
$\mbox{SNR}(x) := 20\log_{10} {(\|\bar{x}\| / \|x - \bar{x}\|)}$,
where $\bar{x}$ represents the original, unblemished image.

\begin{table}[htpb]
\caption{Comparison results of CP-PDHG, PDAc and Algorithm \ref{alg-os} for image inpainting problems under different $\epsilon$ values.
The consumed number of iterations, CPU time and SNR of recovered images are presented.}\label{table2}
\center
\begin{tabular}{cccc}
\hline
\multirow{2}{*}{Image}&\multirow{2}{*}{Method} & $\epsilon=1\times10^{-2}$ & $\epsilon=3\times10^{-3}$ \\
\cline{3-4}
&&Iter/Time/SNR & Iter/Time/SNR \\
\hline
Boats &  CP-PDHG &  346/21.5/22.10& 684/46.9/22.39\\
&  CP-PDHG (relaxed)&  268/18.6/22.33 & 552/41.4/22.42\\
&  PDAc & 299/18.1/22.26 &  578/36.3/22.42\\
&  Alg. \ref{alg-os} ($\tau=\sigma=\sqrt{1.5}/L$) & 268/18.4/22.20 &  541/39.3/22.38\\
&  Alg. \ref{alg-os} ($\tau=1/L, \sigma=1.5/L$) & 222/15.3/22.27 & 489/36.3/22.35\\
\hline
\hline
\multirow{2}{*}{Image}&\multirow{2}{*}{Method} & $\epsilon=10^{-1}$ & $\epsilon=10^{-2}$ \\
\cline{3-4}
&&Iter/Time/SNR & Iter/Time/SNR \\
\hline
Pepper &  CP-PDHG & 173/11.2/15.60  & 533/34.2/15.58\\
&  CP-PDHG (relaxed) & 147/10.5/15.58 & 416/29.9/15.56\\
&  PDAc & 155/9.6/15.68& 464/29.1/15.72\\
&  Alg. \ref{alg-os} ($\tau=\sigma=\sqrt{1.5}/L$) &  141/10.2/15.71 & 402/28.6/15.66\\
&  Alg. \ref{alg-os} ($\tau=1/L, \sigma=1.5/L$)&  146/10.4/15.63 & 377/27.7/15.62\\
\hline
\end{tabular}
\end{table}

\begin{figure}[htbp]
\centering
\includegraphics[width=0.48\textwidth]{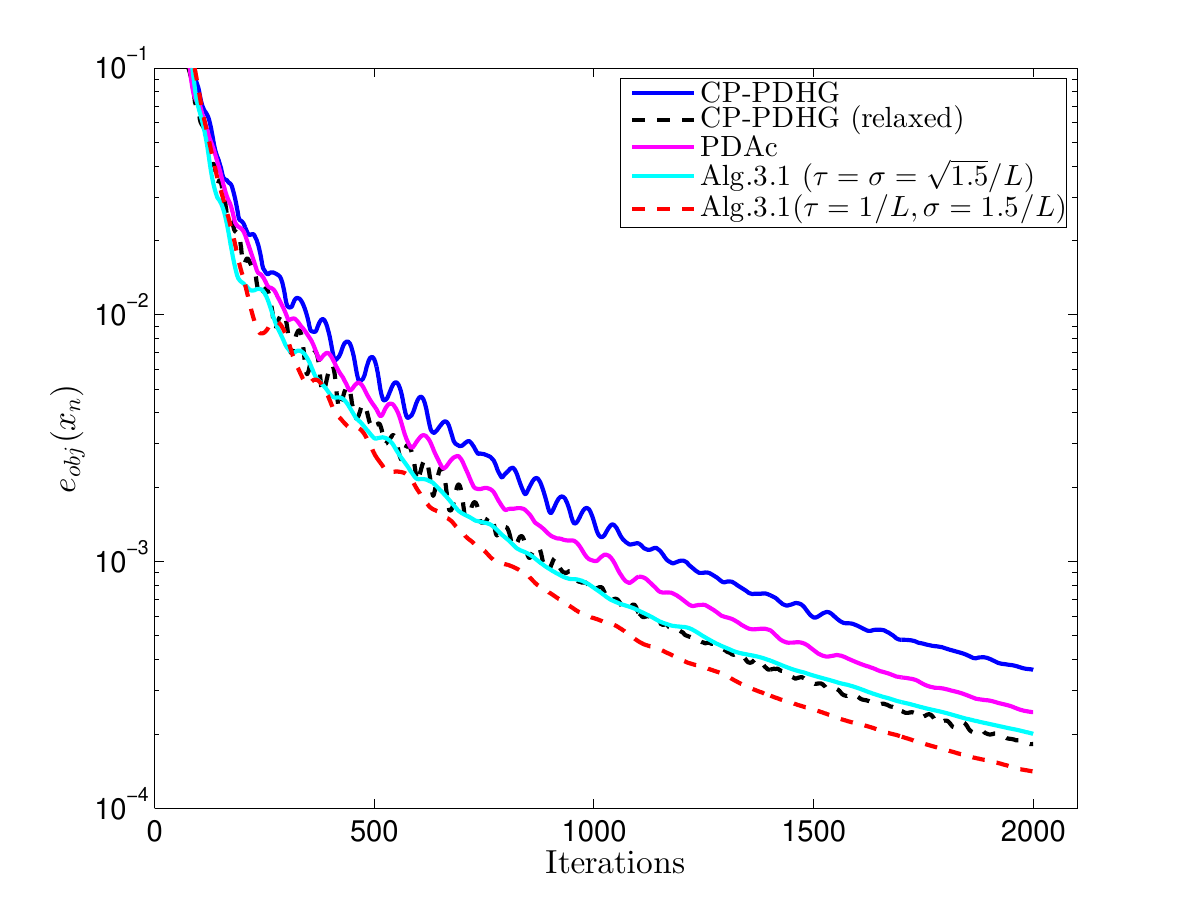}
\includegraphics[width=0.48\textwidth]{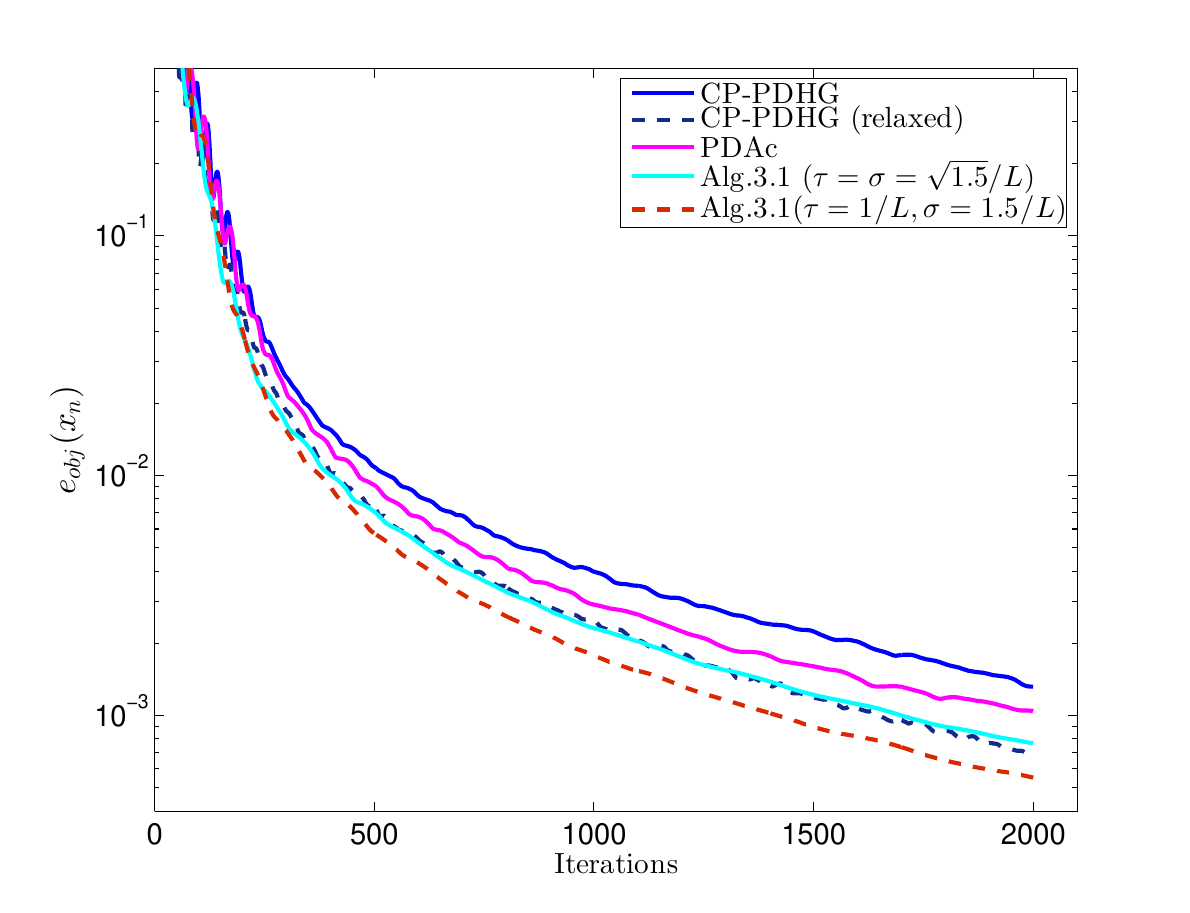}
\caption{Plots presenting the relative error of objective function values with respect to the iteration numbers. The left-hand plot is for the Boats image, and the right-hand plot is for the Pepper image.}
\label{Fig:TV-obj}
\end{figure}

\begin{figure}[htbp]
\centering
\includegraphics[width=0.3\textwidth]{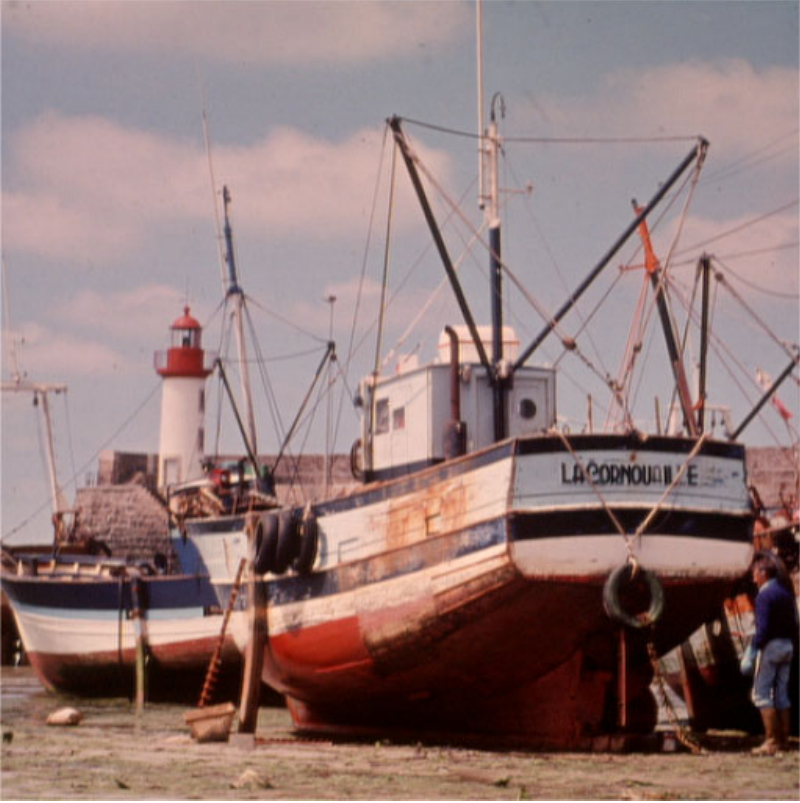}
\includegraphics[width=0.3\textwidth]{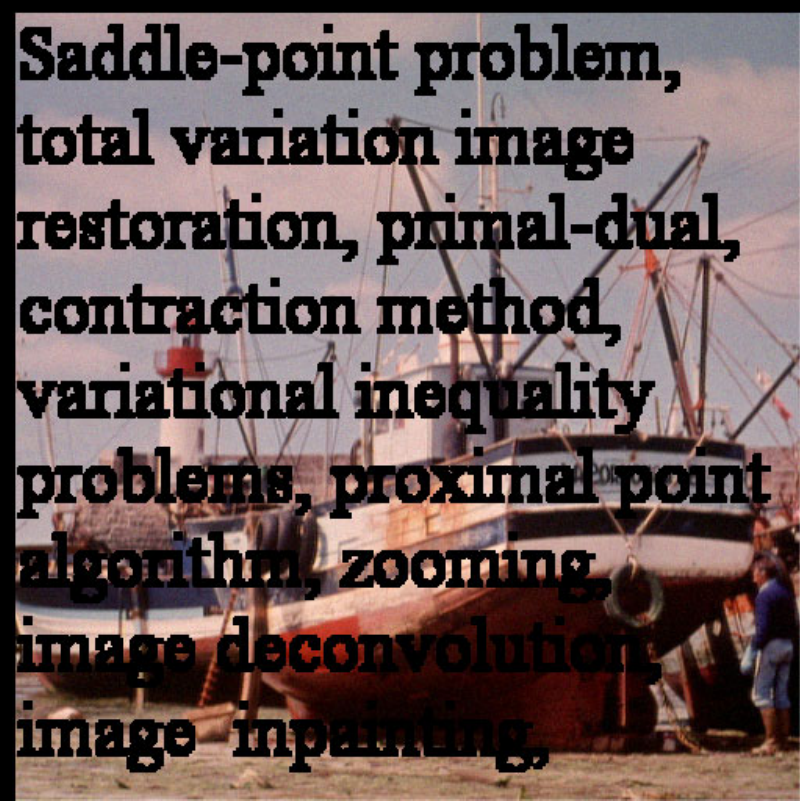}
\includegraphics[width=0.3\textwidth]{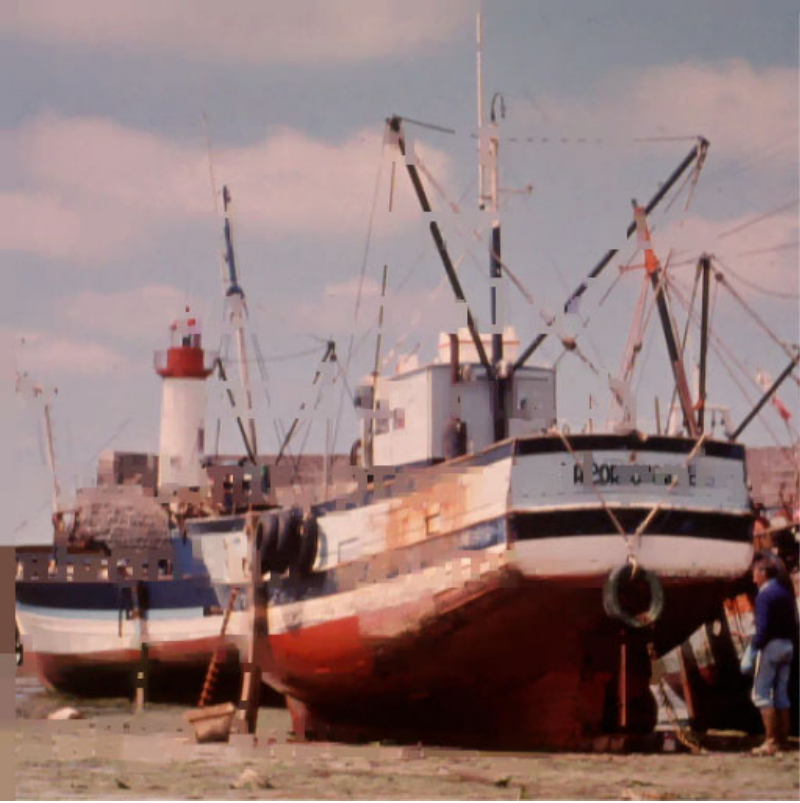}
\includegraphics[width=0.3\textwidth]{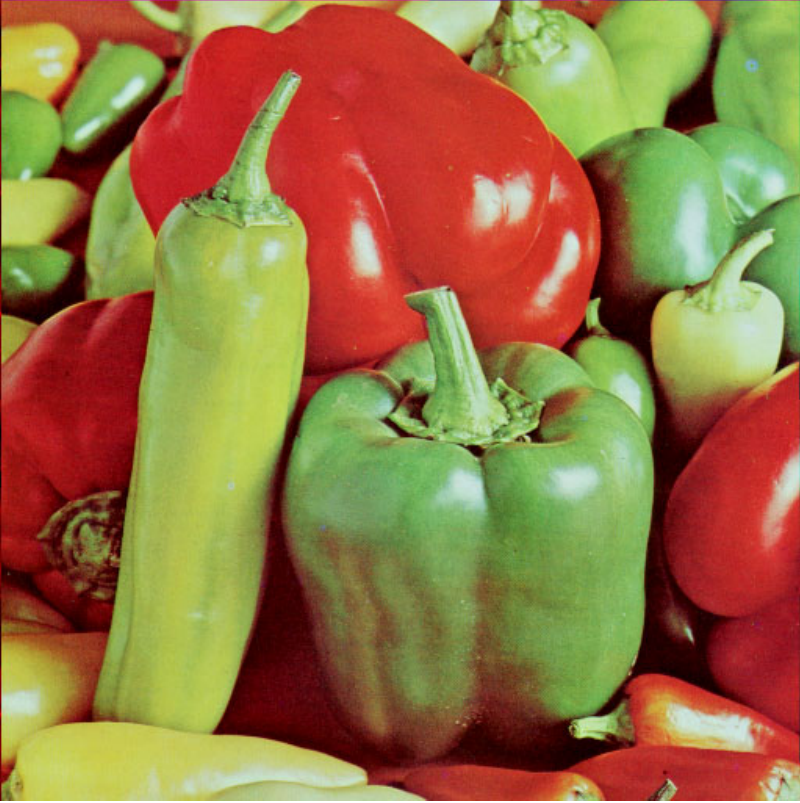}
\includegraphics[width=0.3\textwidth]{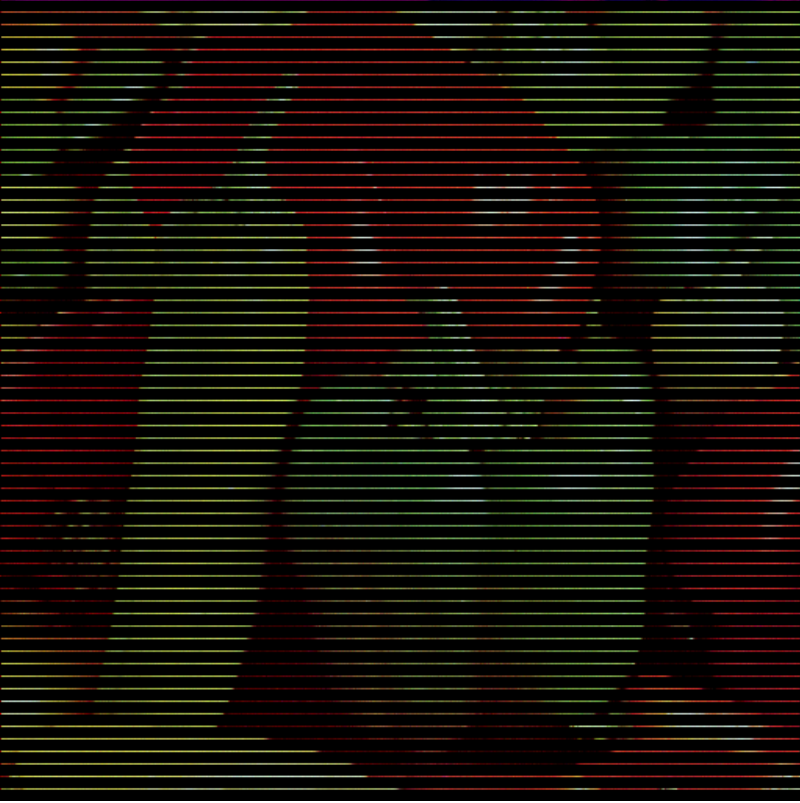}
\includegraphics[width=0.3\textwidth]{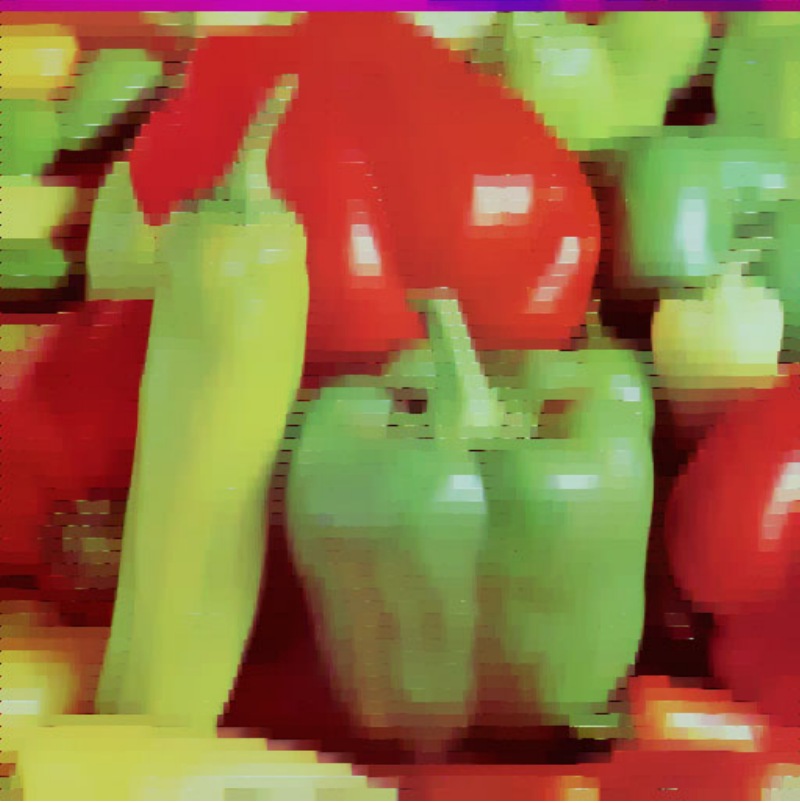}
\caption{The first column shows the original images, the second column presents the input images, and the third column displays the output images recovered by Algorithm \ref{alg-os}. }
\label{Fig:TV-inpainting}
\end{figure}

Table \ref{table1} summarizes the numerical results, including the number of iterations (Iter), CPU time (Time, measured in seconds), and SNRs of the images recovered by CP-PDHG, PDAc, and Algorithm \ref{alg-os}.
We present the relative function value residual $e_{\text{obj}}(x_n)$ within 2000 iterations in Figure \ref{Fig:TV-obj}.
Figure \ref{Fig:TV-inpainting} displays the original images, noisy input images, and inpainted output images obtained by Algorithm \ref{alg-os}.
Given that the inpainted output images obtained by the other two algorithms are visually similar, we only showcase the inpainted output images generated by Algorithm \ref{alg-os} for the sake of conciseness and clarity.

As clearly shown in Figure \ref{Fig:TV-obj}, Algorithm \ref{alg-os} ($\tau=1/L, \sigma=1.5/L$) demonstrates slightly superior convergence performance regarding the function value residual when compared to PDAc and CP-PDHG. With the same termination criterion, Table \ref{table2} reveals that, in contrast to CP-PDHG and PDAc, Algorithm \ref{alg-os} demands the fewest iterations and the shortest running time. Moreover, Algorithm \ref{alg-os} attains comparable or slightly better quality in the inpainted images.

\subsection{Bilinear saddle point problem}
In this section, we test the bilinear saddle-point problem \eqref{saddle-point} under two types of settings.
The first setting is a minimax matrix-game problem where $g(x)=\iota_{\D_q}(x)$, $f^*(y)=\iota_{\D_p}(y)$, $\D_q = \{x\in\mathbb{R}^q: x\geq 0, \; \sum_{i = 1}^q x_i = 1\}$, and $K\in\mathbb{R}^{p\times q}$ is randomly generated.
We tested three types of random matrices $K$, namely:
\begin{itemize}
    \item (i) $(q,p)=(100,100)$ and all entries of $K$ were independently generated from the uniform distribution on $[- 1,1]$.
    \item (ii) $(q,p)=(100,100)$ and all entries of $K$ were independently generated from the normal distribution $\mathcal{N}(0,1)$.
    \item (iii) $(q,p)=(100,100)$ and all entries of $K$ were independently generated from the normal distribution $\mathcal{N}(-1,1)$.
\end{itemize}
To compare different algorithms, we utilize the primal-dual gap  function defined by
$G(x, y) := \max\nolimits_i(Kx)_i - \min\nolimits_j(K^\top y)_j$ for a feasible pair $(x, y) \in\D_q \times \D_p$.
Here, the subscript $i$ (or $j$) represents the $i$th (or $j$th) component of the corresponding vector.

First, in accordance with \eqref{def:parameter-set},
we vary $\theta$ and $\eta$ from $0.1$ to $1.9$ in steps of $0.1$. Additionally, we set $\gamma \|K\|^2 = 0.99(2 - \theta)(2 - \eta)$ and then test Algorithm \ref{alg-os} for solving case (i).
Since the sequence $\{y_n\}_{n\in\N}$ converges to $y^\star$ which belongs to $\D_p$, even though $y_n\in \D_p$ may not hold for Algorithm \ref{alg-os}, we computed $G(x_n,y_n)$ and terminated Algorithm \ref{alg-os} when either $G(x_n,y_n)<10^{-9}$ or the maximum number of iterations $n_{\max}=10^{6}$ was reached.
The results of the required number of iterations are presented in Figure \ref{Fig MMtest}, in which the top ten minimum
iteration numbers are marked out with yellow stars. Notably, our recommended $\eta=7/6$ and $\theta=1/5$ satisfy $(2-\theta)(2-\eta)=1.5$. Since all but one yellow star lie in the curve's lower-left region, where smaller $(\eta,\theta)$ imply larger $\gamma$, this strongly suggests that the algorithm prefers larger step sizes.
Consequently, for the remainder of this section, we select $\gamma \|K\|^2 = 1.5$ with $\theta = 0.99/5$, $\eta=7/6$ and $\tau=\sigma$.

\begin{figure}[htpb]
\centering
\includegraphics[width=0.5\textwidth]{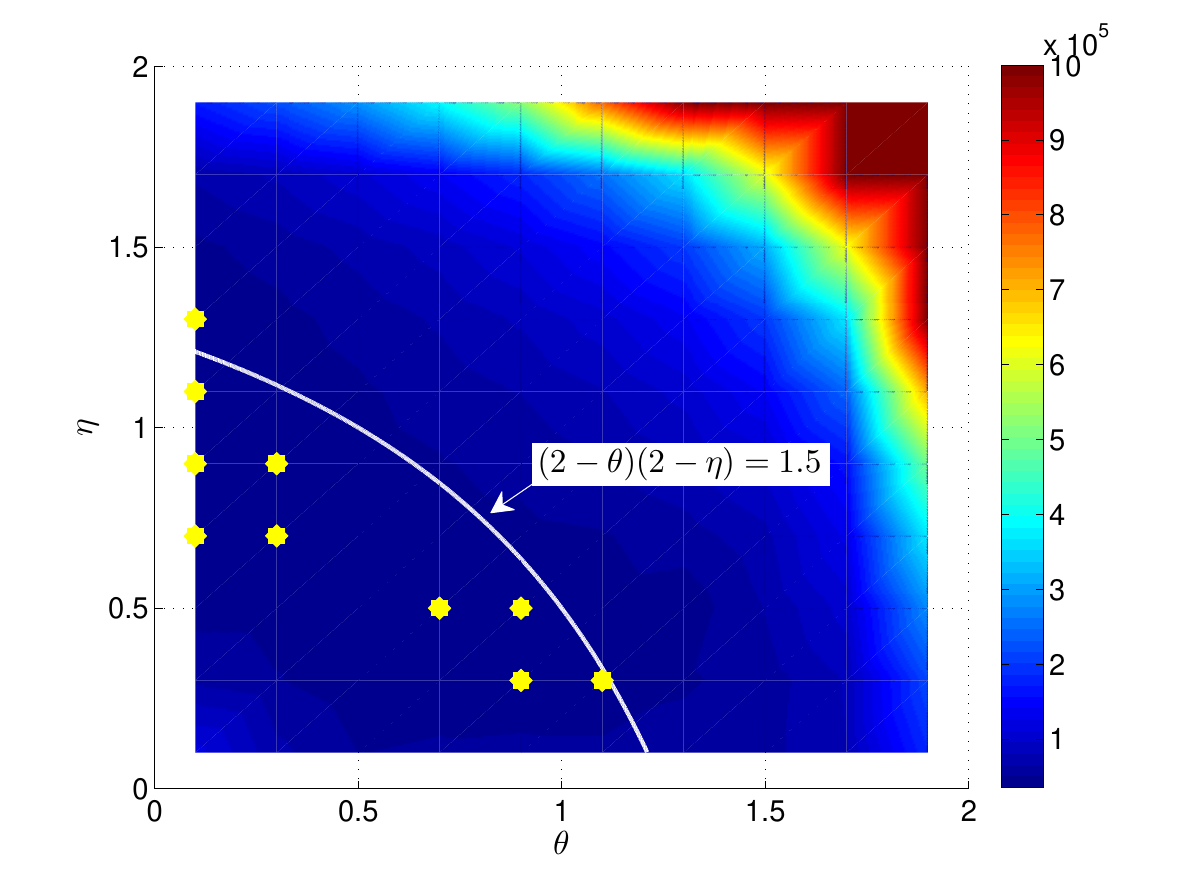}
\caption{Results of the required number of iterations of Algorithm \ref{alg-os} with different parameters $(\theta,\gamma,\eta)$ for solving the minimax matrix-game problem.  }
\label{Fig MMtest}
\end{figure}

Figure \ref{Fig MM} shows the decreasing trend of the primal-dual gap function values as CPU time elapses for CP-PDHG \cite{Chambolle2011A}, PDAc \cite{ChY2020Golden,ChY2022relaxed}, and Algorithm \ref{alg-os}. For CP-PDHG, we set $\tau=\sigma$ and $\tau\sigma = 1/\|K\|^2$. For PDAc, we set $\psi = 1.6$, $\tau=\sigma$ and $\tau\sigma = 1.6/\|K\|^2$.
As can be observed from Figure \ref{Fig MM}, Algorithm \ref{alg-os} outperforms CP-PDHG and PDAc in all tests. We ascribe the superior numerical performance of Algorithm \ref{alg-os} to the larger step sizes specified by $\gamma=\tau\sigma = 1.5/\|K\|^2$ and over-relaxation step $\eta=7/6$. However, for this minimax matrix-game problem, the relaxation of CP-PDHG fails to yield better results.

\begin{figure}[hbtp]
\centering
\subfigure[Case (i), $\|K\|\approx11.00$.]{
\includegraphics[width=0.32\textwidth]{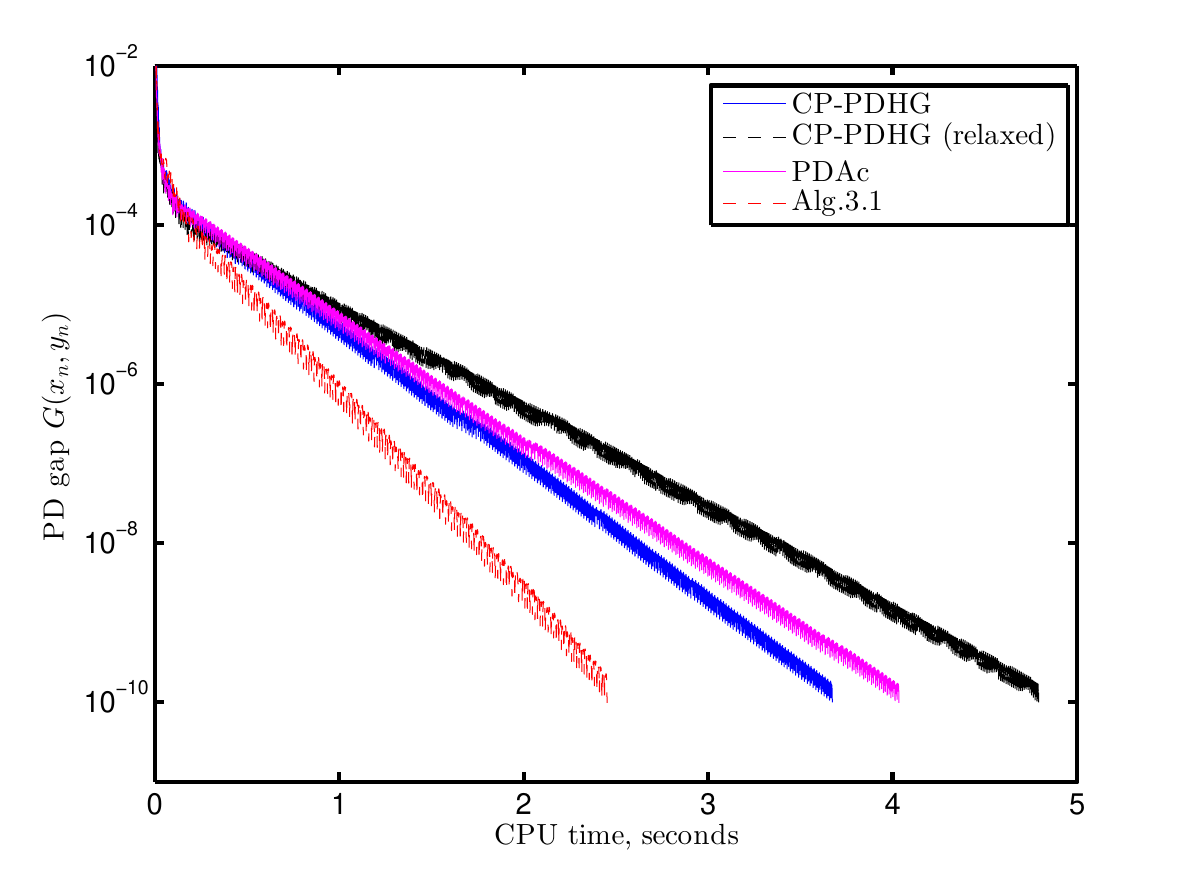}}
\subfigure[Case (ii), $\|K\|\approx20.18$.]{
\includegraphics[width=0.32\textwidth]{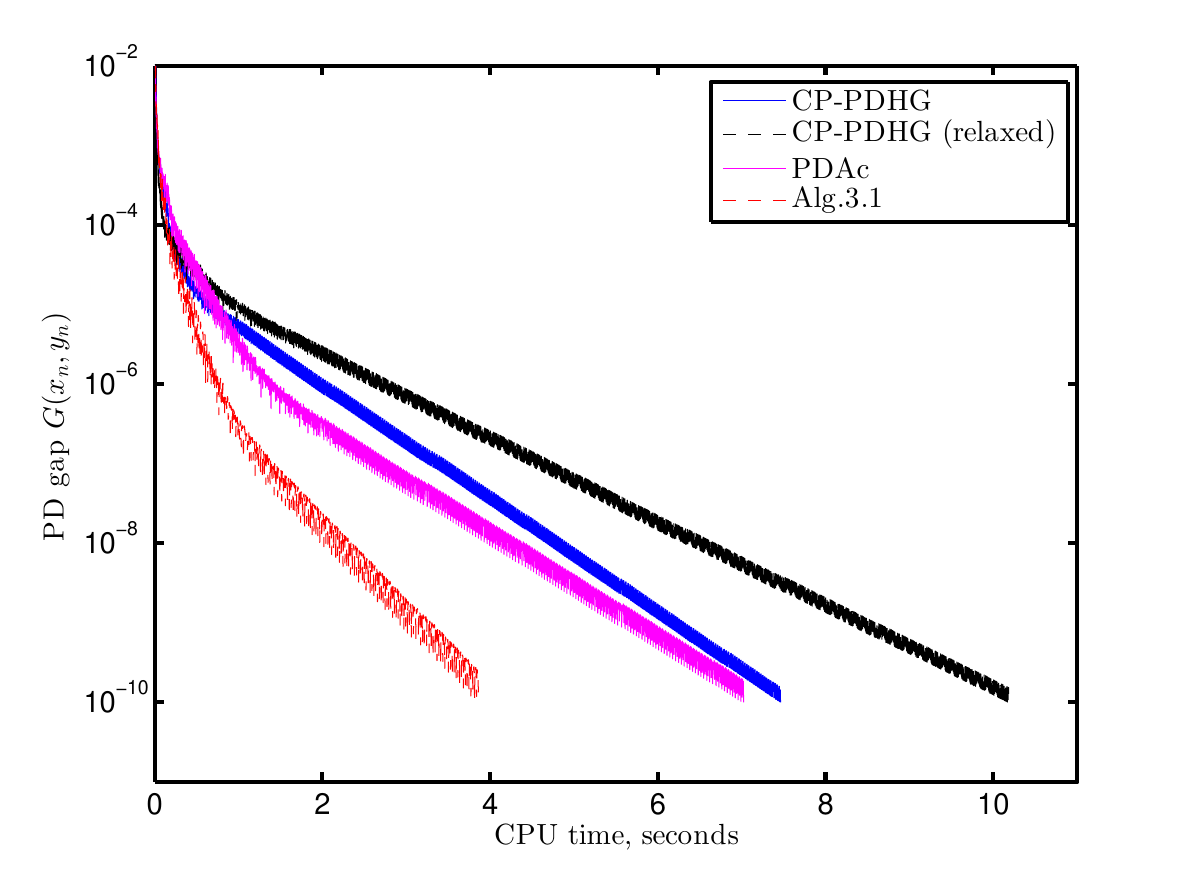}}
\subfigure[Case (iii), $\|K\|\approx97.59$.]{
\includegraphics[width=0.32\textwidth]{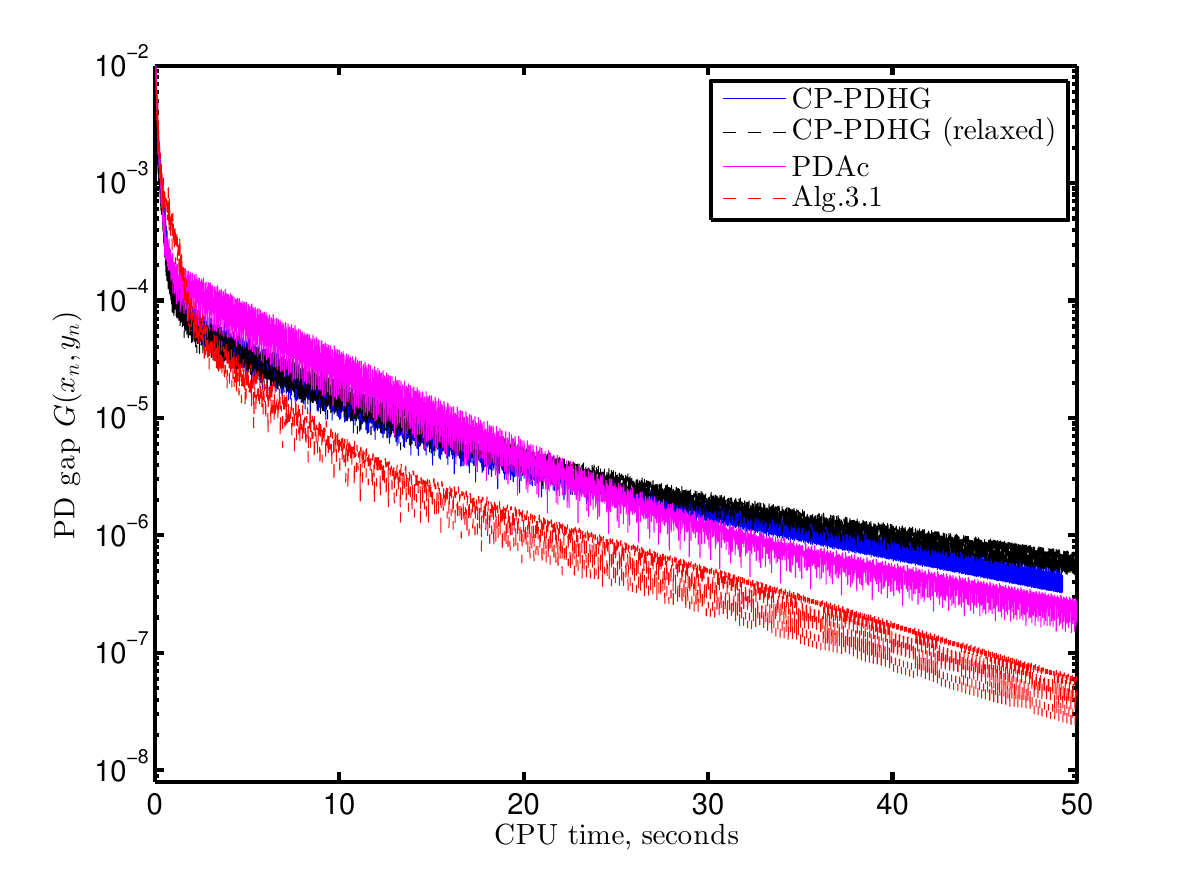}}
\caption{Numerical results for minimax matrix-game problem.
}
\label{Fig MM}
\end{figure}

The second setting of this subsection focuses on the LASSO problem. This problem can be regarded as \eqref{saddle-point} with $g(x) = \mu \|x\|_1$ and $f(y) =\frac{1}{2}\|y-b\|^2$. The corresponding primal problem \eqref{two-block} appears as $\min_x F(x):=\frac 1 2 \|Kx-b\|^2 + \mu  \|x\|_1$.
We randomly generated $x^*\in \mathbb{R}^q$. Specifically, we randomly and uniformly determined $s$ nonzero components of $x^*$, and their values were sampled from the uniform distribution over the interval $[-10,10]$.
The matrix $K\in \mathbb{R}^{p\times q}$ is constructed as per \cite{Malitsky2018A}. First, we generate a matrix $K^0\in\mathbb{R}^{p\times q}$ with entries independently sampled from $\mathcal{N}(0,1)$. Then, for a scalar $v\in(0,1)$, we construct the matrix $K$ column by column in the following manner: $K_1 = K_1^0/\sqrt{1 - v^2}$ and $K_j = vK_{j - 1}+K_j^0$ for $j = 2,\ldots,q$.
Here, $K_j$ and $K_j^0$ represent the $j$th column of $K$ and $K^0$, respectively.
As $v$ increases, $K$ becomes more ill-conditioned.
In this experiment, we set $\mu = 0.1\|A^\top b\|_{\infty}$ and tested $v = 0.5$ and $v = 0.9$. Finally, we set $b = Kx^*+\text{noise}$, where the additive $\text{noise}$ was generated from $\mathcal{N}(0,0.1)$.

\begin{figure}[htbp]
\centering
\subfigure[$v=0.5$, $\|K\|\approx103.15$]{
\includegraphics[width=0.43\textwidth]{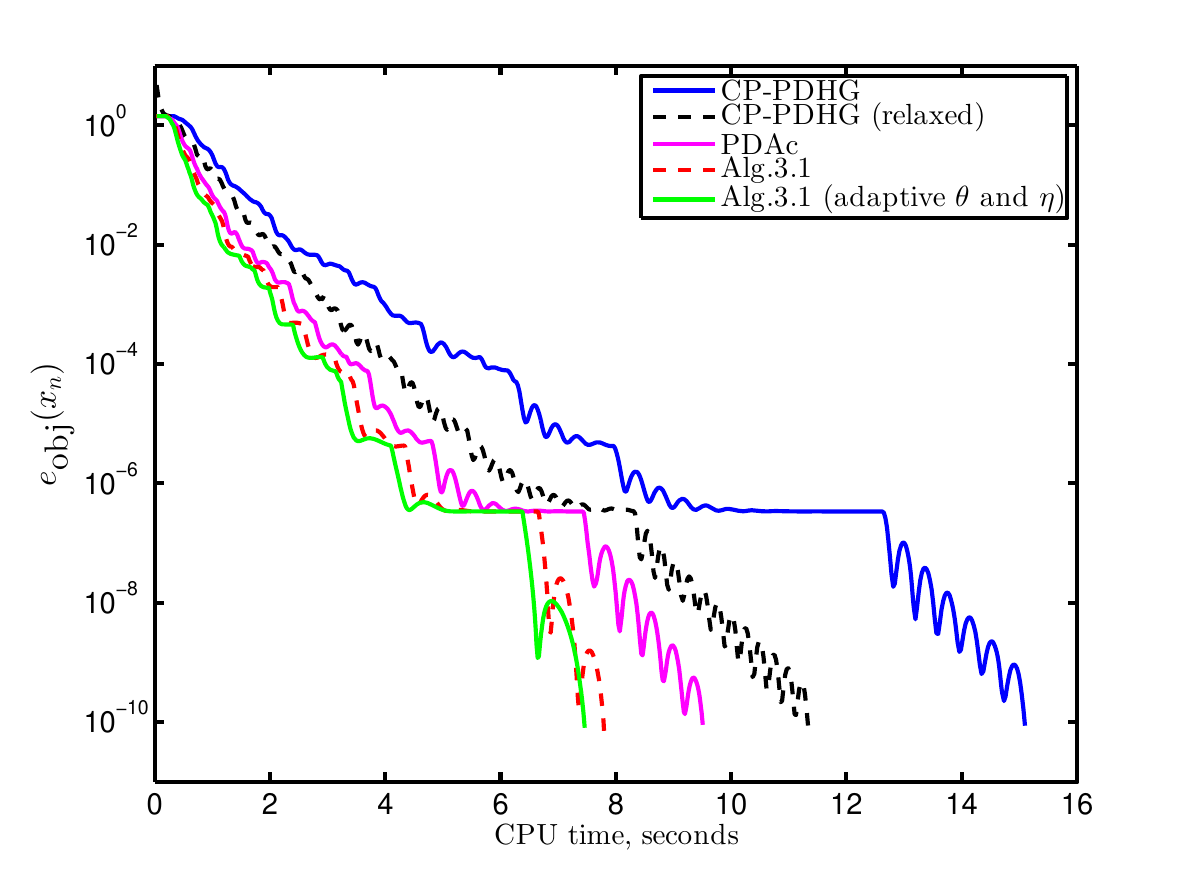}}
\subfigure[$v=0.9$, $\|K\|\approx378.56$]{
\includegraphics[width=0.43\textwidth]{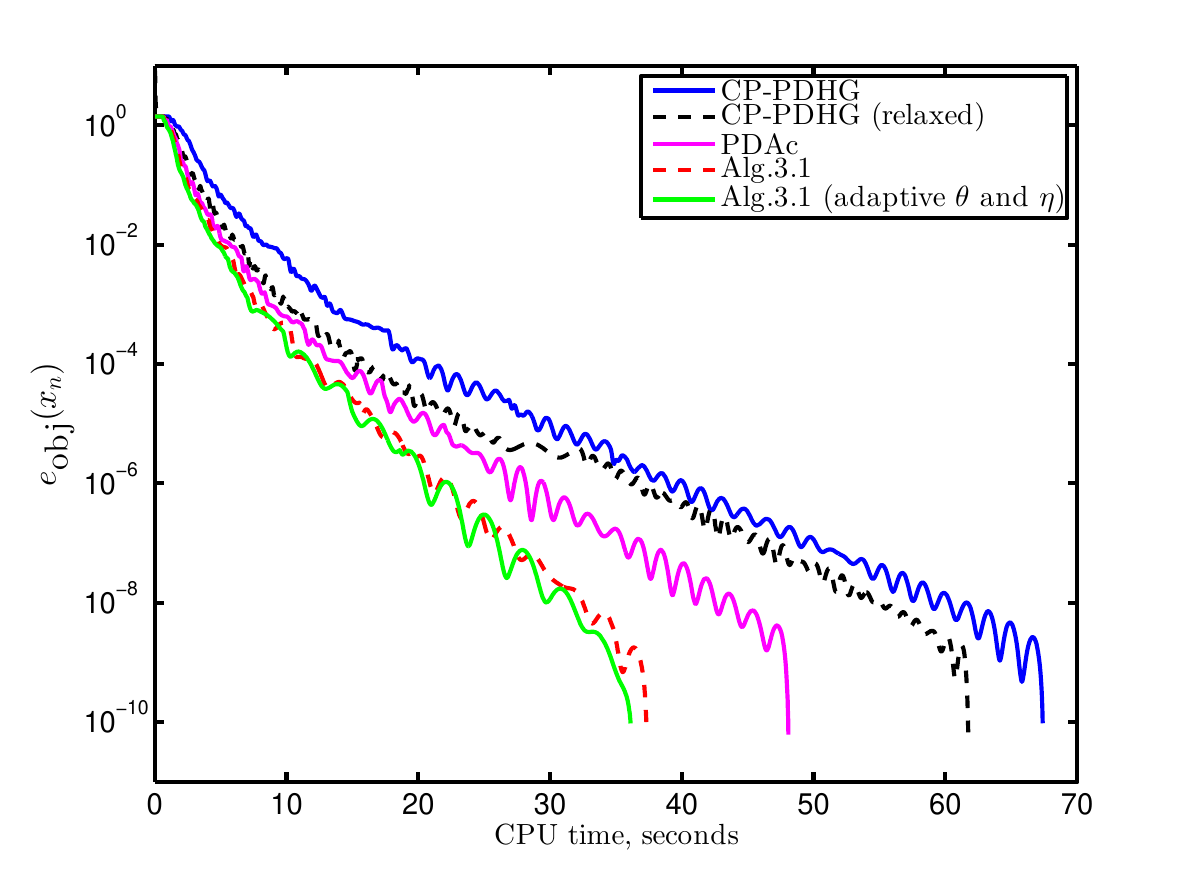}}
\caption{Comparison  results in terms of objective violation on the random LASSO with $(p,q,s)=(1000,2000,100)$.
}
\label{Fig:LASSO}
\end{figure}

Similar to the minimax matrix-game problem, for CP-PDHG, we maintain the setting of $\tau=\sigma$ and $\tau\sigma = 1/\|K\|^2$, and for PDAc, we set $\psi = 1.6$, $\tau=\sigma$ and $\tau\sigma = 1.6/\|K\|^2$.
We compare the algorithms based on the function-value residual, which is defined as $e_{\text{obj}}(x):=  |F(x)-F_{\text{opt}}| / |F_{\text{opt}}|$. Here, $F_{\text{opt}}$ represents the minimum function value and is computed using MOSEK via CVX\footnote{Available at  \url{http://cvxr.com/cvx/}}. The algorithms are terminated when $e_{\text{obj}}(x)<10^{-10}$.

Figure \ref{Fig:LASSO} depicts the evolution of the function value residual $e_{\text{obj}}(x_n)$ as CPU time progresses. These results illustrate the efficiency of Algorithm \ref{alg-os}, and also reveal that both Algorithm \ref{alg-os} and PDAc slightly outperform CP-PDHG in terms of CPU time.
This is probably because Algorithm \ref{alg-os} and PDAc can adopt larger step sizes. Additionally, the adaptive strategy presented in Section \ref{sec:discussion} for dynamically adjusting $(\theta,\eta)$ functions effectively, as evidenced by the results in Figure \ref{Fig:LASSO}.

Moreover, we tested Algorithm \ref{alg-os-imp} for solving the LASSO problem. For 10 randomly generated problems, we plotted in Figure \ref{Fig:LASSO-P} the evolution of $e_{\text{obj}}(x_n)$ with respect to the number of iterations.
For Algorithm \ref{alg-os-imp}, we set $\theta=\eta = 1.95$ and $\gamma\|K\|^2 = 0.6\theta\eta$. From the comparison with Algorithm \ref{alg-os}, we observe that the nondiagonal $P$ given in \eqref{def:P} leads to a lower number of iterations but requires more CPU time. This is because Algorithm \ref{alg-os-imp} requires one additional implementation of $K^*$ per iteration.

\begin{figure}[htbp]
\centering
\subfigure[]{
\includegraphics[width=0.43\textwidth]{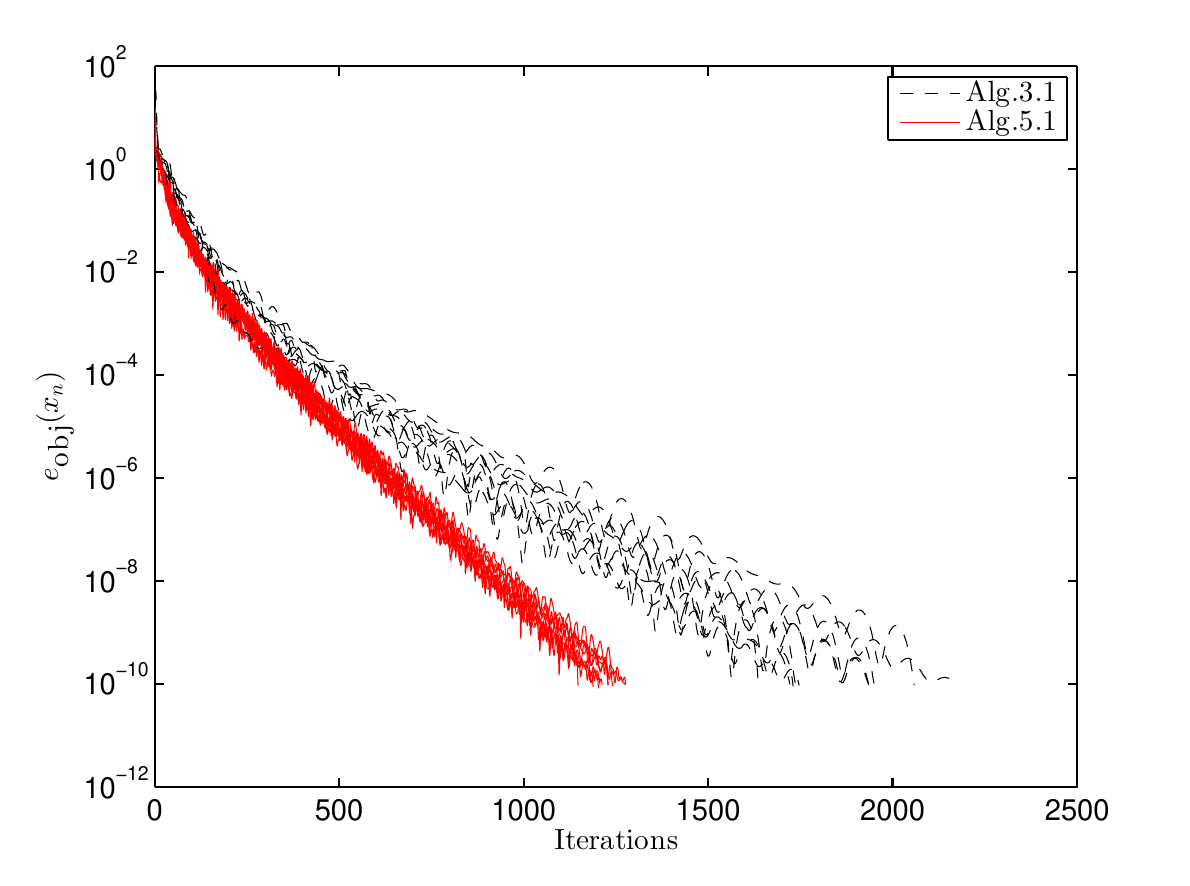}}
\subfigure[]{
\includegraphics[width=0.43\textwidth]{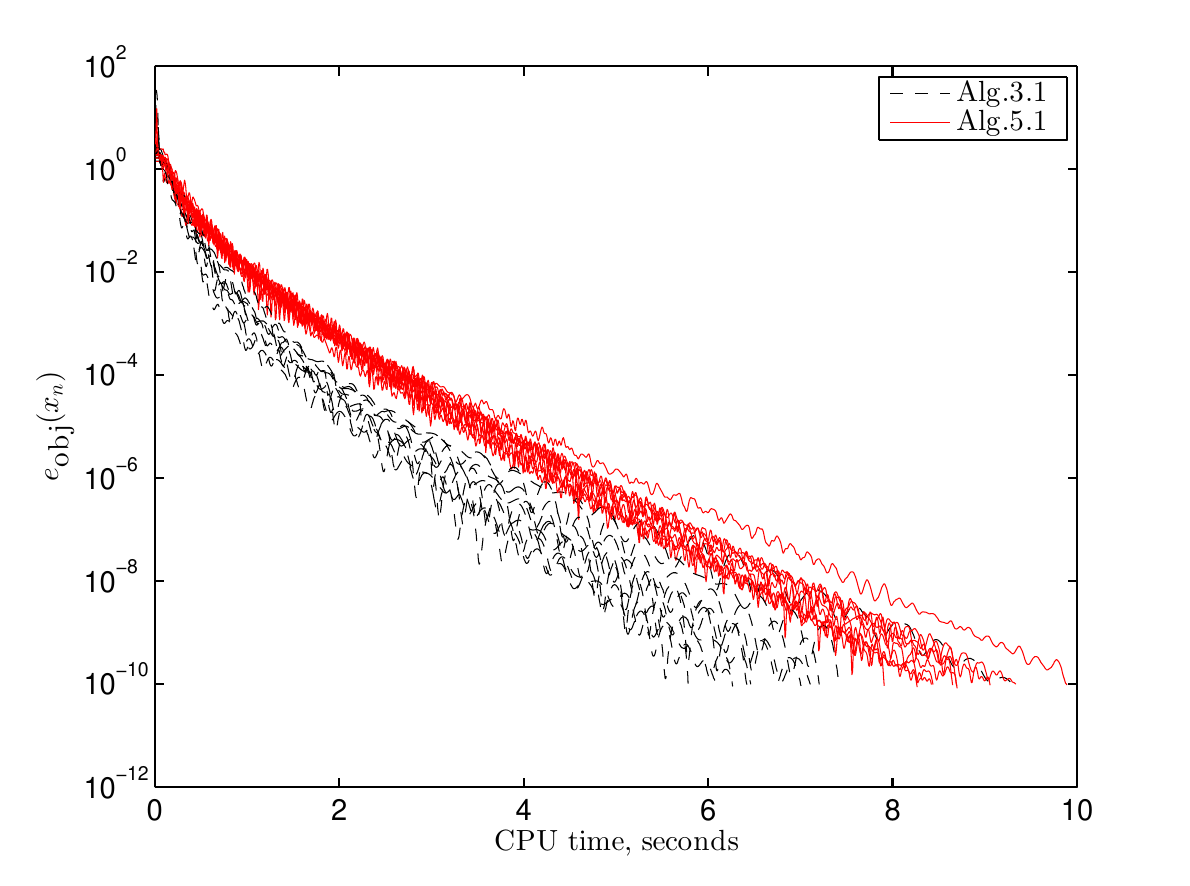}}
\caption{Comparison  results of Algorithms \ref{alg-os} and \ref{alg-os-imp} in terms of the function value residual  for 10 random LASSO  problems with $(p,q,s)=(300,1000,30)$ and $v=0.9$. Left:  $e_{\text{obj}}(x_n)$ against the iteration number. Right:  $e_{\text{obj}}(x_n)$ against the CPU time.
}
\label{Fig:LASSO-P}
\end{figure}

\section{Conclusions}\label{sec:conclusion}
In this paper, we introduced a basic primal-dual splitting algorithm (PDSA) scheme for composite monotone inclusion problems. This framework encompasses three adjustable elements, namely $v_{n + 1}$, $z_{n + 1}$, and $\omega_{n + 1}$ in equation \eqref{bs}. These elements can be derived through simple vector addition and scalar multiplications of the previous iterates.
By making a specific choice of these adjustable elements, we proposed a new PDSA allowing more flexible step sizes. In this new PDSA, the range of $\tau\sigma\|K\|^2$ can be as wide as $(0,4)$ when certain parameter values are appropriately selected.
Furthermore, we established the convergence of the proposed algorithm from the perspective of fixed-point iteration, despite the fact that the scheme depends on a convex combination step as given in equation \eqref{os-z}.

In conclusion, we present two potential directions for further research. First, the PDAc proposed in \cite{ChY2022relaxed} can be readily rewritten as a fixed-point iteration using the operator $T_P$ defined in \eqref{def:T}, with the linear operator $P$ being asymmetric. Thus, it is of particular interest to address the challenge posed by the asymmetry of $P$ and establish the convergence of PDAc from the fixed-point iteration viewpoint.
Second, in \cite{improvement15Bot}, under the strong  monotonicity assumptions for some involved operators, the PDSA presented in \eqref{PDSA} was accelerated by adaptively choosing $\tau$ and $\sigma$. However, for Algorithm \ref{alg-os}, the key question persists: how can we accelerate it through the adaptive adjustment of $\tau$ and $\sigma$? These issues remain interesting for future investigations.

\section*{Acknowledgment} We thank the anonymous referees for their thoughtful, insightful and constructive comments, which have greatly improved the paper.

\bibliographystyle{abbrv}
\bibliography{cxktex}

\begin{appendix}
\section{Proof of Lemma \ref{lem-opt}}\label{proof:lem-opt}
\begin{proof}
Recall that, for any $({x^{\star}}, {w^{\star}}, {y^{\star}})\in \widetilde{\Omega}$ defined in \eqref{def:Omega}, we have $Kx^\star=w^\star$. Since $x_n=\prox_{\tau g}(v_n - K^*(\gamma K v_n-u_n))$, from \eqref{property-p} we have
\begin{align}\label{eq:opt-g}
\tau (g(x_n)-g(x^\star))\leq  \langle x_n-v_n + K^*(\gamma K v_n-u_n), x^\star-x_n\rangle.
\end{align}
Similarly, from $w_n = \prox_{f/\sigma}((Kv_n+Kx_n)-u_n/\gamma)$, $\gamma=\tau\sigma$ and \eqref{property-p}, we have
\begin{align}\label{eq:opt-f}
\tau(f(w_n)-f(w^\star)) \leq  \langle \gamma w_n+u_n-\gamma(Kv_n+Kx_n), w^\star-w_n\rangle.
\end{align}
By summing  \eqref{eq:opt-g} and \eqref{eq:opt-f}, adding $\tau\langle y,Kx_n-w_n\rangle$ to both sides, and applying basic algebraic operations, we can easily derive
\begin{align}\label{eq:key-opt}
\tau J(x_n,w_n,y)\leq \, & \langle x_n-v_n, x^\star-x_n\rangle+\langle \gamma K v_n-u_n, Kx^\star-Kx_n\rangle \nonumber\\
& \, +\langle \gamma w_n+u_n-\gamma (Kv_n+Kx_n), w^\star-w_n\rangle+\tau\langle y,Kx_n-w_n\rangle \nonumber\\
= \, &\langle x_n-v_n,x^\star-x_n \rangle+\langle w_n -Kx_n, u^\star(y)-u_n \rangle-  \gamma \langle K(v^\star- v_n), K(x^\star- x_n)\rangle \nonumber\\
& \, +\gamma\langle w^\star-w_n,   2 Kx^\star -(Kv_n+Kx_n)-(w^\star-w_n)\rangle,
\end{align}
where $u^\star(y)=\gamma Kv^\star-\tau y$, $x^\star=v^\star$ and $K x^\star-w^\star=0$ are utilized.
Next, we treat each term on the right-hand-side  of \eqref{eq:key-opt} separately.
By adding $\langle x_n-v_n,v_n-v^\star \rangle$ to the first term in \eqref{eq:key-opt} and using $x^\star=v^\star$, we obtain
\begin{align}\label{eq:1-opt}
\langle x_n-v_n,x^\star-x_n \rangle +\langle x_n-v_n,v_n-v^\star \rangle = -\|v_n-x_n\|^2.
\end{align}
By subtracting $\langle x_n-v_n,v_n-v^\star \rangle$ from the second term in \eqref{eq:key-opt},
recalling $\bz^\star(y) :=(v^\star, u^\star(y))^\top $, $\bz = (v,u)^\top$, $M = P^{-1}$, noting \eqref{def:T3} and using \eqref{id}, we derive
\begin{align}\label{eq:2-opt}
\langle w_n -Kx_n,  u^\star(y)-u_n\rangle &- \langle x_n-v_n, v_n-v^\star\rangle
 =\langle \bz_{n+1}-\bz_n, \bz^\star(y)-\bz_n\rangle_M \nonumber\\
 =&\frac{1}{2}\left(\|\bz_{n+1}-\bz_n\|_M^2+\|\bz_n-\bz^\star(y)\|_M^2 -\|\bz_{n+1}-\bz^\star(y)\|_M^2 \right).
\end{align}
By using \eqref{id} and $x^\star=v^\star$, the third term   in \eqref{eq:key-opt} can be represented  as
\begin{align}\label{eq:3-opt}
- \gamma \langle K(v^\star- v_n), K(x^\star- x_n)\rangle = \frac{\gamma}{2}\big(\|K(x_n-v_n)\|^2 - \|K(v^\star-v_n)\|^2- \|K(x^\star-x_n)\|^2\big).
\end{align}
Again, noting $Kx^\star=w^\star$ and $x^\star = v^\star$, by \eqref{id}, the last term in \eqref{eq:key-opt} becomes
\begin{align}\label{eq:4-opt}
 \gamma  \langle (Kx^\star-w_n), \; &  K(x^\star-x_n) - (Kv_n-w_n)\rangle \nonumber \\
= \; &  \frac{\gamma}{2} \Big( \|K(v^\star-v_n)\|^2+ \|K(x^\star-x_n)\|^2  - \|Kx_n-w_n\|^2 - \| Kv_n-w_n\|^2 \Big).
\end{align}
Substituting \eqref{eq:1-opt}-\eqref{eq:4-opt}  into \eqref{eq:key-opt}, then multiplying both sides by 2 and  rearranging the terms, we obtain
\be\label{eq:final1}
\|\bz_{n+1}-\bz^\star(y)\|_M^2+2 \tau J(x_n,w_n,y)&\leq& \|\bz_{n}-\bz^\star(y)\|_M^2+\|\bz_{n+1}-\bz_n\|_M^2 - \gamma \| Kv_n-w_n\|^2 \nonumber\\
&&- \left( 2\|v_n-x_n\|^2 - \gamma  \|K(x_n-v_n)\|^2 + \gamma \|Kx_n-w_n\|^2 \right).
\ee
From Lemma \ref{lem-a} and \eqref{def:T3}, we have
$\bz_{n+1}-\bz_n = P \begin{pmatrix} \ba{c}x_n-v_n\\w_n-K x_n\ea\end{pmatrix}$.
Since $M=P^{-1}$, we can rewrite
\[\label{eq:final3}
\left(2\|v_n-x_n\|^2 - \gamma  \|K(x_n-v_n)\|^2 + \gamma \|Kx_n-w_n\|^2\right) = \|\bz_{n+1}-\bz_n\|_{\widehat{G}}^2,
\] with
$\widehat{G} :=M^*\diag(2I_{\cH}-\gamma K^*K, \gamma I_{\cG}) M$.
Similarly, we can rewrite
\be
\|Kv_n -w_n\|^2 = \|Kv_n-Kx_n\|^2 +2 \langle Kv_n-Kx_n, Kx_n-w_n \rangle  + \|Kx_n-w_n\|^2 =\|\bz_{n+1}-\bz_n\|_{\widehat{M}}^2, \label{jy-11}
\ee
where $\widehat{M} := M^* E_K M$ with $E_K :=\left[\ba{cc}K^*K~&~ K^*\\  K~&~I_{\cG}\ea\right]$.
Substituting \eqref{eq:final3} and \eqref{jy-11} into \eqref{eq:final1}, we obtain
\be\label{eq:final2}
\|\bz_{n+1}-\bz^\star(y)\|_M^2+2 \tau J(x_n,w_n,y)
\leq \|\bz_{n}-\bz^\star(y)\|_M^2-\|\bz_{n+1}-\bz_n\|_{\widehat{G}-M + \gamma\widehat{M}}^2.
\ee
Finally, from the definitions of  $\widehat{G}$,  $\widehat{M}$, $\Phi_K$ in \eqref{def:Phi} and $Q$ in \eqref{def:Q}, it is elementary to verify that
$\widehat{G}-M+\gamma\widehat{M}=Q$. Thus, \eqref{eq:final2} is identical to the desired result  \eqref{sec43lem}. This completes the proof.
\end{proof}

\end{appendix}

\end{document}